\newif\ifpaper
\paperfalse % uncomment for arxiv and comment above

\ifpaper
\documentclass[twoside]{article}
\usepackage{aistats2020}
\else
\documentclass[11pt]{article}
\fi

\usepackage[utf8]{inputenc} % allow utf-8 input
\usepackage[T1]{fontenc}    % use 8-bit T1 fonts
\usepackage{hyperref}       % hyperlinks
\usepackage{url}            % simple URL typesetting
\usepackage{booktabs}       % professional-quality tables
\usepackage{amsfonts}       % blackboard math symbols
\usepackage{nicefrac}       % compact symbols for 1/2, etc.

\usepackage{amsmath,amssymb,amsthm}
\usepackage{graphicx,enumerate}
\usepackage{pgf,tikz,pgfplots}
\usetikzlibrary{arrows}
\usepackage{subcaption}
\numberwithin{equation}{section}
\ifpaper
\else
\usepackage[top=0.5in, bottom=0.8in, left=0.7in, right=0.7in, marginparwidth=1.5in]{geometry}
\fi

\newcounter{THMCTR}
\newtheorem{theorem}[THMCTR]{Theorem}

\newtheorem{proposition}[THMCTR]{Proposition}
\newtheorem{lemma}[THMCTR]{Lemma}
\theoremstyle{definition}
\newtheorem{remark}[THMCTR]{Remark}
\newtheorem{definition}[THMCTR]{Definition}
\newtheorem{example}[THMCTR]{Example}

\newtheorem{assumption}{Assumption}
\newtheorem{algorithm}{Algorithm}

\newcommand{\nn}{\mathbb{N}} %Nonnegative integers
\newcommand{\norm}[1]{\left\Vert {#1} \right\Vert} %Norm
\newcommand{\erl}{\left(-\infty , +\infty\right]} %Extended real line
\newcommand{\dom}[1]{\mathrm{dom}\,{#1}} %Domain
\newcommand{\idom}[1]{\mathrm{int\,dom}\,{#1}} %Interior of the Domain
\newcommand{\act}[1]{\left\langle {#1} \right\rangle} %The value of 1

\newcommand{\argmin}{\mathrm{argmin}}

\newcommand{\sgn}{\mathrm{sgn}}
\newcommand{\barc}{\overline{C}}

\newcommand{\PPP}{\mathcal{P}}
\newcommand{\SSS}{\mathcal{S}}
\newcommand{\R}{\mathbb{R}}

\newcommand{\real}{\mathbb{R}} %Real numbers
\newcommand{\uL}{\bar{L}}
\newcommand{\lL}{\underline{L}}
\graphicspath{{./figures/}}

\usepackage[colorinlistoftodos,prependcaption,textsize=tiny]{todonotes}

\usepackage{xspace}

\ifpaper
\else
\title{\vspace{-1.5em}{\bf Bregman Proximal Framework for\\ Deep Linear Neural Networks}}

\author{Mahesh Chandra Mukkamala \thanks{Department of Mathematics, Saarland University, Germany, E-mail: \texttt{mukkamala@math.uni-sb.de}} \quad\quad Felix Westerkamp \thanks{Department of Informatics, Technical University of Munich, Germany, Email: \texttt{felix.westerkamp@tum.de}}\\\\  Emanuel Laude \thanks{Department of Informatics, Technical University of Munich, Germany, Email: \texttt{emanuel.laude@tum.de}}  
\quad\quad Daniel Cremers \thanks{Department of Informatics, Technical University of Munich, Germany, Email: \texttt{cremers@tum.de}} \quad\quad  Peter Ochs \thanks{Department of Mathematics, Saarland University, Germany, E-mail: \texttt{ochs@math.uni-sb.de}} 
}
\date{}
\fi
\begin{document}
\ifpaper
\twocolumn[

\aistatstitle{ Bregman Proximal Framework for Deep Linear Neural Networks}
% \aistatsauthor{ Mahesh Chandra Mukkamala \And  Felix Westerkamp \And  Emanuel Laude}

% \aistatsaddress{ Saarland University, Germany  \And  TU Munich, Germany \And TU Munich, Germany   } 

% \aistatsauthor{ Daniel Cremers \And Peter Ochs}

% \aistatsaddress{  TU Munich, Germany \And Saarland University, Germany } 

]
\else
\maketitle
\vspace{-2em}
\fi
	\begin{abstract}
	A typical assumption for the analysis of first order optimization methods is the Lipschitz continuity of the gradient of the objective function. However, for many practical applications this assumption is violated, including loss functions in deep learning. To overcome this issue, certain extensions based on generalized proximity measures known as Bregman distances were introduced. This initiated the development of the Bregman proximal gradient (BPG) algorithm and an inertial variant (momentum based) CoCaIn BPG, which however rely on problem dependent Bregman distances. In this paper, we develop Bregman distances for using BPG methods to train Deep Linear Neural Networks. The main implications of our results are strong convergence guarantees for these algorithms. We also propose several strategies for their efficient implementation, for example, closed form updates and a closed form expression for the inertial parameter of CoCaIn BPG. Moreover, the BPG method requires neither diminishing step sizes nor line search, unlike its corresponding Euclidean version. We numerically illustrate the competitiveness of the proposed methods compared to existing state of the art schemes.
	\end{abstract}
\ifpaper
\else
	\smallskip

\fi

\section{Introduction} \label{Sec:Intro}

The analysis of many first-order optimization methods relies on the Lipschitz continuous gradient property for the involved objective. Such a property allows for uniform quadratic upper and lower bounds at each point. These bounds enable the usage of a \emph{constant step size rule}, thus resulting in better performance compared to diminishing step sizes. However, remarkably, even the simplest (one hidden layer linear) neural network does not allow for uniform quadratic bounds. The same is true for many problems, e.g., phase retrieval and matrix factorization.  \ifpaper\else\medskip\fi

%%% Remedy by BPG
A remedy with a general class of upper and lower bounds, induced by so-called Bregman distances was proposed in \cite{BBT2016}. Such bounds can be exploited algorithmically via the Bregman Proximal Gradient (BPG) algorithm \cite{BSTV2018} and its inertial variant CoCaIn BPG \cite{MOPS2019} (based on Nesterov's momentum). In particular, BPG enables the usage of a constant step size, which is efficient to implement in practice (see Section~\ref{sec:discuss-settings-for-algs}), instead of diminishing step sizes or line search. However, in order to use BPG methods, an appropriate \emph{problem dependent Bregman distance} must be developed.\ifpaper\else\medskip\fi

\textbf{Key contribution.}  We consider deep linear neural networks with a squared loss, for which we propose a novel class of Bregman distances. This is key to illustrate the applicability and also to transfer the global convergence (to a stationary point) results of BPG and CoCaIn BPG algorithms. To enable an efficient implementation of the update step, we propose closed form analytic expressions for various practical settings. We also propose a novel variant of CoCaIn BPG, to further improve the efficiency for large scale problems. \ifpaper\else\medskip\fi

% We prove that the square loss of a deep linear neural network classifier is smooth adaptable with respect to a novel Bregman distance generating kernel, which establishes convergence of BPG. Moreover, we provide efficient formulas the update step. We believe that this contribution is a significant step towards understanding deep neural networks from an optimization perspective, by a new class of optimization algorithms.\ifpaper\else\medskip\fi

%\paragraph{Key contribution.} In this paper, we apply the recently developed theory of algorithms that rely on a more general class of upper and lower bounds, induced by so-called Bregman distances, to deep linear neural networks. In particular, we develop a new algorithm based on a novel Bregman distance, which non-trivially extends progress for the related setting of Matrix Factorization\footnote{Therefore, we also call our setting Deep Matrix Factorization.}. \PO{Maybe better write ``We believe that ...''.}Our contribution is a significant step towards efficient optimization of deep neural networks by a new class of optimization algorithms.

The developed Bregman distance yields a base algorithm (BPG) that allows for modifications in analogy to the development of alternating, stochastic or inertial variants of the base Proximal Gradient (PG) method. The provided BPG based algorithms are usually competitive and often superior to their Euclidean variants (PG) whenever both are applicable. We discuss several such situations in Section~\ref{sec:discuss-settings-for-algs}.

\subsection{Related Work}

% \textbf{Extensions of Lipschitz Continuity/Strong Smoothness.} The objectives for many practical problems including structured low-rank matrix factorization problems \cite{MO2019a}, quadratic inverse problems \cite{BSTV2018}, poisson inverse problems \cite{BBT2016}, cubically regularized problems \cite{MOPS2019} are not strongly smooth. This implies that numerous algorithms based on proximal gradient method are not applicable, unless line search is used. However,  line search can involve multiple objective evaluations in a single iteration, which can be very expensive. Recently, in \cite{BBT2016}, the standard notion of Lipschitz continuity of gradients was extended with the help of Bregman distances. The notion of the $L$-smad property was introduced in \cite{BSTV2018} for general non-convex problems, which we focus here.\ifpaper\else\medskip\fi

% \textbf{Bregman Proximal Minimization.} The $L$-smad property played a crucial role in the development of Bregman proximal algorithms, in particular for non-convex non-smooth problems. Bregman proximal algorithms generalize the standard proximal algorithms by using general Bregman distances. The popular algorithms developed within this framework include the BPG  and the corresponding inertial variant CoCaIn BPG, which relies on Nesterov's momentum like strategy. The inertial variants were also explored in \cite{ZBMJC2019,HGP2019}. In convex optimization, Bregman proximal algorithms were explored in \cite{LFN2018,hanzely2018fastest}.\ifpaper\else\medskip\fi

% <<<<<<< HEAD
\textbf{Extensions of Lipschitz Continuity of the Gradient.} For many practically relevant problems including  poisson inverse problems \cite{BBT2016}, structured low-rank matrix factorization problems \cite{MO2019a}, quadratic inverse problems \cite{BSTV2018} or cubically regularized problems \cite{MOPS2019}, the corresponding objective functions are not $L$-smooth. This hinders us from a straight application of proximal gradient related schemes, unless a line search is incorporated. However, line search typically involves multiple objective evaluations in a single iteration, which may be prohibitive in large scale setting. To overcome this limitation, in \cite{BBT2016,BSTV2018} the notion of a $L$-smooth adaptable ($L$-smad) function  is introduced which extends the classical $L$-smoothness property by means of a problem-dependent Bregman distance. This includes a much larger class of functions, in particular those that grow with a higher-order than quadratic. However, the choice of the problem-dependent Bregman distance is typically non-trivial.\ifpaper\else\medskip\fi

\textbf{Bregman Proximal Minimization.} The $L$-smad property can be characterized in terms of a generalized non-Euclidean Descent Lemma \cite{BSTV2018}. In analogy to the Euclidean case this yields a non-quadratic global upper-bound whose minimization corresponds to a generalized proximal gradient iteration called Bregman proximal gradient (BPG), see \cite{BBT2016,BSTV2018}. In \cite{MOPS2019} an inertial variant of BPG, called CoCaIn BPG has been introduced, which relies on a Nesterov's momentum like update strategy. Inertial variants were also explored in \cite{ZBMJC2019,HGP2019}. The mirror descent algorithm, (a special case of BPG when the second term in the problem is zero) has been extended to a stochastic setting under convexity in \cite{hanzely2018fastest}. Later in \cite{DDM2018} the BPG algorithm for non-convex composite problems has been generalized to a stochastic setting as well, where the smooth term is assumed to be smooth adaptable and the non-smooth term is convex.\ifpaper\else\medskip\fi

\textbf{Matrix Factorization.} Bregman distances for matrix factorization problems has become an active research area \cite{LZTW2019,AHGP2019}. In \cite{DBA2019} a low-rank semidefinite program is reformulated in terms of a symmetric matrix factorization problem which is solved with BPG. To this end the authors prove that the corresponding objective is $L$-smad relative to a quartic kernel. More recently, in \cite{MO2019a} this idea has been extended to a more general regularized matrix factorization problem, for which the authors design a novel Bregman distance to guarantee the $L$-smad property of the corresponding objective. However, such Bregman distances are not valid for deep linear neural network training (resp. deep matrix factorization) involving an arbitrary number of factors.  \ifpaper\else\medskip\fi

\textbf{Deep Linear Neural Networks.} The main contribution of this work is to derive Bregman distances suitable for training deep linear networks with a quadratic loss, which is an important and interesting optimization problem due to the following reasons: Fristly, as remarked by \cite{goodfellow2016deep} and in view of \cite{choromanska2015loss,K2016,YSJ2018,pmlr-v89-wu19b} it is well justified to first study the theoretically more tractable deep linear networks instead of the more challenging deep nonlinear networks. Secondly, even though deep linear networks essentially describe a linear model, mirror descent eventually inherits the implicit regularization bias observed for gradient descent optimization \cite{GWBNS2017,GBS2019,ACHL2019} which has turned out to be beneficial and important for practical applications, for e.g., \cite{bellkligler2019blind}.\ifpaper\else\medskip\fi

% =======
% \textbf{Deep Matrix Factorization.} Bregman distances for matrix factorization problems has become an active research area \cite{LZTW2019,AHGP2019}. For symmetric matrix factorization, Bregman distances were proposed in \cite{DBA2019}. Such distances are not valid for two factor matrices, thus new Bregman distances for standard matrix factorization problems were introduced in \cite{MO2019a} and $L$-smad property was shown. But, these Bregman distances are restricted to only two factor matrices and not suitable for deep linear neural networks or deep matrix factorization problems, where number of factors can be any finite number. We show $L$-smad property for deep linear neural networks, via new Bregman distances and ready applicability of BPG and CoCaIn BPG.
% >>>>>>> cf4a0ccfae63e49b638889757fbf87c6312ea1bc

\section{Bregman Proximal Minimization}\label{sec:matrix-factor-problem}
In this section, we revisit required concepts from related works \cite{BBT2016,BSTV2018}. Most importantly this includes the definition of a smooth adaptable function ($L$-smad), originally due to \cite{BBT2016}, which builds upon the notion of a Bregman distance. We motivate with a simple one-dimensional example for which classical $L$-smoothness fails. Finally we illustrate that the $L$-smad property gives rise to a global upper bound that can be exploited algorithmically to derive an iterative minimization scheme, called Bregman proximal gradient method \cite{BSTV2018}. This generalizes the classical proximal gradient descent scheme to non-Euclidean geometry.\ifpaper\else\medskip\fi

We use the notation of \cite{RW1998-B}. 

\subsection{Smooth Adaptable Functions}
	Let $g$ be a continuously differentiable function over $\R^d$. Then $g$ is said to be (classically) $L$-smooth (has Lipschitz continuous gradient), if there exists $L>0$, such that for all $x,y \in \R^d$, we have
	\[
	\norm{\nabla g(x) - \nabla g(y)}_2 \leq L \norm{x-y}_2\,.
	\]
	This implies that the functions $L\frac{\norm{\cdot}^2}{2} - g$ and $L\frac{\norm{\cdot}^2}{2} + g$ are convex on $\R^d$, which is equivalent to the statement of the well-known Descent Lemma (as shown in \cite[Lemma 1.2.3]{N1998}). \ifpaper\else\medskip\fi
	
	However, the $L$-smoothness assumption can be too restrictive, which we illustrate by the following example.
	\begin{example} \label{Ex:Lsmooth-fail}
	The simple two dimensional function $g(x,y) = (x^2+y^2)^2$ is not $L$-smooth in $\R^2$, as it lacks a global quadratic upper bound. As long as the initialisation is unknown, this means that proximal gradient algorithms (with constant step size) cannot be used for optimization. Notably, this issue persists even if we resort to alternating, Gauss--Seidel like algorithms such as PALM \cite{BST2014}, iPALM \cite{PS2016}, BCD \cite{XY2013}, which rely on the $L$-smoothness of the objective with respect to one (block) variable. In the above example, even if we fix $y=c$, for some constant $c \in \R$, the function $g_1(x)= (x^2+c^2)^2$  fails to be $L$-smooth. 
	\end{example}

	Likewise, quadratic inverse problems, matrix factorization problems and many other practical problems, lack $L$-smoothness. To overcome this limitation, recent works \cite{BBT2016,BSTV2018,LFN2018,LOC2019} consider an extension of $L$-smooth functions called $L$-smooth adaptable functions, which relies on the concept of a Bregman distance.  Such distances are constructed from a kernel generating distance, defined below. The rest of the section introduces the concepts of \cite{BSTV2018}, specialized to our unconstrained setting.
		\begin{definition}\label{D:KernelGen}
				Let $C \neq \emptyset$ be a convex and open subset of $\real^{d}$. Associated with $C$, a function $h : \real^{d} \rightarrow \erl$ is a \textit{kernel generating distance} if: 
				\begin{itemize}
    			\item[$\rm{(i)}$] $h$ is proper, lower semicontinuous and convex, with $\dom h \subset \barc$ and $\dom \partial h = C$ .
     			\item[$\rm{(ii)}$] $h$ is $C^{1}$ on $\idom h \equiv C$.
 			\end{itemize}
		 \end{definition}
	 	Denote the class of kernel generating distances  by $\mathcal{G}(C)$. For every $h \in \mathcal{G}(C)$, the associated Bregman distance for $(x,y)\in \dom h \times \idom h$ is given by
	    \begin{equation}\label{eq:bregman-distance}
	    D_{h}\left(x , y\right) = h\left(x\right) - h\left(y\right) - \act{\nabla h\left(y\right) , x - y}\,,
	   	\end{equation}
	   	and is set to $+\infty$ otherwise. Henceforth, we assume the following.
		\begin{assumption} \label{A:AssumptionA0}
			\begin{itemize}
				\item[$\rm{(i)}$] $h \in \mathcal{G}(\real^{d})$ with $C =\real^{d}$. 
				\item[$\rm{(ii)}$] $g : \real^{d} \rightarrow \R$ is  continuously differentiable. 
			\end{itemize}
		\end{assumption}

For non-convex functions, the extension of Lipschitz continuity is referred to as the $L$-smad property which we record below. 
\begin{definition}\label{D:l-smad}
	A pair $(g,h)$ is \textit{$L$-smooth adaptable ($L$-smad)}  on $\real^{d}$ if there exists $L>0$ such that  $Lh -g$ and $Lh +g$ are convex on $\real^{d}$.
\end{definition}
\begin{remark}
Note that we can always always assume that $L=1$ by absorbing the constant $L$ into $h$. Also, if a pair $(g,h)$ is $L$-smad on $\real^{d}$, we can equivalently say that $g$ is $L$-smad  on $\real^{d}$ with respect to $h$. 
\end{remark}
The $L$-smad property can be reformulated in terms of Bregman distances, which yields the extended Descent Lemma (see \cite[Lemma 2.1, p. 2134]{BSTV2018}).

\begin{lemma}[Extended Descent Lemma] \label{L:ExtenDL}
	The pair of functions $(g , h)$ is $L$-smooth adaptable on $\real^{d}$ if and only if for all $ x , y \in \R^d$ the following holds
	\begin{equation} \label{L:NoLipsDecent:1}
		\left| g\left(x\right) - g\left(y\right) - \act{\nabla g\left(y\right) , x - y} \right| \leq LD_{h}\left(x , y\right)\,.
	\end{equation}
\end{lemma}
For $h=(1/2)\norm{\cdot}^2$, the notion of $L$-smoothness and the classical Descent Lemma are recovered.  

\subsection{Bregman Proximal Gradient}
In analogy to the Euclidean case the extended Descent Lemma motivates us to consider the following iterative majorize-minimize scheme, which minimizes the following upper bound at each iteration $k$.

Let $x^k \in \R^d$. The extended Descent Lemma yields:
\begin{align} \label{eq:upper-bound}
g(x) &\leq g(x^k) + \act{\nabla g\left(x^k \right) , x - x^k} +LD_{h}\left(x , x^k\right) \notag \\
& =: M_k(x),
\end{align}
with $M_k(x^k) = g(x^k)$. Then clearly for $x^{k+1}$ given as
\begin{align}
x^{k+1} \in \underset{x \in \real^{d}}{\argmin} ~M_k(x),
\end{align}
we have $g(x^{k+1}) \leq M_k(x^{k+1}) \leq M_k(x^k) = g(x^{k})$, i.e. a descent on the objective function. Notably, for $h=(1/2)\|\cdot\|^2$ we recover the classical gradient descent method and more generally the mirror descent \cite{BT2003} algorithm. Like in the classical proximal gradient method the majorization property of $M_k$ still holds if we add a second convex non-smooth term $f$ to both sides of the inequality~\eqref{eq:upper-bound}. Minimization of $M_k + f$ then yields the Bregman proximal gradient scheme for non-convex additive composite problems given as
		%\textcolor{red}{Need to check?. The definition of $\Psi$ is unclear here. Should we set $C = \R^d.$}
		\begin{equation}\label{eq:main-problem}
			(\PPP) \qquad \inf \left\{ \Psi(x) := f\left(x\right) + g\left(x\right) : \; x \in \real^{d} \right\}\,,
		\end{equation}
where $g,h$ satisfy Assumption~\ref{A:AssumptionA0}. The complete BPG algorithm is given in Algorithm~\ref{alg:bpg}. It is formulated in terms of the Bregman Proximal Gradient (BPG) mapping given by
		\begin{align} \label{eq:proximal-gradient-map}
			T_{\lambda}\left(x\right) 
			 :=  \underset{u \in \real^{d}}{\argmin} \left\{ f\left(u\right) + \act{\nabla g\left(x\right) , u} + \frac{1}{\lambda} D_{h}\left(u , x\right) \right\}.
		\end{align}
	This generalizes the proximal gradient mapping, by replacing the Euclidean distance with a Bregman distance.  \ifpaper\else\medskip\fi

	{\centering
		\fcolorbox{black}{white}{\parbox{0.97\columnwidth}{{}
		\vspace{-0.5em}
		\begin{algorithm}[BPG: Bregman Proximal Gradient \cite{BSTV2018}]\label{alg:bpg}\leavevmode\newline
				{\textbf{Input.}} Choose $h \in \mathcal{G}(\real^{d})$  such that $g$ satisfies $L$-smad with respect to $h$ on $\real^{d}$.\\
				{\textbf{Initialization.}} $x^1 \in  \idom h$ and $0<\lambda< (1/L) $. \\
				{\textbf{General Step.}} For $k\geq 1$, compute $x^{k+1} \in  T_{\lambda}(x^k)$.
		\end{algorithm}
		\vspace{-0.5em}
		}}
	}\\
		
	For convergence and well-definedness we require the following standard assumption.

		\begin{assumption} \label{A:AssumptionA}
			\begin{itemize}
				\item[$\rm{(i)}$] $f : \real^{d} \rightarrow \erl$ is a proper, lower semicontinuous, convex function.
				\item[$\rm{(ii)}$]  $v(\PPP) := \inf \left\{ \Psi\left(x\right) : \; x \in \real^{d} \right\} > -\infty$.
				\item[$\rm{(iii)}$] $h$ is $\sigma$-strongly convex on $\real^{d}$.  
				\item[$\rm{(iv)}$] For all $\lambda>0$, the function $h+\lambda f$ is supercoercive, thus satisfying 
				\[
				\lim_{\norm{x} \to \infty} \frac{h(x) + \lambda f(x)}{\norm{x}} = \infty\,.
				\]
			\end{itemize}
		\end{assumption}

	Assumption~\ref{A:AssumptionA}$\rm{(iv)}$ ensures the well-definedness of the $T_{\lambda}$, in the sense that $T_{\lambda}$ is non-empty and compact. 

% \vspace{0.3em}
	We provide below the condensed global convergence result from \cite{BSTV2018}, which states the convergence of the full sequence generated by BPG to a stationary point. The global convergence of Bregman proximal algorithms relies on the standard non-smooth Kurdyka--{\L}ojasiewicz (KL) property \cite{BDLS2007,AB2009}. The KL property is satisfied for semi-algebraic functions (see for example \cite{ABRS2010}). Note that $g$ in \eqref{eq:prob-2} is a real polynomial function, thus semi-algebraic. For the remainder of this paper, we restrict ourselves also to semi-algebraic $f$, for e.g., standard L1 norm and squared L2 norm (see \cite{XY2013}). 
	\begin{theorem}[Global Convergence of BPG]\label{thm:thm-main-bpg} Let  Assumptions~\ref{A:AssumptionA0},\ref{A:AssumptionA} hold and let $g$ be $L$-smad with respect to $h$. Assume $\nabla g, \nabla h$ to be Lipschitz continuous on any bounded subset.  Let  $\left\{x^k\right\}_{k \in \nn}$ be a bounded sequence generated by BPG with $0<\lambda L<1$, and suppose $\Psi$ satisfies the KL property, then, such a sequence has finite length, and converges to a critical point. %$\left({\bf W_1}^{\ast},{\bf W_2}^{\ast}\right) \in \crit \Psi$.
	\end{theorem}

	By a critical point, we mean a point for which the limiting subdifferential of the objective contains zero, i.e., Fermat's rule is satisfied \cite[Theorem 10.1]{RW1998-B}. The boundedness assumption in the statement is automatically satisfied, if, for example, the objective is coercive (lower level-bounded). 

\section{Bregman Distance for DLNN}\label{sec:dll}
This section is the main part of our paper, where we specialize $g$ to be a quadratic loss function with a deep linear neural network (DLNN). In view of Example \ref{Ex:Lsmooth-fail} such a cost function is not classically $L$-smooth and therefore lacks a quadratic upper bound even for the two layer case. Therefore our main goal is to derive a novel kernel generating distance $h$ that allows us to obtain a global upper bound. More precisely in the first part we show that $g$ satisfies the $L$-smad property for a certain non-trivial choice of $h$. In the second part we derive closed form solutions of the Bregman proximal gradient map~\eqref{eq:proximal-gradient-map} for popular choices of $f$ such as the L1- and the squared L2-norm.
To this end we consider the following optimization problem
\begin{equation}\label{eq:prob-2}
	 \min_{{\bf W_i}\in \mathcal W_i\, \forall i\in [N]}\,g({\bf W}) := \frac 12\norm{{\bf W_1W_2\cdots W_NX} - {\bf Y}}^2_{F} \,,
	\end{equation}
	where $N$ denotes the number of layers. Furthermore we denote by $\mathcal W_i = \R^{d_i\times d_{i+1}}$ where $d_i \in \nn$ for all $i \in [N]$. Let $d_{N+1} = d$ and ${\bf X} \in \R^{d \times n_T}$ be fixed, where $n_T \in \nn$, which typically corresponds to the number of training samples. Similarly we have fixed ${\bf Y} \in  \R^{d_1\times n_T}$, which typically corresponds to the labels of the inputs in ${\bf X}$. We denote by ${\bf W}:=({\bf W_1},\ldots,{\bf W_N})$, meaning ${\bf W}$ lies in the product space $\mathcal W:=\mathcal W_1 \times \cdots \times \mathcal W_N$, equipped with the norm $\norm{{\bf W}}_F^2 := \sum_{i=1}^N\norm{{\bf W_i}}_F^2$. We focus on $N\geq2$ in this paper.
 
 \subsection{Smooth Adaptable Property for DLNN}
To prove the $L$-smad property we consider its characterization via the Hessian. More precisely, $Lh - g$ and $g + Lh$ are convex if and only if $L \nabla^2 h(x) \succeq \nabla^2 g(x)$ and $-L \nabla^2 h(x) \preceq \nabla^2 g(x)$, i.e. the eigenvalues of the Hessian of $g$ are bounded by eigenvalues of the Hessian of $Lh$. The analysis suggests that $h$ and the corresponding Bregman distance involve polynomials of degree $2N$ and $N$. We consider the odd and the even case separately.
\subsubsection{Even Number of Layers} 
% \EL[inline]{Maybe you can merge this equation into the next proposition. Clearly motivate why you distinguish odd and even number of layers. Make notation for the kernel generating distance consitent.}
Let $N$ be even and define the following functions
\begin{equation*}
	H_1({\bf W}) := \left(\frac{\norm{{\bf W}}_F^2}{N} \right)^N,\,\quad 	H_2({\bf W}) :=  \left(\frac{\norm{{\bf W}}_F^2}{N} \right)^{\frac{N}{2}} \,.
\end{equation*}
Then, we have the following result, which shows that for an appropriate linear combination of $H_1$ and $H_2$ we obtain the $L$-smad property for $g$ in \eqref{eq:prob-2}. 
% \begin{align*}
% & h({\bf W_1,W_2,\ldots,W_N})\\
% & = c_1(N) \left(\frac{\norm{{\bf W_1}}_F^2 + \norm{{\bf W_2}}_F^2 + \ldots \norm{{\bf W_N}}_F^2}{N}\right)^N + c_2(N) \left(\frac{\norm{{\bf W_1}}_F^2 + \norm{{\bf W_2}}_F^2 + \ldots \norm{{\bf W_N}}_F^2}{N}\right)^{\frac{N}{2}}
% \end{align*}
% \begin{align*}
% & h({\bf W_1,W_2,\ldots,W_N})\\
% & = c_1(N) \left(\frac{\norm{{\bf W_1}}_F^2 + \norm{{\bf W_2}}_F^2 + \ldots \norm{{\bf W_N}}_F^2}{N}\right)^N + c_2(N) \left(\frac{\norm{{\bf W_1}}_F^2 + \norm{{\bf W_2}}_F^2 + \ldots \norm{{\bf W_N}}_F^2 + 1}{N+1}\right)^{\frac{N+1}{2}}
% \end{align*}
\begin{proposition}\label{prop:even-case}
	Let $H_1,H_2$ be as defined above and let $g$ be as in \eqref{eq:prob-2}. Then, for  $L= 1$, the function $g$ satisfies the $L$-smad property with respect to the following kernel generating distance
	\begin{equation}\label{eq:main-lsmad-h}
	 % H_a({\bf W_1},\ldots,{\bf W_N}) =  c_1(N)H_1({\bf W_1},\ldots,{\bf W_N}) + c_2(N) H_2({\bf W_1},\ldots,{\bf W_N}) \,,
	 H_a({\bf W}) =  c_1(N)H_1({\bf W}) + c_2(N) H_2({\bf W}) \,,
	\end{equation}
	where we have 
	% \[
	% c_1(N)  =  \frac{(2N-1)N^N}{2N!}\norm{{\bf X}}_F^2\,,  \quad c_2(N) = \frac{\norm{{\bf Y}}_F\norm{\bf X}_F (N-1) N^{\frac{N}{2}}}{ 2\binom{\frac{N}{2}}{\frac{N-2}{2},1} (N-2)^{\frac{N-2}{2}}}
	% \]
	\ifpaper
	\begin{align*}
	c_1(N)  & =  \frac{(2N-1)N^N}{2N!}\norm{{\bf X}}_F^2 \,,\\
	 c_2(N) &= \frac{\norm{{\bf Y}}_F\norm{\bf X}_F (N-1) N^{\frac{N-2}{2}}}{ (N-2)^{\frac{N-2}{2}}}\,.
	\end{align*}
	\else
	\begin{align*}
	 c_1(N)  & =  \frac{(2N-1)N^N}{2N!}\norm{{\bf X}}_F^2 \,,\quad c_2(N) = \frac{\norm{{\bf Y}}_F\norm{\bf X}_F (N-1) N^{\frac{N-2}{2}}}{ (N-2)^{\frac{N-2}{2}}}\,.
	\end{align*}
	\fi
\end{proposition}
The proof is given in Section~\ref{sec:even-case} in the appendix.\ifpaper\else\medskip\fi

Note that $H_a$ is a polynomial of order $2N$ as a linear combination of a degree $2N$ and a degree $N$ polynomial. Moreover, observe that the resulting Bregman distances are data-dependent. More precisely, the coefficients $c_1(N)$ and $c_2(N)$, are not only dependent on the number of layers but also on ${\bf X}$  and ${\bf Y}$.\ifpaper\else\medskip\fi

We remark, that for $N=2$ and $\norm{{\bf X}}_F = 1$, this matches the results from \cite{MO2019a} for the matrix factorization problems.

\subsubsection{Odd Number of Layers}
Let $N$ be odd and denote
\begin{equation}
	H_3({\bf W}) :=  \left(\frac{\norm{{\bf W}}_F^2 + 1}{N+1}  \right)^{\frac{N+1}{2}} \,.
\end{equation}

As the following proposition reveals,  the loss function for the odd case is $L$-smooth adaptable with respect to a degree $2N$ polynomial $H_b$ which is given as a linear combination of $H_1$ and $H_3$.
\begin{proposition}\label{prop:odd-case}
	Let $H_1,H_3$ be as defined above and let $g$ be as in \eqref{eq:prob-2}. Then, for $L= 1$, the function $g$ satisfies the $L$-smad property with respect to the following kernel generating distance
	\begin{equation}\label{eq:main-lsmad-h}
	 H_b({\bf W}) =  c_1(N)H_1({\bf W}) + c_3(N) H_3({\bf W}) \,,
	\end{equation}
	where we have 
	% \[
	% c_1(N)  =  \frac{(2N-1)N^N}{2N!}\norm{{\bf X}}_F^2\,,  \quad c_3(N) = \frac{\norm{{\bf Y}}_F\norm{\bf X}_F(N-1)(N+1)^{\frac{N+1}{2}}}{ 2 \binom{\frac{N+1}{2}}{\frac{N-1}{2},1} (N-1)^{\frac{N-1}{2}}} 
	% \]
		\ifpaper
	\begin{align*}
	c_1(N) &=  \frac{(2N-1)N^N}{2N!}\norm{{\bf X}}_F^2\,,\\
	c_3(N) &= \frac{\norm{{\bf Y}}_F\norm{\bf X}_F(N-1)(N+1)^{\frac{N-1}{2}}}{  (N-1)^{\frac{N-1}{2}}}\,.
	\end{align*}
		\else
		\begin{align*}
		c_1(N) &=  \frac{(2N-1)N^N}{2N!}\norm{{\bf X}}_F^2\,,\quad c_3(N) = \frac{\norm{{\bf Y}}_F\norm{\bf X}_F(N-1)(N+1)^{\frac{N-1}{2}}}{  (N-1)^{\frac{N-1}{2}}}\,.
		\end{align*}
		\fi
\end{proposition}
The proof is given in Section~\ref{sec:odd-case} in the appendix. \ifpaper\else\medskip\fi

Like in the even case $H_1$ is a polynomial of order $2N$. But, here $H_2$ is not applicable as $N$ is odd. We fix this issue using $H_3$,  a polynomial of order $N+1$. Note that the analysis of the objective results in a polynomial of degree only $N$. This is automatically resolved with $H_3$, because the constant term $1$ in $H_3$ allows for certain terms to be of order $N$, while preserving the convexity of $H_3$. Note that this is just one potential way to obtain polynomials of order $N$. Considering the practical applicability we show that the proposed Bregman distances are efficient to implement in practice.\ifpaper\else\medskip\fi

\textbf{Strong convexity of $h$.} The global convergence results of Bregman proximal algorithms, provided in the next section, rely on the strong convexity of $h$. We denote $\sigma$ as the strong convexity parameter. Notably, for $N=2$ the strong convexity is satisfied directly by $H_a$. For the general case denote $ H_4({\bf W}) = \frac{\norm{{\bf W}}_F^2}{N}$. For $N>2$ and if $N$ is even, then with any $\rho>0$, we use the following $h$ 
	\[
		h({\bf W}) = H_a({\bf W}) + \rho H_4({\bf W}) \,,
	\]
	for which $\sigma=\frac{2\rho}{N}$. For $N>2$ and $N$ being odd, we use the following $h$ 
	\[
		h({\bf W}) = H_b({\bf W}) + \rho H_4({\bf W}) \,,
	\]
	 with any $\rho\geq 0$, where $ \sigma = \frac{1}{(N+1)^{\frac{N-1}{2}}} +\frac{2\rho}{N}$.
	We fix $\rho$ in the initialization phase of the algorithms.

\subsection{Closed Form Updates for BPG}\label{ssec:closed-form}
% Several closed form proximal mappings using Euclidean distance are available. But, such results are not applicable for  Bregman proximal mappings. Moreover, it is non-trivial to compute the closed form expressions for the Bregman proximal mapping $T_\lambda$ in \ref{eq:proximal-gradient-map}. The popular techniques involve computing a certain convex conjugate function \cite{BT2003}, however it is in general difficult to follow such a strategy. To this end, we provide such non-trivial closed form update steps. Such closed form solutions are also valid for any Bregman proximal algorithm or a mirror descent based algorithm, for example, stochastic mirror descent setting \cite{DDM2018}. The updates for BPG and CoCaIn BPG algorithms rely on the same strategy, thus we focus BPG update steps. Here, we denote $g = \Psi$ from \eqref{eq:prob-2} and $f:=0$ and we set $h$ as in Section~\ref{sec:dll}.
While closed form solutions of Euclidean proximal mappings are typically available for common choices of $f$, it is in general difficult to compute the Bregman proximal mapping ($T_\lambda$ in \eqref{eq:proximal-gradient-map}) in closed form, even for common $f$. Typically this involves the computation of the convex conjugate function of the problem-dependent $h$ which can be hard to derive. In our case we show in Proposition~\ref{prop:closed-form-1}, that the computation of the Bregman proximal gradient map~\eqref{eq:proximal-gradient-map} can be reduced to a simple projection problem and a simple one-dimensional nonlinear equation, more precisely a polynomial equation with a unique real root. We remark that this closed form solution is also valid for any other Bregman proximal algorithm including, stochastic BPG \cite{DDM2018}. We denote $g = \Psi$ from \eqref{eq:prob-2} and $f:=0$ and we set $h$ as in Section~\ref{sec:dll}.

\begin{proposition}\label{prop:closed-form-1}
	In BPG, with above defined $g,f,h$, denoting ${\bf P^k_i} := \lambda\nabla_{\bf W_i} g\left({\bf W}^{{\bf k}} \right) - \nabla_{\bf W_i} h({\bf W}^{{\bf k}})$\,, the update steps in each iteration are given by 
	\[
		{\bf W_{i}^{k+1}} = -r\,\frac{\sqrt{N}\,{\bf P^k_{i}}}{\norm{\bf P}_F}\,,
	\]
	for all  $i \in [N]$, where   $\norm{{\bf P}}_F^2 = \sum_{i=1}^{N}\norm{{\bf P^k_i}}_F^2$. Then for $N=2$, $r \geq 0$ satisfies 
	\begin{equation}\label{eq:even-sol-2}
	  2c_1(2) r^{3} + c_2(2)r  -\frac{\norm{{\bf P}}_F}{\sqrt{2}} = 0\,, 
	\end{equation}
	 if $N>2$ and even, $r \geq 0$ satisfies 
	\begin{equation}\label{eq:even-sol}
		  2c_1(N) r^{2N-1} + c_2(N)r^{N - 1}  + \frac{2\rho }{N}r -\frac{\norm{{\bf P}}_F}{\sqrt{N}} = 0\,,
	\end{equation}
	% \[
	% 2c_1(N)\left(\frac{\norm{{\bf P}}_F^2}{N}\right)^{N-1}r^{2N-1} + c_2(N)\left(\frac{\norm{{\bf P}}_F^2}{N}\right)^{\frac{N}{2} - 1}r^{N-1}  + \frac{2\rho }{N} r -1 = 0\,,
	% \]
	and, if $N>2$ and odd, $r \geq 0$ satisfies 
	\ifpaper
	\begin{align}
		  2c_1(N) r^{2N-1} &+ c_3(N)\left(\frac{N r^2+1}{N+1}\right)^{\frac{N-1}{2}}r \nonumber \\
		  &+ \frac{2\rho }{N}r -\frac{\norm{{\bf P}}_F}{\sqrt{N}} = 0 \label{eq:odd-sol}\,.
	\end{align}
	\else
	\begin{align}
		  2c_1(N) r^{2N-1} &+ c_3(N)\left(\frac{N r^2+1}{N+1}\right)^{\frac{N-1}{2}}r + \frac{2\rho }{N}r -\frac{\norm{{\bf P}}_F}{\sqrt{N}} = 0 \label{eq:odd-sol}\,.
	\end{align}
	\fi

	% \begin{align*}
	% &2c_1(N)\left(\frac{\norm{{\bf P}}_F^2}{N}\right)^{N-1}r^{2N-1} + c_3(N)\left(\frac{\left(r^2\norm{{\bf P}}_F^2\right) + 1}{N+1}\right)^{\frac{N-1}{2}}r 
	%  + \frac{2\rho }{N} r -1 = 0\,,
	% \end{align*}

	\end{proposition}
The proof is given in Section~\ref{sec:closed-form} in the appendix.
\paragraph*{Weight decay or L2-regularization.}  Consider 
	\begin{equation} \label{eq:objective-reg}
	 \min_{{\bf W_i}\in \mathcal W_i\, \forall i\in [K]}\, \left\{ \Psi_1({\bf W}) :=  \Psi({\bf W}) +\frac{\lambda_0}{2}\norm{{\bf W}}_F^2\right\}\,,
	\end{equation}
where $\lambda_0 >0$ and the term $\frac{\lambda_0}{2}\norm{{\bf W}}_F^2$ is the L2-regularizer. The closed forms are obtained by replacing $\frac{2\rho}{N}$ with $\left(\frac{2\rho}{N} + \lambda\lambda_0\right)$ in Proposition~\ref{prop:closed-form-1}, by setting $f({\bf W}) := \frac{\lambda_0}{2}\sum_{i=1}^N\norm{{\bf W_i}}_F^2$. 
% It is possible to handle various regularizers in BPG and CoCaIn BPG, including nuclear norm regularizer, L1-norm regularizer, Elastic regularization and many others using the same strategy as in \cite{MO2019a}. 

\paragraph*{L1-Regularization.} It is also possible to obtain the closed form solutions when L1-regularization is used, where we set	$f({\bf W}) := \sum_{i=1}^N \mu_i \norm{{\bf W_i}}_1$.  Then using the element wise soft-thresholding operator $\SSS_{\theta}(x) = \max\{|x|-\theta,0\}\sgn(x)$, the closed form updates are obtained by replacing $-{\bf P_i^k}$ with $\SSS_{\lambda\mu_i}(-{\bf P_i^k})$ in Proposition~\ref{prop:closed-form-1}. Proof is given in Section~\ref{sec:closed-form-l1}, in the appendix.

%\paragraph*{Computational Complexity.} Apart from gradients of $\Psi$, the additional computational complexity is $O(\sum_{i=1}^{N}d_{i}d_{i+1})$, as all the operations are linear. Other state of the art optimizers PALM \cite{BST2014} and iPALM \cite{PS2016}, involve additional computations. For example, the Lipschitz constant computation for each block involves expensive matrix products.  The computations of constants $c_1(N)$ has complexity of $O(dn_T)$. Reusing, the computation of $\norm{{\bf X}}_F$ from $c_1(N)$, the additional complexity required for $c_2(N)$ or $c_3(N)$ is $O(d_1n_T)$. PALM and iPALM are inherently serial block-wise,  due to the computation of Lipschtiz computation for a block depends on the updates of previous blocks. However, for BPG,  the update steps given in Proposition~\ref{prop:closed-form-1} for each weight are proportional to ${\bf P_i^k}$ which are independent. The computation of $\norm{{\bf P}}_F$,  can be done very fast, for e.g., with asynchronous updates. 

\section{Closed Form Inertial BPG}\label{sec:cocain-main}

	%Nesterov's momentum based extrapolation inertial methods (also FISTA \cite{BT2009}) are very popular in optimization with numerous applications. Such a momentum strategy uses extrapolation between the current and previous iterate. The standard extrapolation parameter or the inertial parameter relies on the iteration number. One might ask ``Is it possible to have locally adaptive inertia, where the function surface governs the inertial parameter?''. Recently Convex-Concave Inertial (CoCaIn) BPG  \cite{MOPS2019} was proposed, based on Nesterov's extrapolation which incorporates inertia in an adaptive manner, dependent on the function surface rather than iteration number. 
 In this section, we present an important contribution for efficiently using a momentum based BPG method. We focus on the recently introduced Convex-Concave Inertial (CoCaIn) BPG  \cite{MOPS2019}, which uses Nesterov-type extrapolation in BPG for non-smooth non-convex optimization problems. It is given in Algorithm~\ref{alg:cocain}. Besides inertia, the key feature of CoCaIn BPG is the usage of different constants for the upper bound $\uL h -g$ and lower bound $\lL h+g$. Since the amount of extrapolation is closely tied to the lower bound, tight approximations are desirable. \ifpaper\else\medskip\fi

\ifpaper
\begin{figure}[ht]
\centering % changing from Ivory2 to Orange!10
		\fcolorbox{black}{white}{\parbox{0.98\columnwidth}{{}
		\vspace{-0.5em}
		\begin{algorithm}[Convex-Concave Inertial (CoCaIn) BPG \cite{MOPS2019}]\label{alg:cocain}\leavevmode\newline
			{\bf Input.} $\delta, \varepsilon > 0$ with $1>\delta > \varepsilon$. \\
			{\bf Initialization.} $x^{0} = x^{1} \in \idom h \cap \dom f$, ${\bar L}_{0} > 0$ and $\tau_{0} \leq {\bar L}_{0}^{-1}$. \\
			{\bf General Step.} For $k = 1 , 2 , \ldots$, compute $y^{k}  = x^{k} + \gamma_{k}\left(x^{k} - x^{k - 1}\right) \in \idom h,$
				where $\gamma_{k}$ satisfies
				\begin{equation} \label{CoCaIn:3}
					\frac{\left(\delta - \varepsilon\right)}{\left(1 + {\underline L_{k}}\tau_{k-1}\right)}D_{h}\left(x^{k - 1} , x^{k}\right) \geq D_{h}\left(x^{k} , y^{k}\right)
				\end{equation}
				 where $\lL_{k}$ satisfies ($D_g$ formally defined as in \eqref{eq:bregman-distance})
				 \[
				 D_g(x^k,y^k) \geq - \lL_{k}D_{h}\left(x^{k} , y^{k}\right)\,.
				 \]

				 Choose $\uL_{k} \geq \uL_{k - 1}$, set $\tau_{k} \leq \min \left\{ \tau_{k - 1} , \uL_{k}^{-1} \right\}$ and compute $x^{k+1} \in T_{\tau_k}(y^k)$ with $\uL_{k}$ fulfilling 
				 \[
				 D_g(x^{k+1},y^{k}) \leq  \uL_{k}D_{h}\left(x^{k + 1} , y^{k}\right).
				 \]
		\end{algorithm}
		\vspace{-0.5em}
				}}
\end{figure}
\else
\begin{figure}[t]
		\centering % changing from Ivory2 to Orange!10
		\fcolorbox{black}{white}{\parbox{0.98\textwidth}{{}
		\begin{algorithm}[Convex-Concave Inertial (CoCaIn) BPG \cite{MOPS2019}]\label{alg:cocain}\leavevmode\newline
			{\bf Input.} $\delta, \varepsilon > 0$ with $1>\delta > \varepsilon$. \\
			{\bf Initialization.} $x^{0} = x^{1} \in \idom h \cap \dom f$, ${\bar L}_{0} > 0$ and $\tau_{0} \leq {\bar L}_{0}^{-1}$. \\
			{\bf General Step.} For $k = 1 , 2 , \ldots$, compute
				\begin{align*}
					y^{k} & = x^{k} + \gamma_{k}\left(x^{k} - x^{k - 1}\right) \in \idom h, \label{CoCaIn:1}
				\end{align*}
				where $\gamma_{k}$ is chosen such that
				\begin{equation} \label{CoCaIn:3}
					\left(\delta - \varepsilon\right)D_{h}\left(x^{k - 1} , x^{k}\right) \geq \left(1 + {\underline L_{k}}\tau_{k-1}\right)D_{h}\left(x^{k} , y^{k}\right)
				\end{equation}
				holds and such that $\lL_{k}$ satisfies
				\begin{equation*} \label{CoCaIn:4}
					g\left(x^{k}\right) \geq g\left(y^{k}\right) + \act{\nabla g\left(y^{k}\right) , x^{k} - y^{k}} - \lL_{k}D_{h}\left(x^{k} , y^{k}\right).
				\end{equation*}
				Now, choose $\uL_{k} \geq \uL_{k - 1}$, set $\tau_{k} \leq \min \left\{ \tau_{k - 1} , \uL_{k}^{-1} \right\}$ and compute
				\begin{equation*} \label{CoCaIn:2}
					x^{k + 1} \in  \argmin_{u} \left\{ f\left(u\right) + \act{\nabla g\left(y^{k}\right) , u - y^{k}} + \frac{1}{\tau_{k}}D_{h}\left(u , y^{k}\right) \right\} 										\end{equation*}
				with $\uL_{k}$ fulfilling
				\begin{equation*} \label{CoCaIn:5}
					g\left(x^{k + 1}\right) \leq g\left(y^{k}\right) + \act{\nabla g\left(y^{k}\right) , x^{k + 1} - y^{k}} + \uL_{k}D_{h}\left(x^{k + 1} , y^{k}\right) .
				\end{equation*}
				\end{algorithm}
				}}
	\end{figure}
\fi

Moreover, CoCaIn BPG provides the possibility to adapt the upper and lower bound locally via a backtracking line search strategy. The maximal extrapolation is restricted by the inequality in \eqref{CoCaIn:3}, which can be incorporated into the same backtracking loop. Note that CoCaIn BPG does not require nested loops to satisfy all conditions. The following convergence result analog to Theorem~\ref{thm:thm-main-bpg} holds.
%%
%%	We firstly provide the intuition behind CoCaIn BPG. In the definition of $L$-smad property, note that $L$ for the upper bound, $Lh -g$ and the lower bound, $Lh +g$ are the same. But, it is possible that $L$ is different for lower and upper bounds. This was initially observed in \cite[Remark 2.1]{BSTV2018}. Later, such a difference was exploited in CoCaIn BPG, for adaptive inertia. Related early works in the context of functions with Lipschitz continuous gradients  include \cite{WCP2017,Ochs18}. 
%%	
%%
%%	We provide CoCaIn BPG in Algorithm 2. There are two main steps in CoCaIn BPG. The first step is to generate the extrapolation $y^k$ as a linear combination of $x^k$ and $x^{k-1}$ governed by the inertial parameter $\gamma_k$. The inertial parameter is found via \eqref{CoCaIn:3} in CoCaIn BPG, with the standard backtracking strategy. The second step is the update step, which is similar to BPG. However, unlike BPG, the step size is not fixed here and is found via backtracking. Contrary to the standard backtracking based algorithms, CoCaIn BPG employs two backtracking steps, one for the lower bound $\lL_k$ to solve $D_g(x^k,y^k) \geq - \lL_{k}D_{h}\left(x^{k} , y^{k}\right)$ and one for the upper bound $\uL_k$ for $D_g(x^{k+1},y^{k}) \leq  \uL_{k}D_{h}\left(x^{k + 1} , y^{k}\right)$. The lower bound $\lL_k$ is associated to the inertial parameter and the upper bound $\uL_k$ is associated to the step size. When $\gamma_k$ is explicitly set to zero, then we obtain BPG with backtracking. We provide below the global convergence result from \cite{MOPS2019}.
	\begin{theorem}[Global Convergence of CoCaIn BPG]\label{thm:thm-main}  Let  Assumptions~\ref{A:AssumptionA0},\ref{A:AssumptionA} hold, let $g$ be $L$-smad with respect to $h$. Assume $\nabla g$ and $\nabla h$ to be Lipschitz continuous on any bounded subset in $\R^d$. Let  $\left\{x^{k}\right\}_{k \in \nn}$ be a bounded sequence generated by  CoCaIn BPG, and suppose $f,g$ satisfy the KL property, then, such a sequence has finite length, and converges to a critical point.%$\left({\bf W_1}^{\ast},{\bf W_2}^{\ast}\right) \in \crit \Psi$.
	\end{theorem}

CoCaIn BPG uses an extrapolation strategy where in each iteration we need to solve \eqref{CoCaIn:3}, for a certain constant $\kappa >0$, the following condition has to be satisfied
\begin{equation}\label{eq:inertial-step}
D_{h}\left(x^{k} , y^{k}\right) \leq \kappa D_{h}\left(x^{k - 1} , x^{k}\right)\,,
\end{equation}
which involves finding $\gamma_k \in [0,1]$, where $y^k = x^k + \gamma_k (x^k - x^{k-1})$. 
%A naive way to obtain $\gamma_k$ is to use the backtracking strategy where we start with $\gamma_k=1$, then slowly reduce by a small factor until the condition \eqref{eq:inertial-step} is satisfied.  However, f
For large scale applications, including deep learning, checking the condition in a backtracking loop may be expensive. Hence, we contribute to an efficient implementation of the CoCaIn BPG extrapolation step by providing closed form solution for the extrapolation parameter. For Euclidean distances, we  obtain that $0< \gamma_k \leq \sqrt{\kappa}$ satisfies \eqref{CoCaIn:3}. Such a closed form interval is non-trivial to obtain in general. But,  the structure of the proposed Bregman distances allows also for closed form inertial parameter.
% \begin{lemma}\label{lem:even-extrapolation}
% Let $N>2$ and $N$ be even. Denote $x^k = ({\bf W_1^k},\ldots,{\bf W_N^k})$. For $\kappa>0$ and $y^k:=x^k+\gamma_k(x^k - x^{k-1})$ satisfying
% \[
% 	D_{h}\left(x^{k} , y^{k}\right) \leq \kappa D_{h}\left(x^{k - 1} , x^{k}\right)\,,
% \]
% then the parameter $\gamma_k$ satisfies the following
% \[
% 	0<\gamma_k \leq \sqrt{\frac{\kappa D_{h}\left(x^{k - 1} , x^{k}\right)}{\left(c_1(N) \mathcal{B}_k + c_2(N)\mathcal{C}_k + \rho \norm{\Delta_k}^2\right) }}\,.
% \]
% \end{lemma}
\begin{proposition}\label{lem:even-extrapolation}
Denote $x^k = ({\bf W_1^k},\ldots,{\bf W_N^k})$. For $\kappa>0$, $y^k:=x^k+\gamma_k(x^k - x^{k-1})$ and $x^k \neq x^{k-1}$,  the parameter $\gamma_k$ given by
\[
	0<\gamma_k \leq \sqrt{\frac{\kappa D_{h}\left(x^{k - 1} , x^{k}\right)}{ \chi(N)}} \leq 1
\]
satisfies condition \eqref{eq:inertial-step}, where for $N=2$, we set $\chi(N) = c_1(N) \mathcal{B}_k + c_2(N)\mathcal{C}_k$, for even $N>2$, we set 
\[
	\chi(N)=\left(c_1(N) \mathcal{B}_k + c_2(N)\mathcal{C}_k + \rho \norm{\Delta_k}^2\right)\,,
\] and for odd $N>2$, we set 
\[
	\chi(N) = \left(c_1(N) \mathcal{B}_k + c_3(N)\mathcal{D}_k + \rho \norm{\Delta_k}^2\right)\,,
\]
with $\Delta_k := {x^k - x^{k-1}}$, $\Omega_k := 2\norm{x^k}^2 + 2\norm{\Delta_k}^2 $ and $
	\mathcal{B}_k := \left(\frac{(2N-1)}{N^{N-1}}\right) \norm{\Delta_k}^2 (\Omega_k)^{(N-1)}\,.
$
For even $N$ we denote 
$
 \mathcal{C}_k := \left(\frac{N-1}{N^{\frac{N}{2} -1}}\right)\norm{\Delta_k}^2(\Omega_k)^{\frac{N-2}{2}}\,.
$
For odd $N$ we denote
$
\mathcal{D}_k :=  \frac{N}{(N+1)^{\frac{N-1}{2}}} \norm{\Delta_k}^2 \left(\Omega_k + 1 \right)^{\frac{N-1}{2}}\,.
$
% for $N=2$, the parameter $\gamma_k$ satisfies the following
% \[
% 	0<\gamma_k \leq \sqrt{\frac{\kappa D_{h}\left(x^{k - 1} , x^{k}\right)}{\left(c_1(N) \mathcal{B}_k + c_2(N)\mathcal{C}_k \right) }}\,,
% \]
% then for  $N>2$ and $N$ being even with $\chi_1(N) = \left(c_1(N) \mathcal{B}_k + c_2(N)\mathcal{C}_k + \rho \norm{\Delta_k}^2\right)$ the parameter $\gamma_k$ satisfies 
% \[
% 	0<\gamma_k \leq \sqrt{\frac{\kappa D_{h}\left(x^{k - 1} , x^{k}\right)}{ \chi_1(N)}}\,,
% \]
% and for  $N>2$ and $N$ being odd with $\chi_2(N) = \left(c_1(N) \mathcal{B}_k + c_3(N)\mathcal{D}_k + \rho \norm{\Delta_k}^2\right)$ the parameter $\gamma_k$ satisfies 
% \[
% 	0<\gamma_k \leq \sqrt{\frac{\kappa D_{h}\left(x^{k - 1} , x^{k}\right)}{\chi_2(N)}}\,.
% \]
\end{proposition}
The proof of Proposition~\ref{lem:even-extrapolation} is given in Section~\ref{proof:even-extrapolation}. For $N=2$ (Matrix Factorization) we provide novel tighter bounds in Section~\ref{lem:mat-fac-closed-form-inertia} in the appendix.

\section{Discussion of BPG Variants} \label{sec:discuss-settings-for-algs}

The proposed Bregman distances for DLNN allow for variants that can be adapted to specialized settings, for example, stochastic extensions.  In the following, we comprehensively discuss the applicability and performance of \emph{BPG based algorithms for DLNN} compared to several existing optimization schemes.\ifpaper\else\medskip\fi

\textbf{The base algorithm BPG.} The key advantage of BPG for DLNN compared to its Euclidean variant, the Proximal Gradient (PG) method, is the guaranteed convergence when a constant step size rule is used. This fact is enabled by validity of (global) relative smoothness (Proposition~\ref{prop:even-case} and~\ref{prop:odd-case}). On the contrary, PG, which requires a classical $L$-smoothness can only be used by the following trick. Under a coercivity assumption, all iterates generated by PG lie in a compact set, on which a global Lipschitz constant for the objective's gradient can be found. However, the compact set is usually unknown (and cannot be determined before running the algorithm), which makes the practical computation of such a global Lipschitz constant difficult. A good heuristic guess may result in PG being more efficient than BPG. Therefore, BPG and CoCaIn BPG (with $\uL=\lL$) render promising alternatives to PG when line search must be avoided due to a prohibitively expensive function evaluation. \ifpaper\else\medskip\fi

\textbf{BPG with Backtracking.} If backtracking line search variants are affordable for solving the given optimization problem, then BPG, CoCaIn BPG and their Euclidean variants PG and iPiano provide the same convergence guarantees. Intuitively, from a global perspective, the adapted upper and lower bounds given by the Bregman distance for BPG should tighter to the objective function than quadratic functions of $L$-smoothness. But, this situation can change when backtracking line search is used and only locally tight approximations are sought. We cannot claim that any of the two strategies has a clear and consistent advantage. The performance can depend significantly on the starting point and the initialization of the line search parameters and needs problem dependent exploration.\ifpaper\else\medskip\fi

\textbf{BPG vs PALM.} Proximal Alternating Linearized Minimization (PALM) \cite{BST2014} has a clear bias towards the first block of coordinates, if the update direction points into a narrow valley. This effect may be compensated by its inertial variant iPALM. For DLNN with identical regularizers, this effect cannot be observed due to the symmetry of the objective function with respect to the blocks of coordinates, resulting in an oftentimes favorable performance. We leave the exploration of alternating variants of BPG as future work. Some of the related works include \cite{LZTW2019,HGP2019}.\ifpaper\else\medskip\fi

\textbf{Alternating vs non-alternating strategies.} We would like to stress two important advantages of non-alternating schemes such as BPG over alternating minimization strategies. Firstly, BPG allows for block-wise parallelization, and, secondly, there are interesting settings for which alternating minimization is not applicable. The obvious example is symmetric Matrix Factorization, for which BPG is studied in \cite{DBA2019}. In the context of DLNN ($N>2$ in \eqref{eq:prob-2}) requiring $W_1=W_2=\ldots=W_N$ (upto a transpose) can be considered as a prototype for an unrolled recurrent neural network architecture, where weights are shared across layers. Here, there is no natural way to apply alternating minimization schemes and the objective is not classically $L$-smooth.\ifpaper\else\medskip\fi

\textbf{Stochastic setting extensions.} A stochastic version of BPG was developed recently in \cite{DDM2018}. The proposed Bregman distances are also valid here and can be applied for training DLNN. Furthermore, several popular stochastic variants such as Adam \cite{KB2014}, Adagrad \cite{DHS2011}, SC-Adagrad \cite{MH2017} can potentially be extended with Bregman proximal framework.

\section{Experiments}\label{experiment1}

\begin{figure*}[hbt!]
  \centering
  \begin{tabular}[b]{c}
    \includegraphics[width=.3\textwidth]{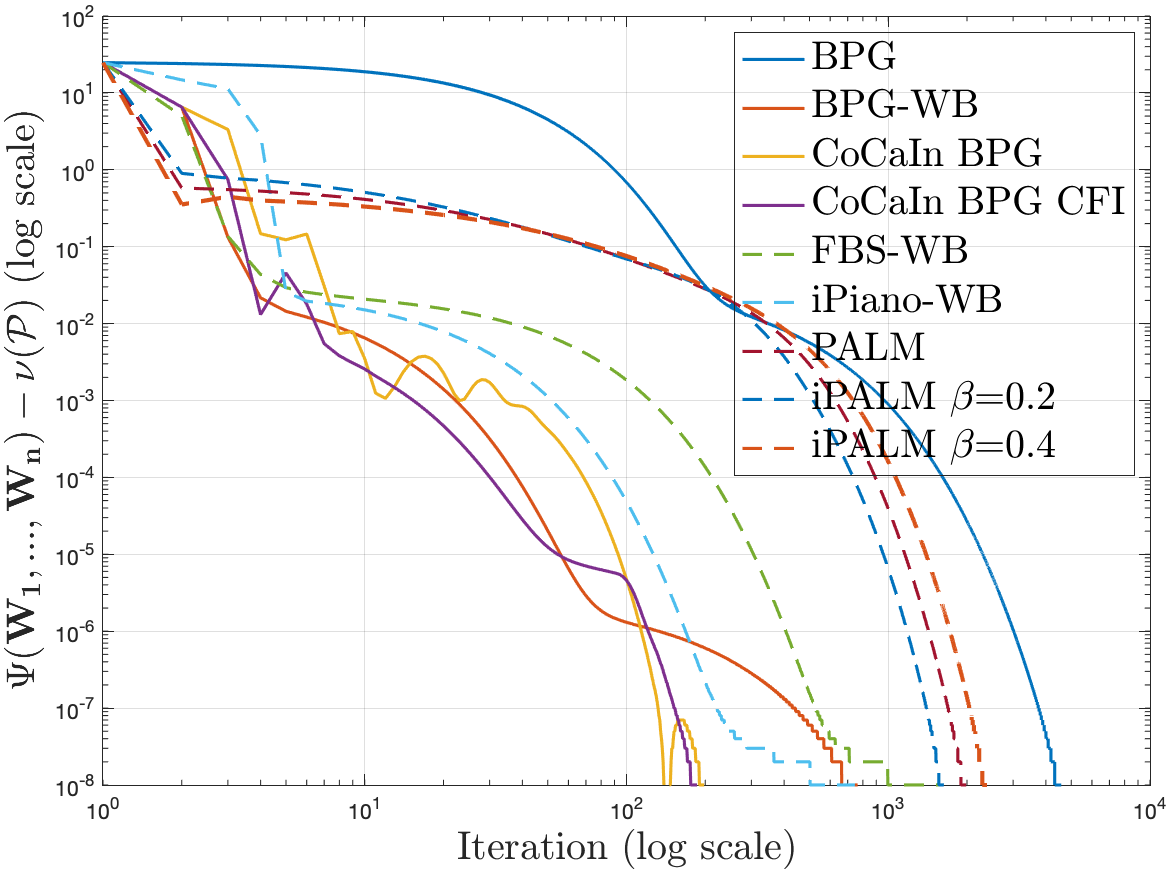} \\
    \small (a) L2-Regularization ($N=3$)
  \end{tabular}
  \begin{tabular}[b]{c}
    \includegraphics[width=.3\textwidth]{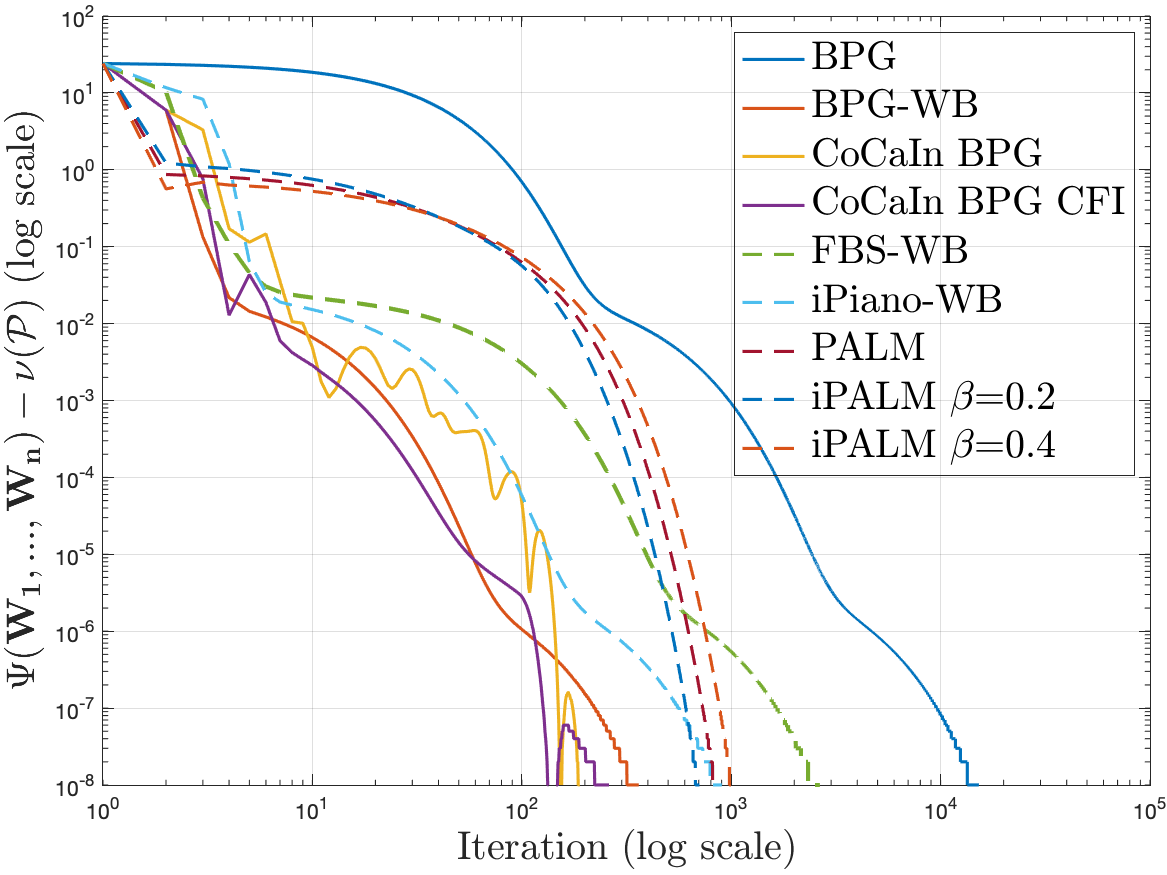} \\
    \small (b) L1-Regularization ($N=3$)
  \end{tabular}
    \begin{tabular}[b]{c}
    \includegraphics[width=.3\textwidth]{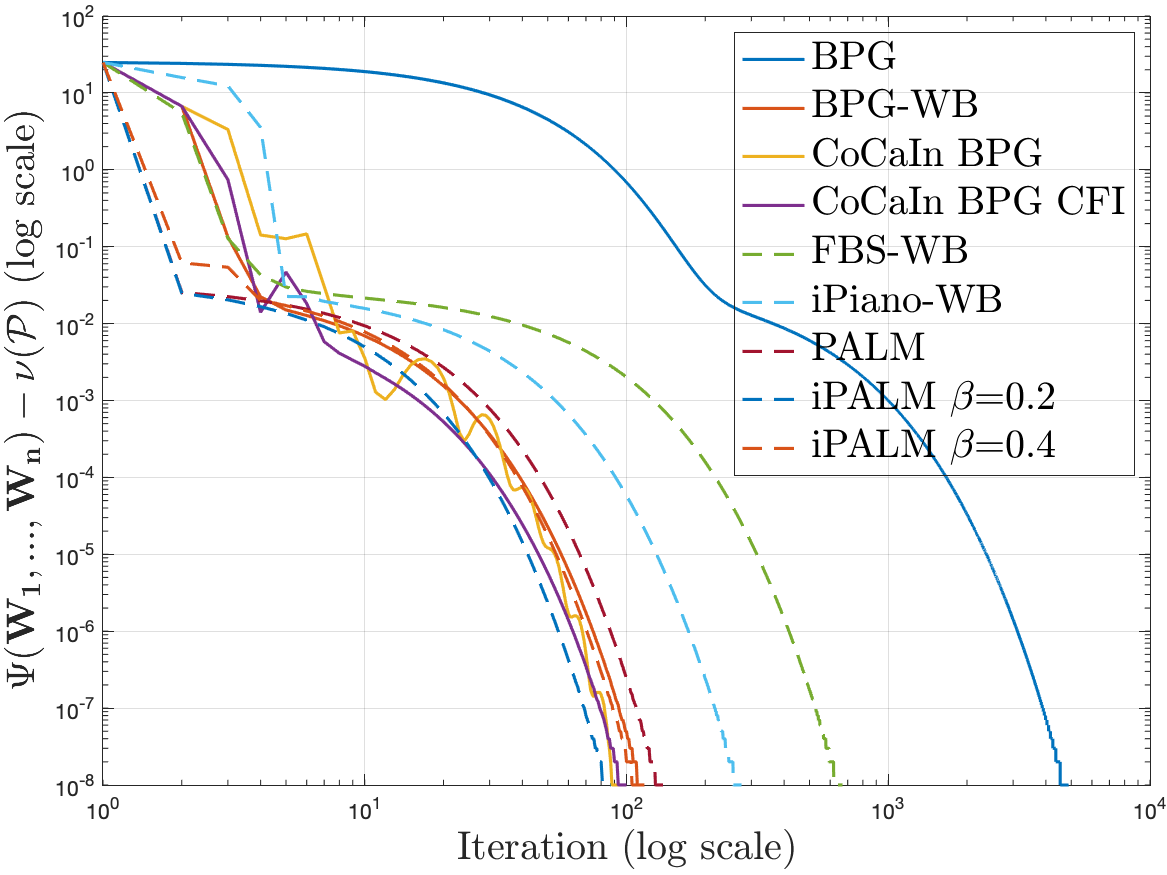} \\
    \small (c) No Regularization ($N=3$)
  \end{tabular} 
  \begin{tabular}[b]{c}
    \includegraphics[width=.3\textwidth]{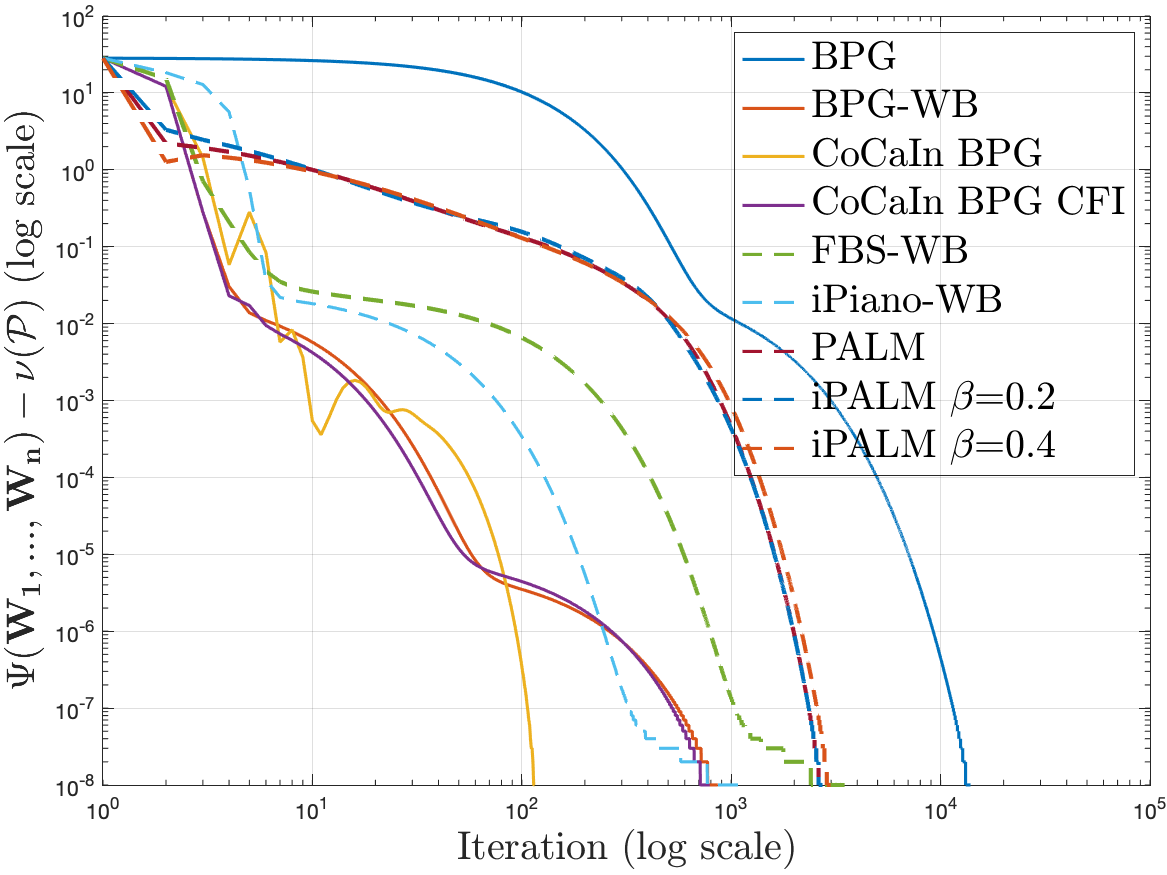} \\
    \small (d) L2-Regularization ($N=4$)
  \end{tabular}
   \begin{tabular}[b]{c}
   \includegraphics[width=.3\textwidth]{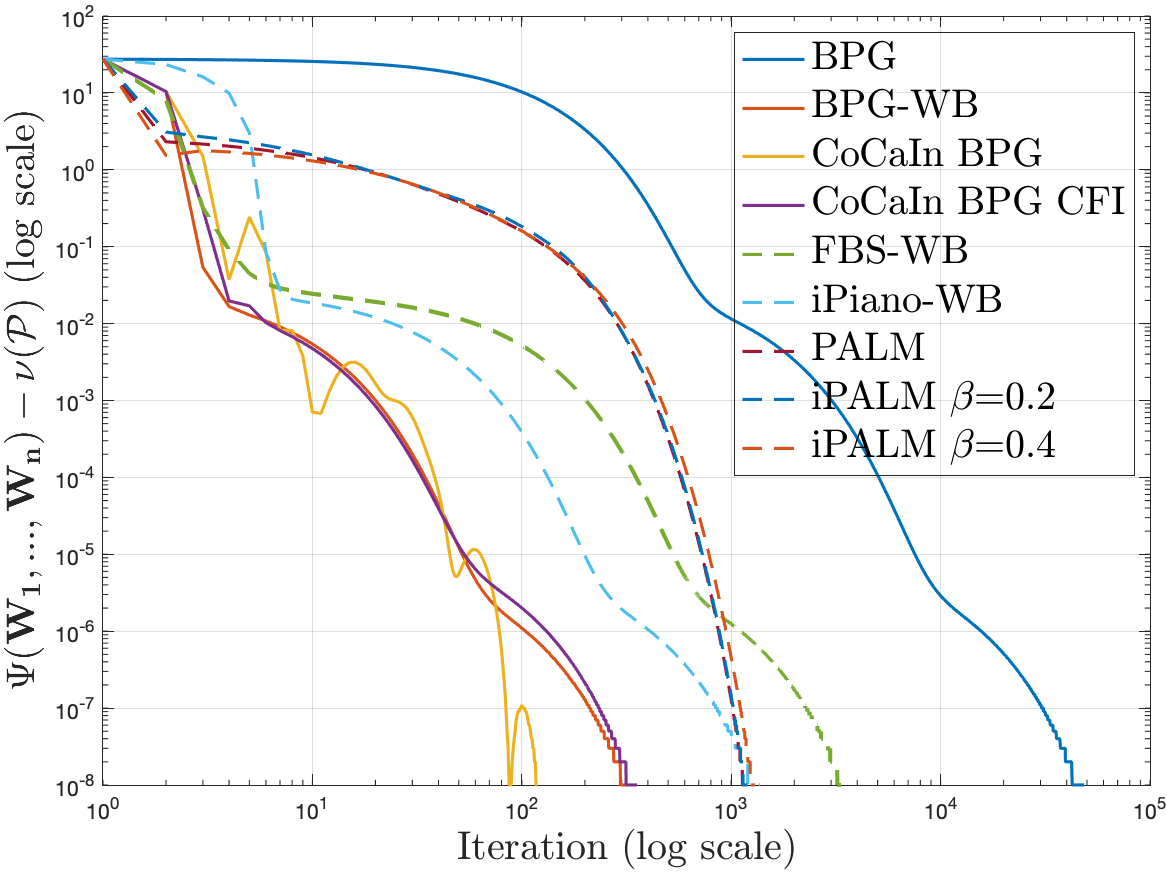} \\
    \small (e) L1-Regularization ($N=4$)
  \end{tabular}
    \begin{tabular}[b]{c}
    \includegraphics[width=.3\textwidth]{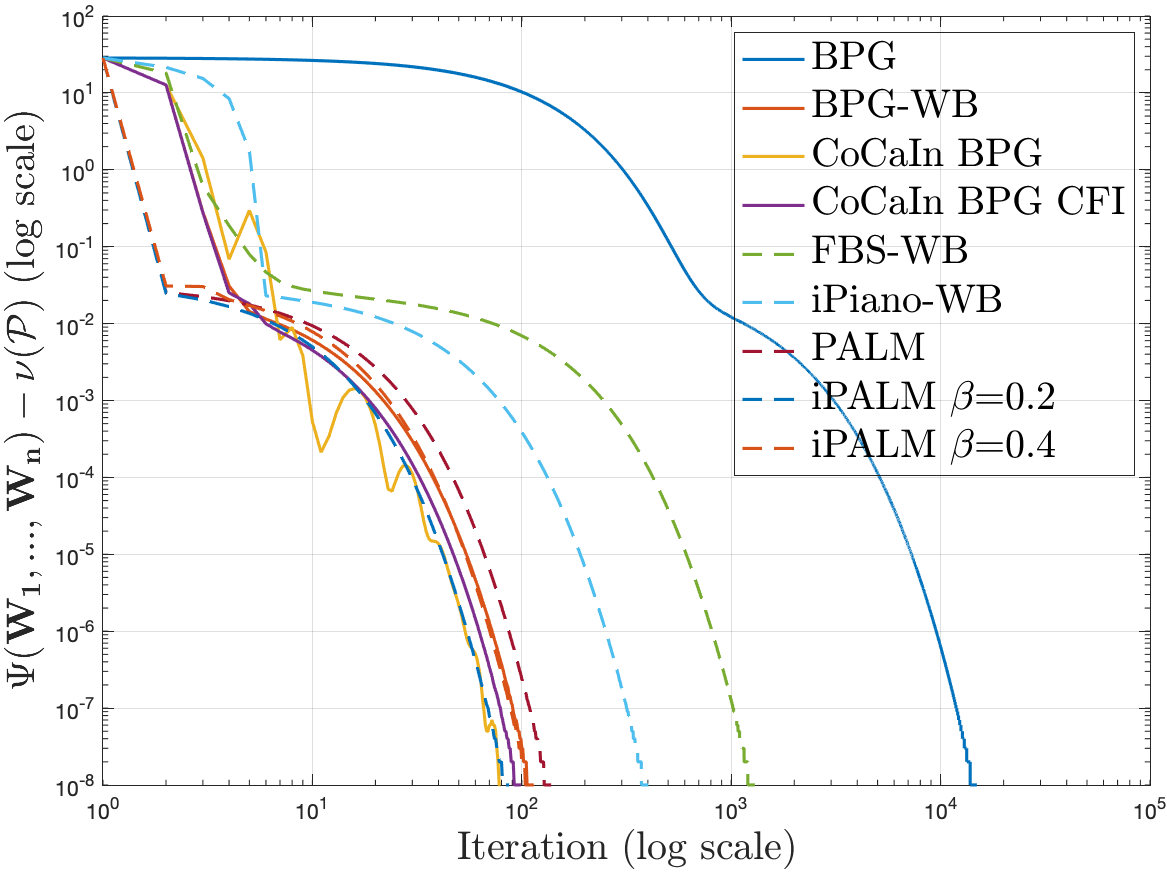} \\
    \small (f) No Regularization ($N=4$)
  \end{tabular} 
  \begin{tabular}[b]{c}
    \includegraphics[width=.3\textwidth]{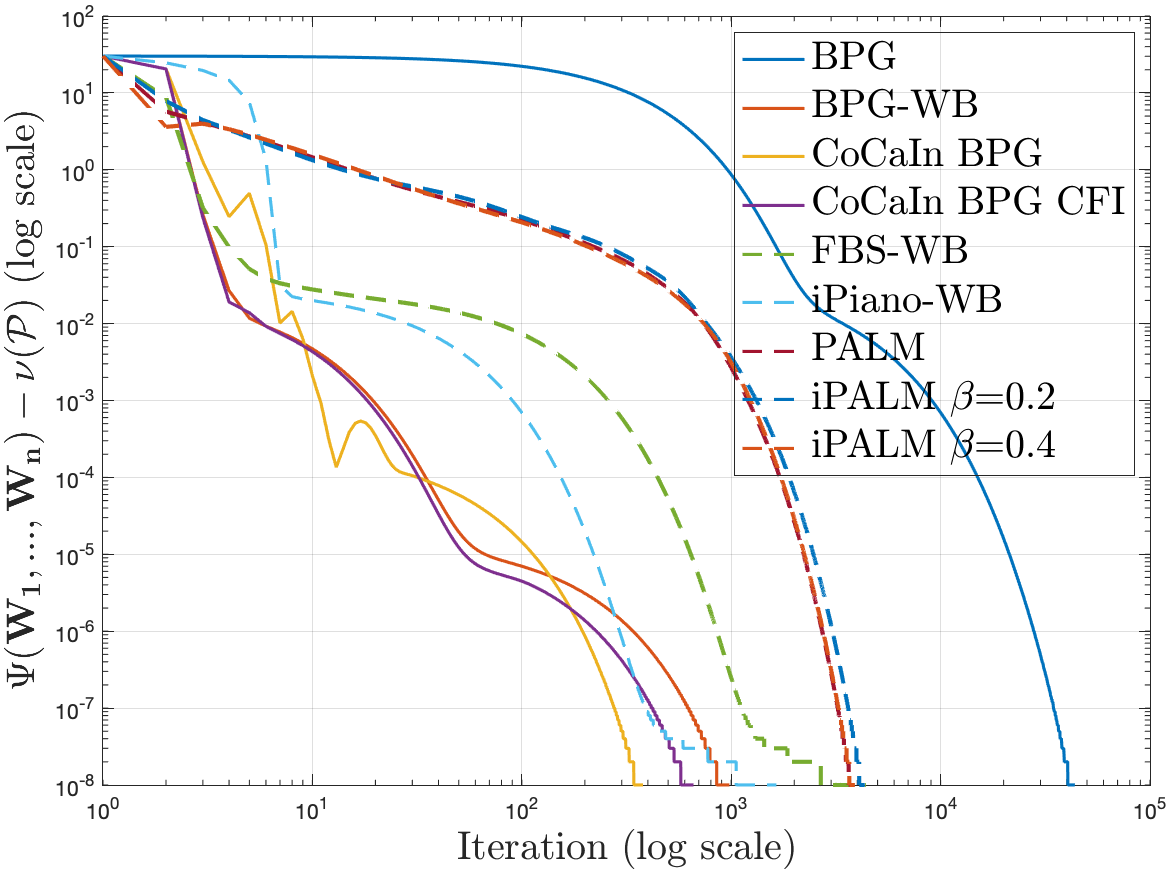} \\
    \small (g) L2-Regularization ($N=5$)
  \end{tabular}
   \begin{tabular}[b]{c}
   \includegraphics[width=.3\textwidth]{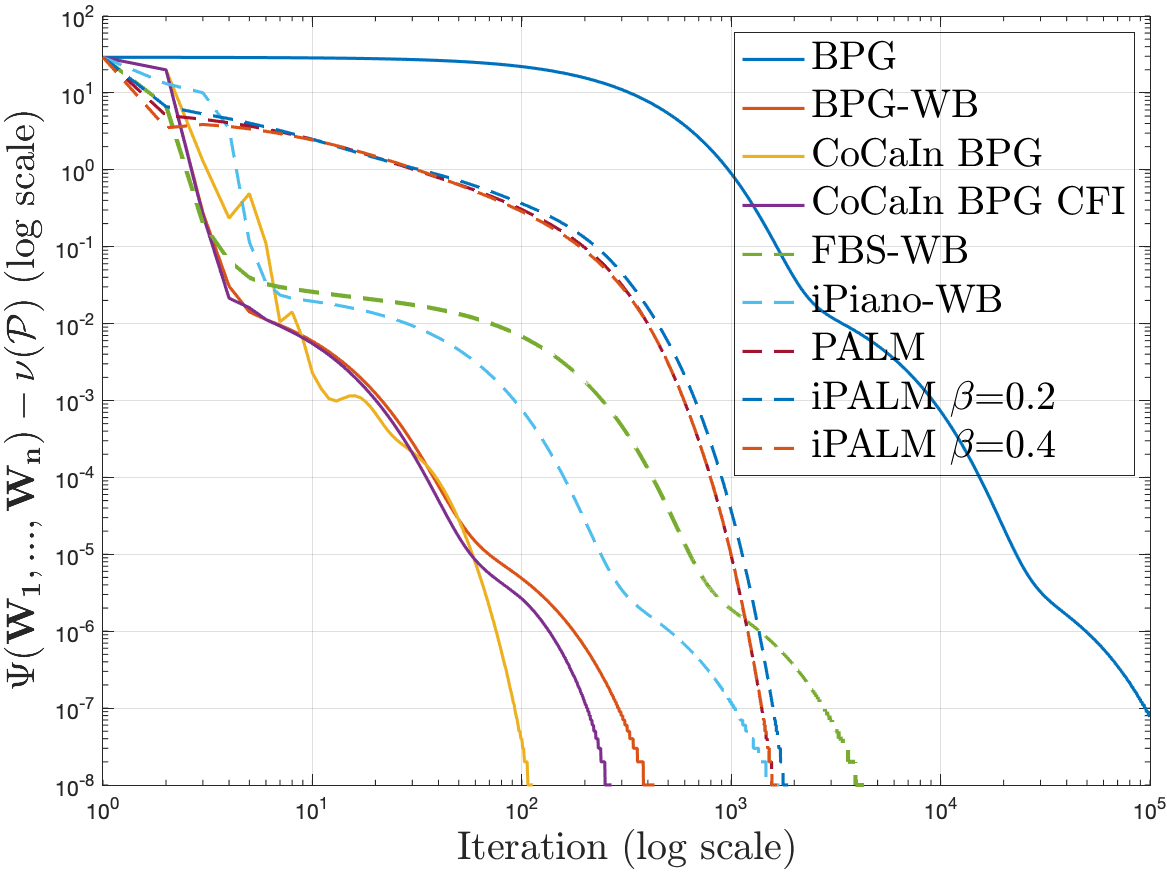} \\
    \small (h) L1-Regularization ($N=5$)
  \end{tabular}
      \begin{tabular}[b]{c}
    \includegraphics[width=.3\textwidth]{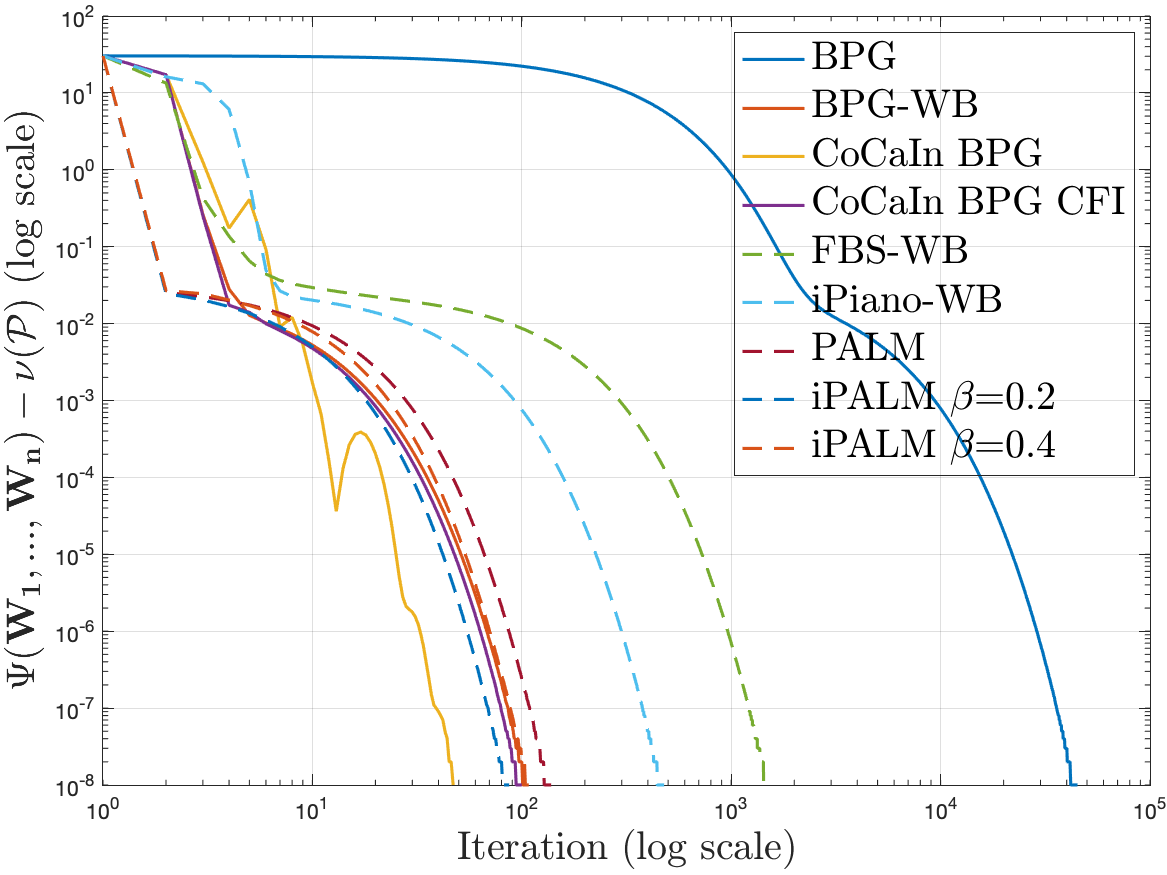} \\
    \small (i) No Regularization ($N=5$)
  \end{tabular} 
  \caption{Convergence plots illustrate the competitive performance of CoCaIn BPG variants for DLNN.}
  \label{fig:exp1_rel_obj}
\end{figure*}

We provide experiments for Deep Linear Neural Networks with squared L2-regularizer and L1-regularizers and a non-regularized setting \eqref{eq:prob-2}.\ifpaper\else\medskip\fi

\textbf{Algorithms.} In the experiments, we compare BPG (Algorithm~\ref{alg:bpg}) and CoCaIn BPG (Algorithm~\ref{alg:cocain}) with many existing optimization methods. We consider alternating strategies such as PALM \cite{BST2014} and iPALM \cite{PS2016}. As non-alternating algorithms, we use forward backward splitting  with backtracking (FBS-WB) and iPiano with backtracking (iPiano-WB) \cite{OCBP2014}.  Apart from BPG and backtracking based CoCaIn BPG, we also inspect CoCaIn BPG with closed form inertia denoted as CoCaIn BPG CFI (see Proposition~\ref{lem:even-extrapolation}) and the backtracking scheme BPG-WB, which is the same version as CoCaIn BPG, but with $\gamma_k \equiv 0$.\ifpaper\else\medskip\fi

\textbf{Experiment.} We set ${\bf W_i} \in \mathbb{R}^{5 \times 5}, ~ \forall i=1,...,N$ where all weights are initialized with $0.1$. Our dataset contains 50 data points with the input ${\bf X} \in \mathbb{R}^{5 \times 50}$ and the output ${\bf Y} \in \mathbb{R}^{5 \times 50}$ being randomly generated in the interval $[0,1]$. In this experiment, we work with a network consisting of three, four and five layers ($N=3,4,5$). The convergence plots are given in Figure \ref{fig:exp1_rel_obj}, where  the $y$-axis measures difference between the absolute objective and the least objective value attained by any of the methods. \ifpaper\else\medskip\fi

\textbf{Analysis.} The performance of CoCaIn BPG, CoCaIn BPG CFI and BPG-WB is mostly better than other methods. The next competitive algorithms include FBS-WB and iPiano-WB, followed by PALM and iPALM. The performance of the alternating algorithms strongly depends on the usage of a regularizer, whereas BPG-WB is competitive in both settings. At first glance, it might appear that the performance of BPG is weaker compared to CoCaIn BPG, BPG-WB, FBS-WB, iPiano-WB and other methods. However, note that line search techniques may not be always desirable in practical scenarios, because line search requires multiple objective evaluations, which can involve computationally expensive matrix multiplications (see Section~\ref{sec:discuss-settings-for-algs}). Moreover, PALM and iPALM require block-wise Lipschitz constant computations in each iteration, which can also be very expensive. \ifpaper\else\medskip\fi

In the appendix, we further illustrate the competitiveness of our methods with time plots, the statistical evaluation and results for an additional dataset. 
% \textbf{Further comments.} In the symmetric deep matrix factorization setting, where the factors $W_1,\ldots,W_N$ are equivalent upto a transpose, PALM and iPALM are not applicable anymore as the block-wise Lipschitz continuity of gradient fails to hold. But, it is straightforward to extend the proposed BPG based methods.  Additionally, the investigation of the nuclear norm plots in the appendix (Figure \ref{fig:exp1_nucl_norm}) raise the assumption that BPG-based methods implicitly converge to solutions with lower nuclear norm and lower Forbenius norm  than alternating schemes.
\section*{Conclusion and Extensions} 
	We proposed new Bregman distances suitable for deep linear neural networks. This result makes BPG and its inertial variant CoCaIn BPG applicable and enables the transfer of their convergence results to such problems.  Moreover, we develop update formulas, which are crucial for efficient large scale optimization. In general, the validity of inertial (or momentum) parameter requires to be checked via backtracking line search. To avoid expensive backtracking operation, we derive a analytic expression.  These contributions serve as a first step towards the optimization of deep (non-linear) neural networks by a new class of Bregman Proximal algorithms.

\ifpaper
% \newpage
{ 
\bibliographystyle{apalike}
\bibliography{notes}
}
\newpage
\else
	\section*{Acknowledgments}
	Mahesh Chandra Mukkamala and Peter Ochs acknowledge the financial support from German Research Foundation (DFG Grant OC 150/1-1).
\fi
\appendix
\ifpaper
\onecolumn
\else
\fi
% \section{Scaling Argument}
% Let the scaling factor be $\theta \neq 0$, then we have the following equivalence

% \begin{equation}
% \frac 12\norm{{\bf W_1W_2\ldots W_NX} - {\bf Y}}^2_{F} \equiv \frac 12\norm{(\theta {\bf W_1})(\theta{\bf W_2})\ldots (\theta{\bf W_N})\left(\frac{{\bf X}}{\theta^N}\right)- {\bf Y}}^2_{F}\,.
% \end{equation}

% So for initialization of ones we need to set $\theta = 0.1$ to bring the weights down to $0.1$. But, this implies we need change our ${\bf X}$ accordingly.

% \subsection{Relaxed/Penalty based Non-Linear Neural Network}
% \url{https://arxiv.org/pdf/1806.09077.pdf} \quad 
% \url{https://arxiv.org/pdf/1212.5921.pdf}
% {Penalty based approaches}
% \textcolor{red}{todo: Deep Linear Nueral Networks aswell probably resnets.}
% \subsection{ResNets: Deep Residual Neural Networks}
% \subsection{Extensions to Robust Losses}
% \section{Closed Form Solutions}\label{sec:closed-forms}
\section*{Appendix} 
\section{Bregman distance and $L$-smad property}

% We consider the following problem
% \begin{equation}
% 	 \min_{{\bf W_i}\in \mathcal W_i\, \forall i\in [N]}\, \left\{ \Psi({\bf W_1},\ldots,{\bf W_N}) :=  \frac 12\norm{{\bf W_1W_2\ldots W_NX} - {\bf Y}}^2_{F} \right\}\,.
% 	\end{equation}
\begin{proposition}\label{prop:main-prop}
Denote $g({\bf W_1},\ldots,{\bf W_N}) :=  \frac 12\norm{{\bf W_1W_2\ldots W_NX} - {\bf Y}}^2_{F}$ as in the setting of \eqref{eq:prob-2}. Then the gradient with respect to weights ${\bf W_i}$ is given by
\[
	\nabla_{{\bf W_i}} g({\bf W_1},\ldots,{\bf W_N}) = \left(\Pi_{j=1}^{i-1}{\bf W_j}\right)^T\left({\bf W_1W_2\ldots W_NX} - {\bf Y}\right)\left(\left(\Pi_{j=i+1}^N{\bf W_j}\right) {\bf X}\right)^T\,.
\]
We have for $N = 2$,
\begin{align*}
& \act{({\bf H}_{{\bf 1}},\ldots,{\bf H}_{{\bf N}}), \nabla^2 g({\bf W_1},\ldots,{\bf W_N}) ({\bf H}_{{\bf 1}},\ldots, {\bf H}_{{\bf N}})}\\
	& \leq 3\norm{\bf X}_F^2 \sum_{i=1}^N\norm{{\bf H_i}}_F^2\Pi_{j=1, j\neq i}^N \norm{{\bf W_j}}_F^2   + \norm{{\bf Y}}_F\norm{\bf X}_F \left( \norm{{\bf H}_1}_F^2 + \norm{{\bf H}_2}_F^2 \right)
\end{align*}
If $N>2$ and even, we have
\begin{align*}
	& \act{({\bf H}_{{\bf 1}},\ldots,{\bf H}_{{\bf N}}), \nabla^2 g({\bf W_1},\ldots,{\bf W_N}) ({\bf H}_{{\bf 1}},\ldots, {\bf H}_{{\bf N}})}\\
	& \leq (2N-1)\sum_{i=1}^N\norm{{\bf H_i}}_F^2\Pi_{j=1, j\neq i}^N \norm{{\bf W_j}}_F^2\norm{\bf X}_F^2   + \frac{\norm{{\bf Y}}_F\norm{\bf X}_F (N-1)}{ (N-2)^{\frac{N-2}{2}}} \left( \sum_{i=1}^{N}  \norm{{\bf H}_i}_F^2\right) \left(\sum_{k=1}^N\norm{\bf W_k}_F^2 \right)^{\frac{N-2}{2}}
\end{align*}
If $N>2$ and odd, we have
\begin{align*}
	& \act{({\bf H}_{{\bf 1}},\ldots,{\bf H}_{{\bf N}}), \nabla^2 g({\bf W_1},\ldots,{\bf W_N}) ({\bf H}_{{\bf 1}},\ldots, {\bf H}_{{\bf N}})}\\
	& \leq (2N-1)\sum_{i=1}^N\norm{{\bf H_i}}_F^2\Pi_{j=1, j\neq i}^N \norm{{\bf W_j}}_F^2\norm{\bf X}_F^2   + \frac{\norm{{\bf Y}}_F\norm{\bf X}_F(N-1)}{ (N-1)^{\frac{N-1}{2}}} \left(\sum_{i=1}^{N} \norm{{\bf H}_i}_F^2 \right)\left(\left(\sum_{k=1, k \notin \{i,j\}}^N\norm{\bf W_k}_F^2\right) +1 \right)^{\frac{N-1}{2}}
\end{align*}
\end{proposition}
\begin{proof}
Consider the following
\begin{equation}\label{eq:main-exp-1}
	  \frac 12\norm{{\bf (W_1+H_1)(W_2+H_2)\ldots (W_N+H_N)X} - {\bf Y}}^2_{F} \,.
	\end{equation}
We are only interested in terms till second order, thus we have
\begin{align*}
  {\bf (W_1+H_1)(W_2+H_2)\ldots (W_N+H_N)X} &= {\bf W_1W_2\ldots W_NX} + \sum_{i=1}^N \left(\Pi_{j=1}^{i-1}{\bf W_j}\right){\bf H_i} \left(\Pi_{j=i+1}^N{\bf W_j}{\bf X}\right) \\
	& +  \sum_{i=1}^{N-1} \sum_{j>i}^{N}\left(\Pi_{k=1}^{i-1}{\bf W_k}\right){\bf H_i} \left(\Pi_{k=i+1}^{j-1}{\bf W_k}\right){\bf H_j}\left(\Pi_{k=j+1}^{N}{\bf W_k}{\bf X}\right)\,.
\end{align*}
Now expanding \eqref{eq:main-exp-1}, we have terms upto second order as following
\begin{align*}
&\frac 12\norm{{\bf W_1W_2\ldots W_NX} - {\bf Y}}_F^2 + \act{{\bf W_1W_2\ldots W_NX} - {\bf Y},\sum_{i=1}^N \left(\Pi_{j=1}^{i-1}{\bf W_j}\right){\bf H_i} \left(\Pi_{j=i+1}^N{\bf W_j}\right){\bf X}}\\
& + \frac 12\norm{\sum_{i=1}^N \left(\Pi_{j=1}^{i-1}{\bf W_j}\right){\bf H_i} \left(\Pi_{j=i+1}^N{\bf W_j}\right){\bf X}}_F^2  - \act{{\bf Y}, \sum_{i=1}^{N-1} \sum_{j>i}^{N}\left(\Pi_{k=1}^{i-1}{\bf W_k}\right){\bf H_i} \left(\Pi_{k=i+1}^{j-1}{\bf W_k}\right){\bf H_j}\left(\Pi_{k=j+1}^{N}{\bf W_k}\right){\bf X}}\\
& + \act{{\bf W_1W_2\ldots W_NX},\sum_{i=1}^{N-1} \sum_{j>i}^{N}\left(\Pi_{k=1}^{i-1}{\bf W_k}\right){\bf H_i} \left(\Pi_{k=i+1}^{j-1}{\bf W_k}\right){\bf H_j}\left(\Pi_{k=j+1}^{N}{\bf W_k}\right){\bf X}}\,.
\end{align*}
Consider the first order terms, we have
\begin{align*}
& \act{{\bf W_1W_2\ldots W_NX} - {\bf Y},\sum_{i=1}^N \left(\Pi_{j=1}^{i-1}{\bf W_j}\right){\bf H_i} \left(\Pi_{j=i+1}^N{\bf W_j}\right){\bf X}}\\
& = \sum_{i=1}^N \act{{\bf W_1W_2\ldots W_NX} - {\bf Y}, \left(\Pi_{j=1}^{i-1}{\bf W_j}\right){\bf H_i} \left(\Pi_{j=i+1}^N{\bf W_j}\right){\bf X}}\,,
\end{align*}
thus, the gradient is 
\[
	\nabla_{{\bf W_i}} g({\bf W_1},\ldots,{\bf W_N}) = \left(\Pi_{j=1}^{i-1}{\bf W_j}\right)^T\left({\bf W_1W_2\ldots W_NX} - {\bf Y}\right)\left(\left(\Pi_{j=i+1}^N{\bf W_j} \right){\bf X}\right)^T\,.
\]
Now, considering second order terms we have with repetitive application of Cauchy-Schwarz inequality, the following
\begin{align*}
\frac 12\norm{\sum_{i=1}^N \left(\Pi_{j=1}^{i-1}{\bf W_j}\right){\bf H_i} \left(\Pi_{j=i+1}^N{\bf W_j}\right){\bf X}}_F^2 &\leq \frac{N}{2}\sum_{i=1}^N\norm{\left(\Pi_{j=1}^{i-1}{\bf W_j}\right){\bf H_i} \left(\Pi_{j=i+1}^N{\bf W_j}\right){\bf X}}_F^2\\
 &\leq \frac{N}{2}\sum_{i=1}^N\norm{{\bf H_i}}_F^2\Pi_{j=1, j\neq i}^N \norm{{\bf W_j}}_F^2\norm{\bf X}_F^2 
\end{align*}
and 
\begin{align*}
&\act{{\bf W_1W_2\ldots W_NX},\sum_{i=1}^{N-1} \sum_{j>i}^{N}\left(\Pi_{k=1}^{i-1}{\bf W_k}\right){\bf H_i} \left(\Pi_{k=i+1}^{j-1}{\bf W_k}\right){\bf H_j}\left(\Pi_{k=j+1}^{N}{\bf W_k}\right){\bf X}}\\
&\leq  \sum_{i=1}^{N-1} \sum_{j>i}^{N} \norm{{\bf X}}_F^2\norm{{\bf H}_i}_F \norm{{\bf H}_j}_F \norm{{\bf W_i}}_F\norm{{\bf W_j}}_F \Pi_{k=1, k \notin \{i,j\}}^N\norm{\bf W_k}_F^2\\
&\leq  \sum_{i=1}^{N-1} \sum_{j>i}^{N}  \norm{{\bf X}}_F^2 \left( \frac{\norm{{\bf H}_i}_F^2 \norm{{\bf W_j}}_F^2 + \norm{{\bf H}_j}_F^2 \norm{{\bf W_i}}_F^2}{2} \right)  \Pi_{k=1, k \notin \{i,j\}}^N\norm{\bf W_k}_F^2\\
&\leq \norm{{\bf X}}_F^2 \left(\frac{N-1}{2}\right)\sum_{i=1}^{N}\norm{{\bf H}_i}_F^2\Pi_{k=1, k \notin \{i\}}^N\norm{\bf W_k}_F^2
\end{align*}

and we have
\begin{align}
&-\act{{\bf Y}, \sum_{i=1}^{N-1} \sum_{j>i}^{N}\left(\Pi_{k=1}^{i-1}{\bf W_k}\right){\bf H_i} \left(\Pi_{k=i+1}^{j-1}{\bf W_k}\right){\bf H_j}\left(\Pi_{k=j+1}^{N}{\bf W_k}\right){\bf X}}\nonumber\\
&\leq \norm{{\bf Y}}_F\sum_{i=1}^{N-1} \sum_{j>i}^{N} \norm{{\bf H}_i}_F\norm{{\bf H}_j}_F\Pi_{k=1, k \notin \{i,j\}}^N\norm{\bf W_k}_F\norm{\bf X}_F\label{eq:temp-1}
\end{align}
Now with the application of Generalized AM-GM inequality, we have the following three cases:
\begin{itemize}
\item  When $N=2$ then we have
\begin{align*}
\norm{{\bf H}_i}_F\norm{{\bf H}_j}_F\norm{\bf X}_F \leq \norm{\bf X}_F\left(  \frac{\norm{{\bf H}_j}_F^2 + \norm{{\bf H}_i}_F^2}{2} \right)\,,
\end{align*}
\item When $N$ is even and  $N>2$. 
\begin{align*}
&\norm{{\bf H}_i}_F\norm{{\bf H}_j}_F\Pi_{k=1, k \notin \{i,j\}}^N\norm{\bf W_k}_F\norm{\bf X}_F \leq \norm{\bf X}_F\left(  \frac{\norm{{\bf H}_j}_F^2 + \norm{{\bf H}_i}_F^2}{2} \right)\left(\frac{\sum_{k=1, k \notin \{i,j\}}^N\norm{\bf W_k}_F^2}{N-2} \right)^{\frac{N-2}{2}}\,,
\end{align*}
\item If $N$ is odd and $N>2$ we have
\begin{align*}
&\norm{{\bf H}_i}_F\norm{{\bf H}_j}_F\Pi_{k=1, k \notin \{i,j\}}^N\norm{\bf W_k}_F\norm{\bf X}_F  \leq \norm{\bf X}_F\left(  \frac{\norm{{\bf H}_j}_F^2 + \norm{{\bf H}_i}_F^2}{2} \right)\left(\frac{\left(\sum_{k=1, k \notin \{i,j\}}^N\norm{\bf W_k}_F^2 \right) + 1}{N-1} \right)^{\frac{N-1}{2}}\,.
\end{align*}
\end{itemize}
Now using the above given results, on extending the calculation of \eqref{eq:temp-1}, for even $N$ and $N\geq 2$, we have
\begin{align*}
& \norm{{\bf Y}}_F\sum_{i=1}^{N-1} \sum_{j>i}^{N} \norm{{\bf H}_i}_F\norm{{\bf H}_j}_F\Pi_{k=1, k \notin \{i,j\}}^N\norm{\bf W_k}_F\norm{\bf X}_F\\
&\leq \norm{{\bf Y}}_F\norm{\bf X}_F\sum_{i=1}^{N-1} \sum_{j>i}^{N} \left(  \frac{\norm{{\bf H}_j}_F^2 + \norm{{\bf H}_i}_F^2}{2} \right)\left(\frac{\sum_{k=1, k \notin \{i,j\}}^N\norm{\bf W_k}_F^2}{N-2} \right)^{\frac{N-2}{2}}\\
&\leq  \frac{\norm{{\bf Y}}_F\norm{\bf X}_F (N-1)}{ 2(N-2)^{\frac{N-2}{2}}} \left( \sum_{i=1}^{N}  \norm{{\bf H}_i}_F^2\right) \left(\sum_{k=1}^N\norm{\bf W_k}_F^2 \right)^{\frac{N-2}{2}}\,,
\end{align*}
where in the first step we used Cauchy-Schwarz inequality.
Similarly, we have for $N>2$ and odd, 
\begin{align*}
& \norm{{\bf Y}}_F\sum_{i=1}^{N-1} \sum_{j>i}^{N} \norm{{\bf H}_i}_F\norm{{\bf H}_j}_F\Pi_{k=1, k \notin \{i,j\}}^N\norm{\bf W_k}_F\norm{\bf X}_F\\
&\leq \norm{{\bf Y}}_F\norm{\bf X}_F\sum_{i=1}^{N-1} \sum_{j>i}^{N} \left(  \frac{\norm{{\bf H}_j}_F^2 + \norm{{\bf H}_i}_F^2}{2} \right)\left(\frac{\left(\sum_{k=1, k \notin \{i,j\}}^N\norm{\bf W_k}_F^2\right) + 1}{N-1}  \right)^{\frac{N-1}{2}}\\
&\leq  \frac{\norm{{\bf Y}}_F\norm{\bf X}_F(N-1)}{ 2(N-1)^{\frac{N-1}{2}}} \left(\sum_{i=1}^{N} \norm{{\bf H}_i}_F^2 \right)\left(\left(\sum_{k=1}^N\norm{\bf W_k}_F^2\right) +1 \right)^{\frac{N-1}{2}}\,.
\end{align*}\qedhere
\end{proof}
Before we start with the proof of  Proposition~\ref{prop:even-case}, we require the following technical results.
\subsection{Results for $H_1$.}
\begin{lemma}\label{lem:basic-ftc} Let $h \in \mathcal{G}(C)$ be twice continuously differentiable on $C$. Then, the following identity holds
\[
	D_h(x^k,y^k) = \int_{0}^1 \left(1-t\right)\int_{0}^1 \act{ \nabla^2h\left(x^k + (t_1+ (1-t_1)t)(y^k-x^k)\right)(x^k-y^k), x^k-y^k}dt_1dt\,.
\]
\end{lemma}
\begin{proof}
With repetitive application of fundamental theorem of calculus we have
\begin{align*}
&h(x^k) - h(y^k) - \act{\nabla h(y^k), x^k - y^k}\\
&= \int_{0}^1 \act{\nabla h(x^k+t(y^k-x^k)) - \nabla h(y^k), x^k-y^k}dt\,,\\
&= \int_{0}^1  \act{\int_{0}^1 \nabla^2h\left((1-t_1)(x^k+t(y^k-x^k)) + t_1 y^k\right)\left(1-t\right)(x^k-y^k) dt_1, x^k-y^k}dt\,,\\
&= \int_{0}^1  \act{\int_{0}^1 \nabla^2h\left(x^k + (t_1+ (1-t_1)t)(y^k-x^k)\right)\left(1-t\right)(x^k-y^k)dt_1, x^k-y^k}dt\,,\\
&= \int_{0}^1 \left(1-t\right)\int_{0}^1 \act{ \nabla^2h\left(x^k + (t_1+ (1-t_1)t)(y^k-x^k)\right)(x^k-y^k), x^k-y^k}dt_1dt\,.
\end{align*}\qedhere
\end{proof}

Henceforth, we use the following notation. Let $n$ be a positive integer and let $k_i$ be a non-negative integer for $i \in [m]$ satisfying $k_1+\ldots+k_m = n$, then we denote 
	\[
	\binom{n}{k_1,k_2,\ldots,k_m} := \frac{n!}{k_1!k_2!\ldots k_m!}\,,
	\]
	which is also known as multinomial coefficient. 
\begin{lemma}\label{lem:basic-h1}
With the following kernel generating distance
\[
	H_1({\bf W_1},\ldots,{\bf W_N}) = \left(\frac{\norm{{\bf W_1}}_F^2  + \ldots \norm{{\bf W_N}}_F^2}{N}\right)^N\,,
\]
the gradient with respect for ${\bf W_i}$, for any $i\in [N]$, is given by 
\begin{align*}
\nabla_{{\bf W_i}}H_1({\bf W_1},\ldots,{\bf W_N}) = \frac{2}{N^N}\binom{N}{N-1,1}\left(\norm{{\bf W_1 }}_F^2 + \ldots + \norm{{\bf W_N }}_F^2\right)^{N-1}{\bf W_i }\,,
\end{align*}
and the following lower bound holds true
\begin{align*}
\act{({\bf H}_{{\bf 1}},\ldots,{\bf H}_{{\bf N}}), \nabla^2 H_1({\bf W_1},\ldots,{\bf W_N}) ({\bf H}_{{\bf 1}},\ldots, {\bf H}_{{\bf N}})} \geq \frac{2N!}{N^N} \sum_{i=1}^N\norm{{\bf H}_i}_F^2\Pi_{k=1, k \notin \{i\}}^N\norm{\bf W_k}_F^2\,,
\end{align*}
and the following upper bound holds true
\begin{align*}
\act{({\bf H}_{{\bf 1}},\ldots,{\bf H}_{{\bf N}}), \nabla^2 H_1({\bf W_1},\ldots,{\bf W_N}) ({\bf H}_{{\bf 1}},\ldots, {\bf H}_{{\bf N}})} \leq \left(\frac{2(2N-1)}{N^{N-1}}\right)\left( \sum_{k=1}^{N}\norm{{\bf H_k}}_F^2 \right)\left(\sum_{k=1}^N\norm{\bf W_k}_F^2\right)^{N-1}\,.
\end{align*}
\end{lemma}
\begin{proof}
Consider the following 
\begin{align*}
&\left(\frac{\norm{{\bf W_1 }+ {\bf H_1}}_F^2  + \ldots \norm{{\bf W_N}+{\bf H_N}}_F^2}{N}\right)^N= \left(\frac{\norm{{\bf W_1 }}_F^2+ \norm{{\bf H_1}}_F^2 + 2\act{{\bf W_1 },{\bf H_1}}  + \ldots \norm{{\bf W_N}}_F^2+\norm{{\bf H_N}}_F^2}{N}\right)^N\,.
\end{align*}
Consider only the first order terms in the expansion, from which the following gradient with respect for ${\bf W_i}$, for any $i\in [N]$, is obtained 
\begin{align*}
\nabla_{{\bf W_i}}H_1({\bf W_1},\ldots,{\bf W_N}) = \frac{2}{N^N}\binom{N}{N-1,1}\left(\norm{{\bf W_1 }}_F^2 + \ldots + \norm{{\bf W_N }}_F^2\right)^{N-1}{\bf W_i }\,.
\end{align*}
Now considering only the second order terms, we have
\begin{align*}
& \frac{1}{2}\act{({\bf H}_{{\bf 1}},\ldots,{\bf H}_{{\bf N}}), \nabla^2 H_1({\bf W_1},\ldots,{\bf W_N}) ({\bf H}_{{\bf 1}},\ldots, {\bf H}_{{\bf N}})}\\
& = \frac{1}{2}\frac{2}{N^N} \sum_{i=1}^N \binom{N}{1,N-1}\norm{{\bf H}_i}_F^2\left(\sum_{k=1}^N\norm{\bf W_k}_F^2\right)^{N-1} \\
& + \frac{1}{2}\frac{2^3}{N^N} \binom{N}{2,N-2} \left(\act{{\bf W_1 },{\bf H_1}} +\ldots + \act{{\bf W_N },{\bf H_N}}\right)^2\left(\sum_{k=1}^N\norm{\bf W_k}_F^2\right)^{N-2}\,.
\end{align*}
Since, the second term in the right hand side is always non-negative, the following result holds
\begin{align*}
\frac{1}{2}\act{({\bf H}_{{\bf 1}},\ldots,{\bf H}_{{\bf N}}), \nabla^2 H_1({\bf W_1},\ldots,{\bf W_N}) ({\bf H}_{{\bf 1}},\ldots, {\bf H}_{{\bf N}})} \geq \frac{1}{2}\frac{2N!}{N^N} \sum_{i=1}^N\norm{{\bf H}_i}_F^2\Pi_{k=1, k \notin \{i\}}^N\norm{\bf W_k}_F^2\,.
\end{align*}
This proves the lower bound. Now, we prove the upper bound. With  application of Cauchy-Schwarz inequality,  we have
\begin{align*}
& \frac{1}{2}\frac{2^3}{N^N} \binom{N}{2,N-2} \left(\act{{\bf W_1 },{\bf H_1}} +\ldots + \act{{\bf W_N },{\bf H_N}}\right)^2\left(\sum_{k=1}^N\norm{\bf W_k}_F^2\right)^{N-2}\\
& \leq \frac{1}{2}\frac{2^3}{N^N} \binom{N}{2,N-2} \left(\sum_{k=1}^{N}\norm{{\bf W_k}}_F^2\right)\left( \sum_{k=1}^{N}\norm{{\bf H_k}}_F^2 \right)\left(\sum_{k=1}^N\norm{\bf W_k}_F^2\right)^{N-2}\\
& = \frac{1}{2}\frac{2^3}{N^N} \binom{N}{2,N-2} \left( \sum_{k=1}^{N}\norm{{\bf H_k}}_F^2 \right)\left(\sum_{k=1}^N\norm{\bf W_k}_F^2\right)^{N-1}\,.
\end{align*}
Now we finally have
\begin{align*}
\act{({\bf H}_{{\bf 1}},\ldots,{\bf H}_{{\bf N}}), \nabla^2 H_1({\bf W_1},\ldots,{\bf W_N}) ({\bf H}_{{\bf 1}},\ldots, {\bf H}_{{\bf N}})} &\leq  \frac{2}{N^N}\binom{N}{1,N-1} \left(\sum_{i=1}^N \norm{{\bf H}_i}_F^2\right)\left(\sum_{k=1}^N\norm{\bf W_k}_F^2\right)^{N-1} \\
& + \frac{2^3}{N^N} \binom{N}{2,N-2} \left( \sum_{k=1}^{N}\norm{{\bf H_k}}_F^2 \right)\left(\sum_{k=1}^N\norm{\bf W_k}_F^2\right)^{N-1}\\
&= \left(\frac{2(2N-1)}{N^{N-1}}\right)\left( \sum_{k=1}^{N}\norm{{\bf H_k}}_F^2 \right)\left(\sum_{k=1}^N\norm{\bf W_k}_F^2\right)^{N-1}\,.\qedhere
\end{align*}
\end{proof}
\begin{lemma}\label{lem:h1-inert}
Denote for any $k\geq 1$,  $x^k = ({\bf W_1^k},\ldots,{\bf W_N^k})$, $\Delta_k := {x^k - x^{k-1}}$ and the following 
\begin{align*}
\mathcal{B}_k := &\left(\frac{(2N-1)}{N^{N-1}}\right) \norm{\Delta_k}^2 \left(2\norm{x^k}^2 + 2\norm{\Delta_k}^2 \right)^{(N-1)} \,.
\end{align*}
The following upper bound holds true
\[
	D_{H_1}(x^k,y^k) \leq \gamma_k^2\mathcal{B}_k\,.
\]
\end{lemma}
\begin{proof}
From Lemma~\ref{lem:basic-ftc}, we have
\begin{align*}
&\int_{0}^1 \left(1-t\right)\int_{0}^1 \act{ \nabla^2H_1\left(x^k + (t_1+ (1-t_1)t)(y^k-x^k)\right)(x^k-y^k), x^k-y^k}dt_1dt\\
= & \gamma_k^2\int_{0}^1 \left(1-t\right)\int_{0}^1 \act{ \nabla^2H_1\left(x^k + (t_1+ (1-t_1)t)(y^k-x^k)\right)(x^k-x^{k-1}), x^k-x^{k-1}}dt_1dt \,,\\
\leq& \gamma_k^2\int_{0}^1 \left(1-t\right)\int_{0}^1 \frac{2(2N-1)}{N^{N-1}} \norm{x^k-x^{k-1}}^2 \norm{x^k + (t_1+ (1-t_1)t)(y^k-x^k)}^{(2N-2)}  dt_1dt\,,
\end{align*}
where in the last step we used the upper bound from Lemma~\ref{lem:basic-h1}. Using the following inequality
\begin{align*}
	\norm{x^k + (t_1+ (1-t_1)t)(y^k-x^k)}^{2} &\leq 2\norm{x^k}^2 +  2(t_1+ (1-t_1)t)^{2}\gamma_k^2\norm{x^k - x^{k-1}}^2\leq 2\norm{x^k}^2 +  2\norm{x^k - x^{k-1}}^2
\end{align*}
where in the last step we used $\gamma_k^2 \leq 1$ and $(t_1+ (1-t_1)t)^{2} \leq 1$. With $\int_{0}^1(1-t)dt = \frac{1}{2}$ the result follows.\qedhere
\end{proof}
\subsection{Results for $H_2$.}
\begin{lemma}\label{lem:basic-h2}
With the following kernel generating distance
\[
	H_2({\bf W_1},\ldots,{\bf W_N}) = \left(\frac{\norm{{\bf W_1}}_F^2 + \norm{{\bf W_2}}_F^2 + \ldots \norm{{\bf W_N}}_F^2}{N}\right)^{\frac{N}{2}}\,,
\]
the gradient with respect for ${\bf W_i}$, for any $i\in [N]$, is given by 
% \begin{align*}
% \nabla_{{\bf W_i}}H_2({\bf W_1},\ldots,{\bf W_N}) = \frac{2}{N^{\frac{N}{2}}}\binom{\frac{N}{2}}{{\frac{N}{2}}-1,1}\left(\norm{{\bf W_1 }}_F^2 + \ldots + \norm{{\bf W_N }}_F^2\right)^{{\frac{N}{2}}-1}{\bf W_i }
% \end{align*}
\begin{align*}
\nabla_{{\bf W_i}}H_2({\bf W_1},\ldots,{\bf W_N}) = \frac{1}{N^{\frac{N}{2}-1}}\left(\norm{{\bf W_1 }}_F^2 + \ldots + \norm{{\bf W_N }}_F^2\right)^{{\frac{N}{2}}-1}{\bf W_i }\,,
\end{align*}
and the following lower bound holds true
\begin{align*}
	& \act{({\bf H}_{{\bf 1}},\ldots,{\bf H}_{{\bf N}}), \nabla^2 H_2({\bf W_1},\ldots,{\bf W_N}) ({\bf H}_{{\bf 1}},\ldots, {\bf H}_{{\bf N}})}\geq \frac{1}{N^{\frac{N}{2}-1}}  \left( \norm{{\bf H}_1}_F^2 +\ldots+ \norm{{\bf H}_N}_F^2 \right)\left(\sum_{k=1}^N\norm{\bf W_k}_F^2 \right)^{\frac{N-2}{2}}\,,
\end{align*}
and the following upper bound holds true
\begin{align*}
& \act{({\bf H}_{{\bf 1}},\ldots,{\bf H}_{{\bf N}}), \nabla^2 H_2({\bf W_1},\ldots,{\bf W_N}) ({\bf H}_{{\bf 1}},\ldots, {\bf H}_{{\bf N}})} \leq \left(\frac{N-1}{N^{\frac{N}{2} -1}}\right)\left(\sum_{k=1}^N\norm{{\bf H_k}}_F^2\right)\left(\sum_{k=1}^N\norm{\bf W_k}_F^2\right)^{\frac{N-2}{2}}\,.
\end{align*}
\end{lemma}
\begin{proof}
Consider the following expansion
\begin{align*}
&\left(\frac{\norm{{\bf W_1 }+ {\bf H_1}}_F^2 + \ldots \norm{{\bf W_N}+{\bf H_N}}_F^2}{N}\right)^{\frac{N}{2}}= \left(\frac{\norm{{\bf W_1 }}_F^2+ \norm{{\bf H_1}}_F^2 + 2\act{{\bf W_1 },{\bf H_1}}  + \ldots \norm{{\bf W_N}}_F^2+\norm{{\bf H_N}}_F^2}{N}\right)^{\frac{N}{2}}\,.
\end{align*}
Consider only the first order terms in the expansion, from which the following gradient with respect for ${\bf W_i}$, for any $i\in [N]$, is obtained 
\begin{align*}
\nabla_{{\bf W_i}}H_2({\bf W_1},\ldots,{\bf W_N}) = \frac{2}{N^{\frac{N}{2}}}\binom{\frac{N}{2}}{{\frac{N}{2}}-1,1}\left(\norm{{\bf W_1 }}_F^2 + \ldots + \norm{{\bf W_N }}_F^2\right)^{{\frac{N}{2}}-1}{\bf W_i }\,.
\end{align*}
Now considering only the second order terms, we have
\begin{align*}
	&\frac{1}{2} \act{({\bf H}_{{\bf 1}},\ldots,{\bf H}_{{\bf N}}), \nabla^2 H_2({\bf W_1},\ldots,{\bf W_N}) ({\bf H}_{{\bf 1}},\ldots, {\bf H}_{{\bf N}})}\\
&= \frac{1}{2}\frac{2}{N^{\frac{N}{2}}} \binom{\frac{N}{2}}{\frac{N-2}{2},1} \left( \norm{{\bf H}_1}_F^2 +\ldots+ \norm{{\bf H}_N}_F^2 \right)\left(\sum_{k=1}^N\norm{\bf W_k}_F^2 \right)^{\frac{N-2}{2}}\\
& + \frac{1}{2} \frac{2^3}{N^{\frac{N}{2}}} \binom{\frac{N}{2}}{2,\frac{N}{2}-2} \left(\act{{\bf W_1 },{\bf H_1}}+\ldots + \act{{\bf W_N},{\bf H_N}}\right)^2\left(\sum_{k=1}^N\norm{\bf W_k}_F^2\right)^{\frac{N}{2}-2}
\end{align*}
Since, the second term in the right hand side is always non-negative, the following result holds
\begin{align*}
	& \act{({\bf H}_{{\bf 1}},\ldots,{\bf H}_{{\bf N}}), \nabla^2 H_2({\bf W_1},\ldots,{\bf W_N}) ({\bf H}_{{\bf 1}},\ldots, {\bf H}_{{\bf N}})}\geq \frac{2}{N^{\frac{N}{2}}} \binom{\frac{N}{2}}{\frac{N-2}{2},1} \left( \norm{{\bf H}_1}_F^2 +\ldots+ \norm{{\bf H}_N}_F^2 \right)\left(\sum_{k=1}^N\norm{\bf W_k}_F^2 \right)^{\frac{N-2}{2}}\,.
\end{align*}
This proves the lower bound as in the statement. Now we prove the upper bound. With application of Cauchy-Schwarz inequality,  we have
\begin{align*}
&\frac{1}{2} \frac{2^3}{N^{\frac{N}{2}}} \binom{\frac{N}{2}}{2,\frac{N}{2}-2} \left(\act{{\bf W_1 },{\bf H_1}}+\ldots + \act{{\bf W_N},{\bf H_N}}\right)^2\left(\sum_{k=1}^N\norm{\bf W_k}_F^2\right)^{\frac{N}{2}-2}\\
&\leq \frac{1}{2} \frac{2^3}{N^{\frac{N}{2}}} \binom{\frac{N}{2}}{2,\frac{N}{2}-2}\left( \sum_{k=1}^N\norm{{\bf W_k}}_F^2\right) \left(\sum_{k=1}^N\norm{{\bf H_k}}_F^2\right)\left(\sum_{k=1}^N\norm{\bf W_k}_F^2\right)^{\frac{N}{2}-2}\,,\\
&= \frac{1}{2} \frac{2^3}{N^{\frac{N}{2}}} \binom{\frac{N}{2}}{2,\frac{N}{2}-2} \left(\sum_{k=1}^N\norm{{\bf H_k}}_F^2\right)\left(\sum_{k=1}^N\norm{\bf W_k}_F^2\right)^{\frac{N}{2}-1}\,.
\end{align*}
Thus, we finally have
\begin{align*}
\act{({\bf H}_{{\bf 1}},\ldots,{\bf H}_{{\bf N}}), \nabla^2 H_2({\bf W_1},\ldots,{\bf W_N}) ({\bf H}_{{\bf 1}},\ldots, {\bf H}_{{\bf N}})} &\leq  \frac{2}{N^{\frac{N}{2}}} \binom{\frac{N}{2}}{\frac{N-2}{2},1} \left(\sum_{k=1}^N\norm{{\bf H_k}}_F^2\right)\left(\sum_{k=1}^N\norm{\bf W_k}_F^2 \right)^{\frac{N-2}{2}}\\
&+ \frac{2^3}{N^{\frac{N}{2}}} \binom{\frac{N}{2}}{2,\frac{N}{2}-2} \left(\sum_{k=1}^N\norm{{\bf H_k}}_F^2\right)\left(\sum_{k=1}^N\norm{\bf W_k}_F^2\right)^{\frac{N-2}{2}}\,,\\
&= \left(\frac{N-1}{N^{\frac{N}{2} -1}}\right)\left(\sum_{k=1}^N\norm{{\bf H_k}}_F^2\right)\left(\sum_{k=1}^N\norm{\bf W_k}_F^2\right)^{\frac{N-2}{2}}\,.\qedhere
\end{align*}
\end{proof}
\begin{lemma}\label{lem:h2-inert}
Denote for any $k\geq 1$,  $x^k = ({\bf W_1^k},\ldots,{\bf W_N^k})$, $\Delta_k := {x^k - x^{k-1}}$ and the following 
\begin{align*}
\mathcal{C}_k &:=   \left(\frac{N-1}{N^{\frac{N}{2} -1}}\right) \norm{\Delta_k}^2\left(2\norm{x^k}^2 +  2\norm{\Delta}^2\right)^{\frac{N-2}{2}}\,.
\end{align*}
The following holds
\[
	D_{H_2}(x^k,y^k) \leq \gamma_k^2\mathcal{C}_k\,.
\]
\end{lemma}
\begin{proof}
From Lemma~\ref{lem:basic-ftc}, we have
\begin{align*}
&\int_{0}^1 \left(1-t\right)\int_{0}^1 \act{ \nabla^2H_2\left(x^k + (t_1+ (1-t_1)t)(y^k-x^k)\right)(x^k-y^k), x^k-y^k}dt_1dt\\
&= \gamma_k^2\int_{0}^1 \left(1-t\right)\int_{0}^1 \act{ \nabla^2H_2\left(x^k + (t_1+ (1-t_1)t)(y^k-x^k)\right)(x^k-x^{k-1}), x^k-x^{k-1}}dt_1dt\\
& \leq \gamma_k^2\int_{0}^1 \left(1-t\right)\int_{0}^1 \left(\frac{2N-3}{N^{\frac{N}{2} -1}}\right) \norm{x^k-x^{k-1}}^2\norm{x^k + (t_1+ (1-t_1)t)(y^k-x^k)}^{N-2}   dt_1dt
\end{align*}
where in the last we used Lemma~\ref{lem:basic-h2}. Now, we use the following inequality
\begin{align*}
	\norm{x^k + (t_1+ (1-t_1)t)(y^k-x^k)}^{2} 	&\leq 2\norm{x^k}^2 +  2\norm{x^k - x^{k-1}}^2 = 2\norm{x^k}^2 +  2\norm{\Delta}^2\,.
\end{align*}
Thus, the result follows using $\int_{0}^1 (1-t)dt = \frac{1}{2}$.\qedhere
\end{proof}

% \begin{proposition}
% 	Let $N$ be even and $N>2$. Let $g,H_1,H_2$ be as defined above. Then, for a certain constant $L\geq 1$, the function $g = \Psi$ satisfies the $L$-smad property with respect to the following kernel generating distance
% 	\begin{equation}\label{eq:main-lsmad-h}
% 	 H_a({\bf W_1},\ldots,{\bf W_N}) =  c_1(N)H_1({\bf W_1},\ldots,{\bf W_N}) + c_2(N) H_2({\bf W_1},\ldots,{\bf W_N}) \,,
% 	\end{equation}
% 	where we have 
% 	\[
% 	c_1(N)  =  \frac{(2N-1)N^N}{2N!}\norm{{\bf X}}_F^2  \quad c_2(N) = \frac{\norm{{\bf Y}}_F\norm{\bf X}_FN^{\frac{N}{2}}}{ \binom{\frac{N}{2}}{\frac{N-2}{2},1}2(N-2)^{\frac{N-2}{2}}}
% 	\]
% \end{proposition}
\subsection{Proof of Proposition~\ref{prop:even-case}}\label{sec:even-case}

We need to prove the convexity of $LH_a-g$. 
% Consider the following
% \begin{align*}
% & \frac{1}{2}\act{({\bf H}_{{\bf 1}},\ldots,{\bf H}_{{\bf N}}), \nabla^2 H_1({\bf W_1},\ldots,{\bf W_N}) ({\bf H}_{{\bf 1}},\ldots, {\bf H}_{{\bf N}})}\\
% &\geq \frac{1}{2}\frac{2N!}{N^N} \sum_{i=1}^N\norm{{\bf H}_i}_F^2\Pi_{k=1, k \notin \{i,j\}}^N\norm{\bf W_k}_F^2
% \end{align*}
From Lemma~\ref{lem:basic-h1} we obtain 
\begin{align*}
 \frac{N^N}{2N!} \act{({\bf H}_{{\bf 1}},\ldots,{\bf H}_{{\bf N}}), \nabla^2 H_1({\bf W_1},\ldots,{\bf W_N}) ({\bf H}_{{\bf 1}},\ldots, {\bf H}_{{\bf N}})} \geq \sum_{i=1}^N\norm{{\bf H}_i}_F^2\Pi_{k=1, k \notin \{i,j\}}^N\norm{\bf W_k}_F^2
\end{align*}
Similarly from Lemma~\ref{lem:basic-h2} we obtain
% \begin{align*}
% 	&\frac{1}{2} \act{({\bf H}_{{\bf 1}},\ldots,{\bf H}_{{\bf N}}), \nabla^2 H_2({\bf W_1},\ldots,{\bf W_N}) ({\bf H}_{{\bf 1}},\ldots, {\bf H}_{{\bf N}})}\\
% &\geq \frac{1}{2}\frac{2}{N^{\frac{N}{2}}} \binom{\frac{N}{2}}{\frac{N-2}{2},1} \left( \norm{{\bf H}_1}_F^2 +\ldots+ \norm{{\bf H}_N}_F^2 \right)\left(\sum_{k=1}^N\norm{\bf W_k}_F^2 \right)^{\frac{N-2}{2}}
% \end{align*}
% on rearranging we have
\begin{align*}
	& \frac{N^{\frac{N}{2}}}{2\binom{\frac{N}{2}}{\frac{N-2}{2},1}}\act{({\bf H}_{{\bf 1}},\ldots,{\bf H}_{{\bf N}}), \nabla^2 H_2({\bf W_1},\ldots,{\bf W_N}) ({\bf H}_{{\bf 1}},\ldots, {\bf H}_{{\bf N}})}\geq \left( \norm{{\bf H}_1}_F^2 +\ldots+ \norm{{\bf H}_N}_F^2 \right)\left(\sum_{k=1}^N\norm{\bf W_k}_F^2 \right)^{\frac{N-2}{2}}
\end{align*}
Thus, now invoking Proposition~\ref{prop:main-prop}, we obtain the result.\qed
\subsection{Results for $H_3$.}
\begin{lemma}\label{lem:basic-h3}
With the following kernel generating distance
\[
	H_3({\bf W_1},\ldots,{\bf W_N}) = \left(\frac{\norm{{\bf W_1}}_F^2 + \norm{{\bf W_2}}_F^2 + \ldots \norm{{\bf W_N}}_F^2 + 1}{N+1}\right)^{\frac{N+1}{2}}\,,
\]
the gradient with respect for ${\bf W_i}$, for any $i\in [N]$, is given by 
\begin{align*}
\nabla_{{\bf W_i}}H_3({\bf W_1},\ldots,{\bf W_N}) = \frac{2\binom{\frac{N+1}{2}}{{\frac{N-1}{2}},1}}{(N+1)^{\frac{N+1}{2}}}\left(\norm{{\bf W_1 }}_F^2 + \ldots + \norm{{\bf W_N }}_F^2 + 1\right)^{{\frac{N-1}{2}}} {\bf W_i }\,,
\end{align*}
and the following lower bound holds true
\begin{align*}
	& \act{({\bf H}_{{\bf 1}},\ldots,{\bf H}_{{\bf N}}), \nabla^2 H_3({\bf W_1},\ldots,{\bf W_N}) ({\bf H}_{{\bf 1}},\ldots, {\bf H}_{{\bf N}})}\\
&\geq \frac{2}{(N+1)^{\frac{N+1}{2}}} \binom{\frac{N+1}{2}}{\frac{N-1}{2},1} \left( \norm{{\bf H}_1}_F^2 +\ldots+ \norm{{\bf H}_N}_F^2 \right)\left(\left(\sum_{k=1}^N\norm{\bf W_k}_F^2\right) + 1 \right)^{\frac{N-1}{2}}\,,
\end{align*}
and the following upper bound holds true
\begin{align*}
	& \act{({\bf H}_{{\bf 1}},\ldots,{\bf H}_{{\bf N}}), \nabla^2 H_3({\bf W_1},\ldots,{\bf W_N}) ({\bf H}_{{\bf 1}},\ldots, {\bf H}_{{\bf N}})} \leq \frac{N}{(N+1)^{\frac{N-1}{2}}} \left(\sum_{k=1}^N\norm{{\bf H_k}}_F^2\right)^2\left(\left(\sum_{k=1}^N\norm{\bf W_k}_F^2\right)+1\right)^{\frac{N-1}{2}}\,.
\end{align*}
\end{lemma}
\begin{proof}
In the expansion of  $H_3({\bf W_1}+{\bf H_1},\ldots,{\bf W_N}+{\bf H_N})$, consider only the first order terms in the expansion, from which the following gradient with respect for ${\bf W_i}$, for any $i\in [N]$, is obtained 
\begin{align*}
\nabla_{{\bf W_i}}H_3({\bf W_1},\ldots,{\bf W_N}) = \frac{\binom{\frac{N+1}{2}}{{\frac{N-1}{2}},1}}{(N+1)^{\frac{N+1}{2}}}\left(\norm{{\bf W_1 }}_F^2 + \ldots + \norm{{\bf W_N }}_F^2 + 1\right)^{{\frac{N-1}{2}}}(2{\bf W_i })\,.
\end{align*}
The second order terms are given by
\begin{align*}
	&\frac{1}{2} \act{({\bf H}_{{\bf 1}},\ldots,{\bf H}_{{\bf N}}), \nabla^2 H_3({\bf W_1},\ldots,{\bf W_N}) ({\bf H}_{{\bf 1}},\ldots, {\bf H}_{{\bf N}})}\\
&= \frac{1}{2}\frac{2}{(N+1)^{\frac{N+1}{2}}} \binom{\frac{N+1}{2}}{\frac{N-1}{2},1} \left( \norm{{\bf H}_1}_F^2 +\ldots+ \norm{{\bf H}_N}_F^2 \right)\left(\left(\sum_{k=1}^N\norm{\bf W_k}_F^2\right) + 1 \right)^{\frac{N-1}{2}}\\
& + \frac{1}{2} \frac{2^3}{(N+1)^{\frac{N+1}{2}}} \binom{\frac{N+1}{2}}{2,\frac{N-3}{2}} \left(\act{{\bf W_1 },{\bf H_1}}+\ldots + \act{{\bf W_N},{\bf H_N}}\right)^2\left(\left(\sum_{k=1}^N\norm{\bf W_k}_F^2\right)+1\right)^{\frac{N-3}{2}}\,.
\end{align*}
it is easy to see that the following lower holds true
\begin{align*}
	& \act{({\bf H}_{{\bf 1}},\ldots,{\bf H}_{{\bf N}}), \nabla^2 H_3({\bf W_1},\ldots,{\bf W_N}) ({\bf H}_{{\bf 1}},\ldots, {\bf H}_{{\bf N}})}\\
&\geq \frac{2}{(N+1)^{\frac{N+1}{2}}} \binom{\frac{N+1}{2}}{\frac{N-1}{2},1} \left( \norm{{\bf H}_1}_F^2 +\ldots+ \norm{{\bf H}_N}_F^2 \right)\left(\left(\sum_{k=1}^N\norm{\bf W_k}_F^2\right) + 1 \right)^{\frac{N-1}{2}}\,.
\end{align*}
Now we prove the upper bound. With application of Cauchy-Schwarz inequality,  we have 
\begin{align*}
& \frac{2^3}{(N+1)^{\frac{N+1}{2}}} \binom{\frac{N+1}{2}}{2,\frac{N-3}{2}} \left(\act{{\bf W_1 },{\bf H_1}}+\ldots + \act{{\bf W_N},{\bf H_N}}\right)^2\left(\left(\sum_{k=1}^N\norm{\bf W_k}_F^2\right)+1\right)^{\frac{N-3}{2}}\\
&\leq  \frac{2^3}{(N+1)^{\frac{N+1}{2}}} \binom{\frac{N+1}{2}}{2,\frac{N-3}{2}} \left(\sum_{k=1}^N\norm{\bf W_k}_F^2\right)\left(\sum_{k=1}^N\norm{\bf H_k}_F^2\right) \left(\left(\sum_{k=1}^N\norm{\bf W_k}_F^2\right)+1\right)^{\frac{N-3}{2}}\,,\\
&\leq \frac{2^3}{(N+1)^{\frac{N+1}{2}}} \binom{\frac{N+1}{2}}{2,\frac{N-3}{2}} \left(\left(\sum_{k=1}^N\norm{\bf W_k}_F^2\right) + 1\right)\left(\sum_{k=1}^N\norm{\bf H_k}_F^2\right) \left(\left(\sum_{k=1}^N\norm{\bf W_k}_F^2\right)+1\right)^{\frac{N-3}{2}}\,,\\
&=  \frac{2^3}{(N+1)^{\frac{N+1}{2}}} \binom{\frac{N+1}{2}}{2,\frac{N-3}{2}} \left(\sum_{k=1}^N\norm{\bf H_k}_F^2\right) \left(\left(\sum_{k=1}^N\norm{\bf W_k}_F^2\right)+1\right)^{\frac{N-1}{2}}\,,
\end{align*}
where in the second inequality we used $\left(\sum_{k=1}^N\norm{\bf W_k}_F^2\right) \leq \left(\sum_{k=1}^N\norm{\bf W_k}_F^2\right) + 1$. Now the full bound is
\begin{align*}
	&\frac{1}{2} \act{({\bf H}_{{\bf 1}},\ldots,{\bf H}_{{\bf N}}), \nabla^2 H_3({\bf W_1},\ldots,{\bf W_N}) ({\bf H}_{{\bf 1}},\ldots, {\bf H}_{{\bf N}})}\\
&\leq \frac{1}{2}\frac{2}{(N+1)^{\frac{N+1}{2}}} \binom{\frac{N+1}{2}}{\frac{N-1}{2},1} \left(\sum_{k=1}^N\norm{{\bf H_k}}_F^2\right)^2\left(\left(\sum_{k=1}^N\norm{\bf W_k}_F^2\right) + 1 \right)^{\frac{N-1}{2}} \,,\\
& + \frac{1}{2} \frac{2^3}{(N+1)^{\frac{N+1}{2}}} \binom{\frac{N+1}{2}}{2,\frac{N-3}{2}}  \left(\sum_{k=1}^N\norm{{\bf H_k}}_F^2\right)^2\left(\left(\sum_{k=1}^N\norm{\bf W_k}_F^2\right)+1\right)^{\frac{N-1}{2}} \,,\\
& =\frac{1}{2}\frac{N}{(N+1)^{\frac{N-1}{2}}} \left(\sum_{k=1}^N\norm{{\bf H_k}}_F^2\right)^2\left(\left(\sum_{k=1}^N\norm{\bf W_k}_F^2\right)+1\right)^{\frac{N-1}{2}}\,.\qedhere
\end{align*}
\end{proof}
\begin{lemma}\label{lem:h3-inert}
Denote for any $k\geq 1$,  $x^k = ({\bf W_1^k},\ldots,{\bf W_N^k})$, $\Delta_k := {x^k - x^{k-1}}$ and the following 
\begin{align*}
\mathcal{D}_k &:=  \frac{N}{(N+1)^{\frac{N-1}{2}}}  \norm{\Delta_k}^2 \left(2\norm{x^k}^2 +  2\norm{\Delta}^2 + 1 \right)^{\frac{N-1}{2}}\,.
\end{align*}
Then, the condition $D_{H_3}(x^k,y^k) \leq \gamma_k^2\mathcal{D}_k$ holds true.
\end{lemma}
\begin{proof}
From Lemma~\ref{lem:basic-ftc} and using $y^k = x^k + \gamma_k (x^k - x^{k-1})$ we have
\begin{align*}
&\int_{0}^1 \left(1-t\right)\int_{0}^1 \act{ \nabla^2H_3\left(x^k + (t_1+ (1-t_1)t)(y^k-x^k)\right)(x^k-y^k), x^k-y^k}dt_1dt\\
= & \gamma_k^2\int_{0}^1 \left(1-t\right)\int_{0}^1 \act{ \nabla^2H_3\left(x^k + (t_1+ (1-t_1)t)(y^k-x^k)\right)(x^k-x^{k-1}), x^k-x^{k-1}}dt_1dt \,,\\
\leq&  \frac{\gamma_k^2N}{(N+1)^{\frac{N-1}{2}}}\int_{0}^1 \left(1-t\right)\int_{0}^1 \norm{x^k-x^{k-1}}^2 (\norm{x^k + (t_1+ (1-t_1)t)(y^k-x^k)}^2 + 1 )^{\frac{N-1}{2}}dt_1dt.
\end{align*}
where in the last we used Lemma~\ref{lem:basic-h3}. Now, we use the following inequality
\begin{align*}
	\norm{x^k + (t_1+ (1-t_1)t)(y^k-x^k)}^{2} 	&\leq 2\norm{x^k}^2 +  2\norm{x^k - x^{k-1}}^2 = 2\norm{x^k}^2 +  2\norm{\Delta}^2\,.
\end{align*}
Thus, the result follows using $\int_{0}^1 (1-t)dt = \frac{1}{2}$.\qedhere
\end{proof}
\subsection{Proof of Proposition~\ref{prop:odd-case}.}\label{sec:odd-case}
% \begin{proposition}
% 	Let $N$ be odd and $N>2$. Let $g,H_1,H_3$ be as defined above. Then, for a certain constant $L\geq 1$, the function $g$ satisfies the $L$-smad property with respect to the following kernel generating distance
% 	\begin{equation}\label{eq:main-lsmad-h}
% 	 H_b({\bf W_1},\ldots,{\bf W_N}) =  c_1(N)H_1({\bf W_1},\ldots,{\bf W_N}) + c_3(N) H_3({\bf W_1},\ldots,{\bf W_N}) \,,
% 	\end{equation}
% 	where we have 
% 	\[
% 	c_1(N)  =  \frac{(2N-1)N^N}{2N!}\norm{{\bf X}}_F^2  \quad c_3(N) = \frac{\norm{{\bf Y}}_F\norm{\bf X}_F(N+1)^{\frac{N+1}{2}}}{ \binom{\frac{N+1}{2}}{\frac{N-1}{2},1}2(N-1)^{\frac{N-1}{2}}}
% 	\]
% \end{proposition}

We need to prove the convexity of $LH_b-g$.  
% Consider the following
% \begin{align*}
% \act{({\bf H}_{{\bf 1}},\ldots,{\bf H}_{{\bf N}}), \nabla^2 H_1({\bf W_1},\ldots,{\bf W_N}) ({\bf H}_{{\bf 1}},\ldots, {\bf H}_{{\bf N}})} \geq \frac{2N!}{N^N} \sum_{i=1}^N\norm{{\bf H}_i}_F^2\Pi_{k=1, k \notin \{i,j\}}^N\norm{\bf W_k}_F^2
% \end{align*}
From Lemma~\ref{lem:basic-h1} we obtain  
\begin{align*}
 \frac{N^N}{2N!} \act{({\bf H}_{{\bf 1}},\ldots,{\bf H}_{{\bf N}}), \nabla^2 H_1({\bf W_1},\ldots,{\bf W_N}) ({\bf H}_{{\bf 1}},\ldots, {\bf H}_{{\bf N}})} \geq \sum_{i=1}^N\norm{{\bf H}_i}_F^2\Pi_{k=1, k \notin \{i,j\}}^N\norm{\bf W_k}_F^2
\end{align*}
Similarly, from Lemma~\ref{lem:basic-h3} we obtain
% \begin{align*}
% 	& \act{({\bf H}_{{\bf 1}},\ldots,{\bf H}_{{\bf N}}), \nabla^2 H_3({\bf W_1},\ldots,{\bf W_N}) ({\bf H}_{{\bf 1}},\ldots, {\bf H}_{{\bf N}})}\\
% &\geq \frac{2}{(N+1)^{\frac{N+1}{2}}} \binom{\frac{N+1}{2}}{\frac{N-1}{2},1} \left( \norm{{\bf H}_1}_F^2 +\ldots+ \norm{{\bf H}_N}_F^2 \right)\left(\left(\sum_{k=1}^N\norm{\bf W_k}_F^2\right) + 1 \right)^{\frac{N-1}{2}}
% \end{align*}
% with rearranging we obtain
\begin{align*}
& (N+1)^{\frac{N-1}{2}}  \act{({\bf H}_{{\bf 1}},\ldots,{\bf H}_{{\bf N}}), \nabla^2 H_3({\bf W_1},\ldots,{\bf W_N}) ({\bf H}_{{\bf 1}},\ldots, {\bf H}_{{\bf N}})}  \left( \norm{{\bf H}_1}_F^2 +\ldots+ \norm{{\bf H}_N}_F^2 \right)\left(\left(\sum_{k=1}^N\norm{\bf W_k}_F^2\right) + 1 \right)^{\frac{N-1}{2}}
\end{align*}
 and invoking Proposition~\ref{prop:main-prop}, we obtain the result. The proof of $LH_b+g$ is similar (see \cite[Remark 2.1]{BSTV2018}). \qed

\section{Closed Form Update Steps}
\begin{lemma}\label{lem:helper-1} Let ${\bf Q} \in \R^{A \times B}$ for some positive integers $A$ and $B$. Let $t \geq 0$ and $\norm{{{\bf Q}}}_F \neq 0$ then 
	\[
	\min_{{\bf X} \in \R^{A \times B}}\left\{ \act{{{\bf Q}}, {\bf X}} : \norm{\bf X}_F^2 = t^2 \right\} \equiv \min_{{\bf X} \in \R^{A \times B}}\left\{ \act{{{\bf Q}}, {\bf X}} : \norm{\bf X}_F^2 \leq t^2 \right\} = -{t}\norm{{{\bf Q}}}_F\,,
	 \]
	  with the minimizer at  ${\bf X}^* = -t{{\bf Q}}/\norm{{{\bf Q}}}_F$\,.
\end{lemma}
Consider the following non-convex optimization problem
	\begin{equation}
	 \min_{{\bf W_i}\in \mathcal W_i\, \forall i\in [K]}\, \left\{ \Psi({\bf W_1},\ldots,{\bf W_N}) :=  \frac 12\norm{{\bf W_1W_2\ldots W_NX} - {\bf Y}}^2_{F} \right\}\,,
	\end{equation}
	Recall that $g = \frac 12\norm{{\bf W_1W_2\ldots W_NX} - {\bf Y}}^2_{F}$, $f :=0$ and $h$ as explained in Section~\ref{ssec:closed-form}.
\subsection{Proof of Proposition~\ref{prop:closed-form-1}}\label{sec:closed-form}
We use the same proof strategy as \cite[Proposition C.1]{MO2019a}.
	Consider the following subproblem, involved in the update step
	\begin{align*}
	({\bf W_1^{k+1}},\ldots, {\bf W_N^{k+1}}) &\in \underset{{({\bf W_1},\ldots,{\bf W_N}) \in C}}{\argmin} \left\{ \left(\sum_{i=1}^N\act{ {\bf P^k_i}, {\bf W_i}} \right)  + c_1(N) \left(\frac{\norm{{\bf W}}_F^2}{N} \right)^N + c_2(N)\left( \frac{\norm{{\bf W}}_F^2}{N}\right)^{\frac{N}{2}}  + \rho \left( \frac{\norm{{\bf W}}_F^2}{N}\right) \right\}\,.
	\end{align*}
	In order to solve the above minimization problem, we introduce additional optimization variables $t_1,\ldots,t_N\geq0$ and the constraint $\norm{{\bf W_i}}_F=t_i$ for all $i$. This splits the optimization problem, where the constraints of the inner problem with respect to ${\bf W_1},\ldots,{\bf W_N}$  can be relaxed to $\norm{{\bf W_i}}_F \leq t_i$ without changing the minimal value thanks to Lemma~\ref{lem:helper-1} . We arrive at 
 	\begin{align*}
		\min_{t_i \geq 0, \forall i\in [N]} &\left\{ \sum_{i=1}^N  \min_{{\bf W_i} \in \mathcal{W}_i} \left\{ \act{ {\bf P^k_i}, {\bf W_i}  } : \norm{{\bf W_i}}_F^2 \leq t_i^2\right\} + c_1(N)\left(\frac{\sum_{i=1}^{N}t_i^2}{N}\right)^N + c_2(N)\left(\frac{\sum_{i=1}^N t_i^2}{N}\right)^{\frac{N}{2}}  + \rho \left(\frac{\sum_{i=1}^N t_i^2}{N}\right)\right\}\,.
	\end{align*}
	 Then the solution to the subproblem for the $i$-th block due to Lemma~\ref{lem:helper-1}, in each iteration is as follows
	\begin{align*}
	{\bf W_i^{k+1}} 
	&= \left.
	\begin{cases}
	    t_i^*\frac{-{\bf P^k_i}}{\norm{{\bf P^k_i}}_F}, & \text{for } \norm{{\bf P^k_i}}_F \neq 0\,, \\
	    {\bf 0} &  otherwise\,.
	\end{cases}\right.
	\end{align*}
	We solve for $t_i^*$ with the following  minimization problem
	\begin{align*}
	\underset{t_i\geq 0, \forall i \in [N]}{\mathrm{argmin}}  \left\{  -\sum_{i=1}^Nt_i\norm{{\bf P^k_i}}_F   + c_1(N)\left(\frac{\sum_{i=1}^{N}t_i^2}{N}\right)^N + c_2(N)\left(\frac{\sum_{i=1}^N t_i^2}{N}\right)^{\frac{N}{2}}  + \rho \left(\frac{\sum_{i=1}^N t_i^2}{N}\right)\right\}.
	\end{align*}
	Thus, the solutions $t_i^*$ are the non-negative real roots of the following equations
	\begin{equation}\label{eq:system-of-eq}
	-\norm{{\bf P^k_i}}_F  + 2c_1(N)\left(\frac{\sum_{i=1}^{N}t_i^2}{N}\right)^{N-1}t_i + c_2(N)\left(\frac{\sum_{i=1}^N t_i^2}{N}\right)^{\frac{N}{2} - 1}t_i  + \frac{2\rho }{N} t_i = 0\,,\quad  \forall i \in [N]
	\end{equation}
	Substitute the following 
	\[
	t_i = r \frac{\sqrt{N}\,\norm{{\bf P^k_i}}_F}{\sqrt{\sum_{i=1}^{N}\norm{{\bf P^k_i}}_F^2}} \,, 
	\]
	which implies that $\frac{\sum_{i=1}^{N}t_i^2}{N} = r^2$ for certain $r>0$. Now, we find $r$ via substituting $t_i$ in \eqref{eq:system-of-eq}, which results in 
	\begin{equation}
	  2c_1(N) r^{2N-1} + c_2(N)r^{N - 1}  + \frac{2\rho }{N}r -\frac{\sqrt{\sum_{i=1}^{N}\norm{{\bf P^k_i}}_F^2}}{\sqrt{N}} = 0\,.
	\end{equation}

	% Further simplifications lead to $t_i = r \norm{{\bf P^k_i}}_F$, $\forall i \in [N]$ for some $r \geq 0$ such that $r$ satisfies the following polynomial
	% \[
	%  2c_1(N)\left(\frac{\sum_{i=1}^{N}\norm{{\bf P^k_i}}_F^2}{N}\right)^{N-1}r^{2N-1} + c_2(N)\left(\frac{\sum_{i=1}^N \norm{{\bf P^k_i}}_F^2}{N}\right)^{\frac{N}{2} - 1}r^{N-1}  + \frac{2\rho }{N} r -1 = 0\,.
	% \]
	The proof is similar for $N>2$ and $N$ being odd.
	\qed

\subsection{Weight decay or L2-Regularization}
Consider the following non-convex optimization problem
	\begin{equation}
	 \min_{{\bf W_i}\in \mathcal W_i\, \forall i\in [K]}\, \left\{ \Psi({\bf W_1},\ldots,{\bf W_N}) :=  \frac 12\norm{{\bf W_1W_2\ldots W_NX} - {\bf Y}}^2_{F} +\frac{\lambda_0}{2}\left(\sum_{i=1}^N\norm{{\bf W_i}}_F^2\right)\right\}\,.
	\end{equation}
	Denote $g := \frac 12\norm{{\bf W_1W_2\ldots W_NX} - {\bf Y}}^2_{F}$, $f :=\frac{\lambda_0}{2}\left(\sum_{i=1}^N\norm{{\bf W_i}}_F^2\right)$ and $h$ as explained in Section~\ref{ssec:closed-form}.

\begin{proposition}\label{prop:closed-form-2}
	In BPG, with above defined $g,f,h$, using the notation  ${\bf P^k_i} = {\bf P^k}_i\left({\bf W_1}^{{\bf k}},\ldots, {\bf W_N}^{{\bf k}} \right) = \lambda\nabla_{\bf W_i} g\left({\bf W_1}^{{\bf k}},\ldots, {\bf W_N}^{{\bf k}} \right) - \nabla_{\bf W_i} h({\bf W_1}^{{\bf k}},\ldots,{\bf W_N}^{{\bf k}})$\,. the update steps in each iteration are given by ${\bf W_{i}^{k+1}} = -r\,\frac{\sqrt{N}\,{\bf P^k_{i}}}{\norm{{\bf P}}_F}$ for all $i \in [N]$ where $r$ is the non-negative real root of for $N=2$
	\begin{equation}
	  2c_1(2) r^{3} + (c_2(2)+ \lambda\lambda_0)r  -\frac{\sqrt{\sum_{i=1}^{2}\norm{{\bf P^k_i}}_F^2}}{\sqrt{2}} = 0\,,
	\end{equation}
	If $N>2$ and even, we have
	\begin{equation}
		  2c_1(N) r^{2N-1} + c_2(N)r^{N - 1}  + \left(\frac{2\rho }{N} + \lambda\lambda_0\right)r -\frac{\sqrt{\sum_{i=1}^{N}\norm{{\bf P^k_i}}_F^2}}{\sqrt{N}} = 0\,,
	\end{equation}
	and if $N>2$ and odd, then 
	\begin{equation}
		  2c_1(N) r^{2N-1} + c_3(N)\left(\frac{N r^2+1}{N+1}\right)^{\frac{N-1}{2}}r  + \left(\frac{2\rho }{N} + \lambda\lambda_0\right)r -\frac{\sqrt{\sum_{i=1}^{N}\norm{{\bf P^k_i}}_F^2}}{\sqrt{N}} = 0\,.
	\end{equation}
	% \begin{align*}
	% & 2c_1(N)\left(\frac{\sum_{i=1}^{N}\norm{{\bf P^k_i}}_F^2}{N}\right)^{N-1}r^{2N-1} + c_3(N)\left(\frac{r^2\left(\sum_{i=1}^N \norm{{\bf P^k_i}}_F^2\right) + 1}{N+1}\right)^{\frac{N-1}{2}}r  + \left(\frac{2\rho }{N} +\lambda_0\right) r -1 = 0\,,
	% \end{align*}
	\end{proposition}
	\begin{proof}
	The proof is exactly the same as Proposition~\ref{prop:closed-form-1} and the only change is in the value $\rho$ for $N>2$ and $c_2$ for $N=2$. For $N=2$, the results coincide with \cite{MO2019a}. \qedhere
	\end{proof}
\subsection{Closed Form Updates for L1 Regularization}\label{sec:closed-form-l1}
	Recall that the soft-thresholding operator is defined as follows $	\SSS_{\theta}(x) = \max\{|x|-\theta,0\}\sgn(x)\,,$
	where the operations are performed coordinate-wise. We consider below an extension of \eqref{eq:prob-2},
	\begin{equation}
	 \min_{{\bf W_i}\in \mathcal W_i\, \forall i\in [K]}\, \left\{ \Psi({\bf W_1},\ldots,{\bf W_N}) :=  \frac 12\norm{{\bf W_1W_2\ldots W_NX} - {\bf Y}}^2_{F} + \sum_{i=1}^N\mu_i \norm{{\bf W_i}}_1 \right\}\,,
	\end{equation}
	where  $\mu_i >0$ for all $i \in [N]$ and $\norm{{\bf W_i}}_1$ is the standard L1-norm, which denotes the sum of absolute of values of the all the elements in ${\bf W_i}$. We require the following technical result from \cite{MO2019a} before we provide the closed form solutions. 
	\begin{lemma}\label{lem:helper-2} Let ${\bf Q} \in \R^{A \times B}$ for some positive integers $A$ and $B$. Let $t_0 >0$ and let $t \geq 0$ then 
	\[
	 \min_{{\bf X} \in \R^{A \times B}}\left\{ \act{{{\bf Q}}, {\bf X}} + t_0\norm{{\bf X}}_1 : \norm{{\bf X}}_F^2 \leq t^2 \right\} = -t\norm{S_{t_0}(-{\bf Q})}_F\,.
	 \]
	 with the minimizer at ${\bf X^*} = t\frac{S_{t_0}(-{\bf Q})}{\norm{S_{t_0}(-{\bf Q})}_F}$ for $\norm{S_{t_0}(-{\bf Q})}_F \neq 0$ and otherwise all ${\bf X}$ such that $\norm{{\bf X}}_F^2 \leq t^2$ are minimizers. Moreover  we have the following equivalence,
	 \begin{equation}\label{eq:equivalence}
	 \min_{{\bf X} \in \R^{A \times B}}\left\{ \act{{{\bf Q}}, {\bf X}} + t_0\norm{{\bf X}}_1 : \norm{{\bf X}}_F^2 \leq t^2 \right\} \equiv \min_{{\bf X} \in \R^{A \times B}}\left\{ \act{{{\bf Q}}, {\bf X}} + t_0\norm{{\bf X}}_1 : \norm{{\bf X}}_F^2 = t^2 \right\} \,.
	 \end{equation}
	\end{lemma}
	Denote $g := \frac 12\norm{{\bf W_1W_2\ldots W_NX} - {\bf Y}}^2_{F}$, $f := \sum_{i=1}^N\mu_i \norm{{\bf W_i}}_1$ and $h$ as explained in Section~\ref{ssec:closed-form}.
	\begin{proposition}\label{prop:closed-form-l1}
	In BPG, with above defined $g,f,h$, with the notation  ${\bf P^k_i} = {\bf P^k_i}\left({\bf W_1}^{{\bf k}},\ldots, {\bf W_N}^{{\bf k}} \right) = \lambda\nabla_{\bf W_i} g\left({\bf W_1}^{{\bf k}},\ldots, {\bf W_N}^{{\bf k}} \right) - \nabla_{\bf W_i} h({\bf W_1}^{{\bf k}},\ldots,{\bf W_N}^{{\bf k}})$\,, the update steps in each iteration are given by ${\bf W_{i}^{k+1}} = r\frac{\sqrt{N}\,\SSS_{\lambda \mu_i}(-{\bf P^k_i})}{\sqrt{\sum_{i=1}^{N}\norm{\SSS_{\lambda \mu_i}(-{\bf P^k_i})}_F^2}}$ for all $i \in [N]$  where for $N=2$, $r$ is the non-negative real root of 
	\begin{equation}
	  2c_1(2) r^{3} + c_2(2)r  -\frac{\sqrt{\sum_{i=1}^{2}\norm{\SSS_{\lambda \mu_i}(-{\bf P^k_i})}_F^2}}{\sqrt{2}} = 0\,.
	\end{equation}
	If $N>2$ and even, we have
	\begin{equation}
		  2c_1(N) r^{2N-1} + c_2(N)r^{N - 1}  +  \frac{2\rho }{N}r -\frac{\sqrt{\sum_{i=1}^{N}\norm{\SSS_{\lambda \mu_i}(-{\bf P^k_i})}_F^2}}{\sqrt{N}} = 0\,,
	\end{equation}
	and if $N>2$ and odd, then 
	\begin{equation}
		  2c_1(N) r^{2N-1} + c_3(N)\left(\frac{N r^2+1}{N+1}\right)^{\frac{N-1}{2}}r  + \frac{2\rho }{N}r -\frac{\sqrt{\sum_{i=1}^{N}\norm{\SSS_{\lambda \mu_i}(-{\bf P^k_i})}_F^2}}{\sqrt{N}} = 0\,.
	\end{equation}

	% where for $N=2$, $r$ is the non-negative real root of 
	% \[
	%  c_1(2)\left(\norm{\SSS_{\lambda \mu_i}(-{\bf P^k_1})}_F^2 + \norm{\SSS_{\lambda \mu_i}(-{\bf P^k_2})}_F^2\right)r^3 + c_2(2)r   -1 = 0\,.
	% \]
	% Denote $\norm{{\bf P}}_F^2 = \sum_{i=1}^{N}\norm{\SSS_{\lambda \mu_i}(-{\bf P^k_i})}_F^2$.	If $N>2$ and even, $r$ is the non-negative real root of 
	% \[
	%  2c_1(N)\left(\frac{\norm{{\bf P}}_F^2}{N}\right)^{N-1}r^{2N-1} + c_2(N)\left(\frac{\norm{{\bf P}}_F^2}{N}\right)^{\frac{N}{2} - 1}r^{N-1}  + \frac{2\rho }{N} r -1 = 0\,,
	% \]
	% and if $N>2$ and odd, $r$ is the non-negative real root of  
	% \begin{align*}
	% &2c_1(N)\left(\frac{\norm{{\bf P}}_F^2}{N}\right)^{N-1}r^{2N-1} + c_3(N)\left(\frac{\left(r^2\norm{{\bf P}}_F^2\right) + 1}{N+1}\right)^{\frac{N-1}{2}}r 
	%  + \frac{2\rho }{N} r -1 = 0\,.
	% \end{align*}
	\end{proposition}
	\begin{proof} We use the same proof strategy as \cite[Proposition C.1]{MO2019a}.The subproblem is
	\begin{align*}
	{\bf W^{k+1}} &\in \underset{{({\bf W_1},\ldots,{\bf W_N}) \in C}}{\argmin} \left\{ \sum_{i=1}^N\left(\lambda\mu_i \norm{{\bf W_i}}_1 + \act{ {\bf P^k_i} , {\bf W_i}}\right) + c_1(N) \left(\frac{\norm{{\bf W}}_F^2 }{N} \right)^N + c_2(N)\left( \frac{\norm{{\bf W}}_F^2}{N}\right)^{\frac{N}{2}}  + \rho \left( \frac{\norm{{\bf W}}_F^2 }{N}\right) \right\}\,.
	\end{align*}
	In order to solve the above minimization problem, we introduce additional optimization variables $t_1,\ldots,t_N\geq0$ and the constraint $\norm{{\bf W_i}}_F=t_i$ for all $i$. This splits the optimization problem, where the constraints of the inner problem with respect to ${\bf W_1},\ldots,{\bf W_N}$  can be relaxed to $\norm{{\bf W_i}}_F \leq t_i$ without changing the minimal value thanks to Lemma~\ref{lem:helper-2}.  We arrive at 
	\begin{align*}
				\min_{t_1 \geq 0, \ldots, t_N \geq 0} &\left\{ \sum_{i=1}^N  \min_{{\bf W_i} \in \mathcal{W}_i} \left\{ \act{ {\bf P^k_i}, {\bf W_i}  } + \lambda \mu_i\norm{{\bf W_i}}_1 : \norm{{\bf W_i}}_F^2 \leq  t_i^2\right\} \right. \\
	&\left. + c_1(N)\left(\frac{\sum_{i=1}^{N}t_i^2}{N}\right)^N + c_2(N)\left(\frac{\sum_{i=1}^N t_i^2}{N}\right)^{\frac{N}{2}}  + \rho \left(\frac{\sum_{i=1}^N t_i^2}{N}\right)\right\}\,.
	\end{align*}
	Lemma~\ref{lem:helper-1} provides for the $i$-th block  the optimal solution ${\bf{\widetilde W_i^*(t_i)}}$ and minimal function value $-t_i\norm{\mathcal S_{\lambda\mu_i}({\bf -P_i^k})}_F$ of the inner problem depending on $t_1,\ldots, t_N$. Thus, we obtain the solution as 
	\begin{align*}
	{\bf W_i^{k+1}} 
	&= \left.
	\begin{cases}
	    t_i^*\frac{\SSS_{\lambda \mu_i}(-{\bf P^k_i})}{\norm{\SSS_{\lambda \mu_i}(-{\bf P^k_i})}_F}, & \text{for } \norm{\SSS_{\lambda \mu_i}(-{\bf P^k_i})}_F \neq 0\,, \\
	    {\bf 0} &  otherwise\,.
	\end{cases}\right.
	\end{align*}
	We solve for $t_i^*$ with the following  minimization problem
	\begin{align*}
	\underset{t_i\geq 0, \forall i \in [N]}{\mathrm{argmin}}  \left\{  -\sum_{i=1}^Nt_i\norm{\SSS_{\lambda \mu_i}(-{\bf P^k_i})}_F   + c_1(N)\left(\frac{\sum_{i=1}^{N}t_i^2}{N}\right)^N + c_2(N)\left(\frac{\sum_{i=1}^N t_i^2}{N}\right)^{\frac{N}{2}}  + \rho \left(\frac{\sum_{i=1}^N t_i^2}{N}\right)\right\}.
	\end{align*}
	Thus, the solutions $t_i^*$ are the non-negative real roots of the following equations
	\begin{align*}
	&-\norm{\SSS_{\lambda \mu_i}(-{\bf P^k_i})}_F  + 2c_1(N)\left(\frac{\sum_{i=1}^{N}t_i^2}{N}\right)^{N-1}t_i + c_2(N)\left(\frac{\sum_{i=1}^N t_i^2}{N}\right)^{\frac{N}{2} - 1}t_i  + \frac{2\rho }{N} t_i = 0\,, \forall i \in [N]\,.
	\end{align*}
	Substitute the following 
	\[
	t_i = r \frac{\sqrt{N}\,\norm{\SSS_{\lambda \mu_i}(-{\bf P^k_i})}_F}{\sqrt{\sum_{i=1}^{N}\norm{\SSS_{\lambda \mu_i}(-{\bf P^k_i})}_F^2}} \,, 
	\]
	which implies that $\frac{\sum_{i=1}^{N}t_i^2}{N} = r^2$ for certain $r>0$. Now, we find $r$ via substituting $t_i$ in \eqref{eq:system-of-eq}, which results in 
	\begin{equation}
	  2c_1(N) r^{2N-1} + c_2(N)r^{N - 1}  + \frac{2\rho }{N}r -\frac{\sqrt{\sum_{i=1}^{N}\norm{\SSS_{\lambda \mu_i}(-{\bf P^k_i})}_F^2}}{\sqrt{N}} = 0\,.
	\end{equation}
	 The proof is similar for $N>2$ and $N$ being odd.

	% Further simplifications lead to $t_i = r \norm{\SSS_{\lambda \mu_i}(-{\bf P^k_i})}_F$, $\forall i \in [N]$ for some $r \geq 0$ such that $r$ satisfies the following polynomial equation
	% \[
	%  2c_1(N)\left(\frac{\sum_{i=1}^{N}\norm{\SSS_{\lambda \mu_i}(-{\bf P^k_i})}_F^2}{N}\right)^{N-1}r^{2N-1} + c_2(N)\left(\frac{\sum_{i=1}^N \norm{\SSS_{\lambda \mu_i}(-{\bf P^k_i})}_F^2}{N}\right)^{\frac{N}{2} - 1}r^{N-1}  + \frac{2\rho }{N} r -1 = 0\,.
	% \]
	% On solving the above polynomial in $r$, we are bound to have a unique real solution, as the polynomial has a odd degree.
	\qedhere
	\end{proof}
\section{Closed Form Inertia}
\subsection{Proof of Proposition~\ref{lem:even-extrapolation}}\label{proof:even-extrapolation}

We use 
\[
	h({\bf W_1},\ldots,{\bf W_N}) = H_a({\bf W_1},\ldots,{\bf W_N}) + \rho H_4({\bf W_1},\ldots,{\bf W_N})\,,
\]
where 
\[
	H_a({\bf W_1},\ldots,{\bf W_N}) = c_1(N) H_1({\bf W_1},\ldots,{\bf W_N}) + c_2(N) H_2({\bf W_1},\ldots,{\bf W_N})\,.
\]
Now for any $x \in \barc,y \in C$, we have $D_{h_1 + h_2}(x,y) = D_{h_1}(x,y) + D_{h_2}(x,y)$ for any $h_1, h_2 \in \mathcal{G}(C)$. Thus,
\[
	D_{h}(x,y) = c_1(N) D_{H_1}(x,y) + c_2(N)D_{H_2}(x,y) + \rho D_{H_4}(x,y)\,.
\]
We solve $D_{h}\left(x^{k} , y^{k}\right) \leq \kappa D_{h}\left(x^{k - 1} , x^{k}\right)$ using the results from Lemma~\ref{lem:h1-inert}\,,\ref{lem:h2-inert}, to obtain
\[
	D_{h}\left(x^{k} , y^{k}\right) \leq \gamma_k^2\left(c_1(N) \mathcal{B}_k + c_2(N)\mathcal{C}_k + \rho \norm{\Delta_k}^2\right) \leq \kappa D_{h}\left(x^{k - 1} , x^{k}\right)\,.
\]
The proof for $N>2$ and $N$ being odd is similar.\qed

\subsection{Closed Form Inertia for Matrix Factorization}\label{proof:mat-fac-closed-form-inertia}
\begin{lemma}\label{lem:l-smad-helper-1}
	Given $h_1({\bf W_1},{\bf W_2}) := \left(\frac{\norm{\bf W_1}_F^2 +\norm{\bf W_2}_F^2 }{2} \right)^2$, then we have the following
	% \[
	% 	\nabla h_1({\bf W_1},{\bf W_2}) = \left(\left(\norm{\bf W_1}_F^2+\norm{\bf W_2}_F^2\right){\bf W_1}, \left(\norm{\bf W_1}_F^2+\norm{\bf W_2}_F^2\right){\bf W_2}\right)\,,
	% \]
	% \[
	% \act{\nabla h_1({\bf W_1},{\bf W_2}), ({\bf H}_{{\bf 1}}, {\bf H}_{{\bf 2}})} = \left( \norm{\bf W_1}_F^2+\norm{\bf W_2}_F^2\right) \act{	{\bf W_1}, {\bf H}_{{\bf 1}}} +  \left( \norm{\bf W_1}_F^2+\norm{\bf W_2}_F^2\right) \act{	{\bf W_2}, {\bf H}_{{\bf 2}}}
	% \]
	\begin{align*}
	\act{({\bf H}_{{\bf 1}},{\bf H}_{{\bf 2}}), \nabla^2 h_1({\bf W_1},{\bf W_2}) ({\bf H}_{{\bf 1}}, {\bf H}_{{\bf 2}})}  \leq 3\left( \norm{{\bf H_1}}_F^2 + \norm{{\bf H_2}}_F^2 \right)\left(\norm{{\bf W_1}}_F^2 + \norm{{\bf W_2}}_F^2\right)\,.
	\end{align*}
	Given $h_2 := \left(\frac{\norm{\bf W_1}_F^2 +\norm{\bf W_2}_F^2 }{2} \right)$, then we have the following
	\[
	\act{({\bf H}_{{\bf 1}},{\bf H}_{{\bf 2}}), \nabla^2 h_2({\bf W_1},{\bf W_2}) ({\bf H}_{{\bf 1}}, {\bf H}_{{\bf 2}})} = \norm{{\bf H}_{{\bf 1}}}_F^2 + \norm{{\bf H}_{{\bf 2}}}_F^2 \,.
	\]
	Then, with $h_a({\bf W_1},{\bf W_2}) =  3h_1({\bf W_1},{\bf W_2}) + \norm{{\bf Y}}_F h_2({\bf W_1},{\bf W_2})$ we have the following
	\begin{align*}
	&\act{({\bf H}_{{\bf 1}},{\bf H}_{{\bf 2}}), \nabla^2 h_a({\bf W_1},{\bf W_2}) ({\bf H}_{{\bf 1}}, {\bf H}_{{\bf 2}})}\leq 9\left( \norm{{\bf H_1}}_F^2 + \norm{{\bf H_2}}_F^2 \right)\left(\norm{{\bf W_1}}_F^2 + \norm{{\bf W_2}}_F^2\right) + \norm{{\bf Y}}_F \left(\norm{{\bf H}_{{\bf 1}}}_F^2 + \norm{{\bf H}_{{\bf 2}}}_F^2\right)\,.
	\end{align*}
	\end{lemma}
	\begin{proof}
	The result regarding $h_1$ is from Lemma~\ref{lem:basic-h1} with $N=2$. The results for $h_2$ follows trivially (see for example \cite{MO2019a}). The statement for $h_a$ holds trivially.\qedhere
	\end{proof}
In the context of matrix factorization problem, where $N=2$, ${\bf X} =1$, $\norm{{\bf X}}_F = 1$, we obtain the following result on the extrapolation parameter.
\begin{lemma}\label{lem:mat-fac-closed-form-inertia} Denote $x^k = ({\bf W_1^k},\ldots,{\bf W_N^k})$. For $\kappa>0$, $y^k:=x^k+\gamma_k(x^k - x^{k-1})$ and $x^k \neq x^{k-1}$, the parameter $\gamma_k \in [0,1]$ such that  
\[
	0\leq\gamma_k \leq \sqrt{\frac{\kappa}{\left(\xi_1^k + \xi_2^k\right)}D_h(x^{k-1},x^k)}\,,
\]
satisfies the condition \eqref{eq:inertial-step}, where  $\xi_1^k = 42\norm{x^k-x^{k-1}}^4$ and $\xi_2^k = 15\left(\norm{x^k}^2+\frac{\norm{{\bf Y}}_F}{30}\right)\norm{x^k-x^{k-1}}^2$.
\end{lemma}
 
\begin{proof}
From Lemma~\ref{lem:basic-ftc} we obtain
\begin{align*}
&\int_{0}^1 \left(1-t\right)\int_{0}^1 \act{ \nabla^2h\left(x^k + (t_1+ (1-t_1)t)(y^k-x^k)\right)(x^k-y^k), x^k-y^k}dt_1dt\\
& \leq \int_{0}^1 \left(1-t\right)\int_{0}^1  9\norm{x^k-y^k}^2\norm{x^k + (t_1+ (1-t_1)t)(y^k-x^k)}^2 + \norm{{\bf Y}}_F \norm{x^k-y^k}^2dt_1dt\\
&\leq \int_{0}^1\int_{0}^1 18\left(1-t\right) \left(\norm{x^k}^2+\frac{\norm{{\bf Y}}_F}{18}\right)\norm{x^k-y^k}^2dt_1dt + \int_{0}^1\int_{0}^1 + 18 \left(1-t\right)(t_1+ (1-t_1)t)^2\norm{x^k-y^k}^4 dt_1dt\\
& = 9\left(\norm{x^k}^2+\frac{\norm{{\bf Y}}_F}{18}\right)\norm{x^k-y^k}^2  + \int_{0}^1 18\left(1-t\right)(2t^2 + \frac{1}{3})\norm{x^k-y^k}^4dt\\
& = 9\left(\norm{x^k}^2+\frac{\norm{{\bf Y}}_F}{18}\right)\norm{x^k-y^k}^2  + 6\norm{x^k-y^k}^4\\
& = 9\gamma_k^2\left(\norm{x^k}^2+\frac{\norm{{\bf Y}}_F}{18}\right)\norm{x^k-x^{k-1}}^2  + 6\gamma_k^4\norm{x^k-x^{k-1}}^4\,,
\end{align*}
where in the first inequality we used Lemma~\ref{lem:l-smad-helper-1} and the second inequality is due to the following
\begin{align*}
	\norm{x^k + (t_1+ (1-t_1)t)(y^k-x^k)}^{2} &\leq 2\norm{x^k}^2 +  2(t_1+ (1-t_1)t)^{2}\norm{x^k - y^{k}}^2\,.
\end{align*}
Denote $\xi_2^k = 9\left(\norm{x^k}^2+\frac{\norm{{\bf Y}}_F}{18}\right)\norm{x^k-x^{k-1}}^2$ and $\xi_1^k = 6\norm{x^k-x^{k-1}}^4$ we have 
\[
	\xi_1^k\gamma_k^4 + \xi_2^k\gamma_k^2 \leq \kappa D_h(x^{k-1},x^k)\,,
\]
and the result follows due to the condition  $0\leq \gamma \leq 1$. \qedhere
% \[
% 	\left(\xi_1^k + \xi_2^k\right)\gamma_k^2 \leq \kappa D_h(x^{k-1},x^k)\,.
% \]
% Thus we have
% \[
% 	\gamma_k \leq \sqrt{\frac{\kappa}{\left(\xi_1^k + \xi_2^k\right)}D_h(x^{k-1},x^k)}\,.
% \]
\end{proof}
Note that for a general ${\bf X}$, we need to set $\xi_2^k := 15\left(\norm{x^k}^2+\frac{\norm{{\bf Y}}_F\norm{{\bf X}}_F}{30}\right)\norm{x^k-x^{k-1}}^2$.

\section{Additional Experiments}
% \textcolor{blue}{
% The following are the critical todos.
% 	\begin{itemize}
% 	\item  We can implement PALM, IPALM ($\beta =0.2,0.4,0.6$), BPG, CoCaIn BPG. Sample code for Matrix factorization is available at \url{https://github.com/mmahesh/cocain-bpg-matrix-factorization}. The numerical stability issues regarding condition checking, bregman distances, inertia need to be watched out for. 
% 	\item $N$ needs to be varied (potentially $N=2,4,6,8,10$).
% 	\item Practical datasets need to be used, for example MNIST.
% 	\item I think we need some statistical evaluation with various random initialization (maybe for two values of $N$, probably $N=4,6$ depending on the time.) Here, we vary the initialization to obtain the function value at the last iterate and obtain statistics. Atleast, CoCaIn BPG should be have the least objective value.
% 	\item Maybe some comparisons for CoCaIn BPG, with the closed form inertia and with backtracking. 
% 	\item All the above experiments need to repeated with $L2$-Regularization.
% 	\item We can also use the non-negativity constraints, but these experiments are not critical to have.
% 	\end{itemize}
% }

We provide time plots and statistical evaluation for the same experimental setting introduced in section \ref{experiment1}. In these experiments, we set the regularization parameter $\lambda_0=0.1$, the step size $\lambda$ of BPG to $0.99$ and $\rho=1$.  For iPALM we use two settings $\beta=0.2$ and $\beta=0.4$. Furthermore, we present convergence plots of a second experiment.

\subsection{Time plots}
The results for time comparison are given in Figure \ref{fig:exp1_time}. For better visualization an offset of $10^{-2}$ is used in the time plots. \ifpaper\else\medskip\fi

In most of the settings the convergence speed of CoCaIn BPG is similar to iPiano-WB. The alternating schemes PALM and iPALM do not include a time consuming backtracking mechansim. In terms of speed, this results in a better performance for the non-regularized DLNN problem. However, in the regularized setting BPG based methods with a possibly more effective update step remain superior together with iPiano-WB. In this experiment, there is no clear speed advantage of CoCaIn BPG over CoCaIn BPG CFI. The size of the used data is small yet and the strength of the closed form intertial BPG might lie in large scale datasets.

\begin{figure*}[hbt!]
  \centering
  \begin{tabular}[b]{c}
    \includegraphics[width=.3\textwidth]{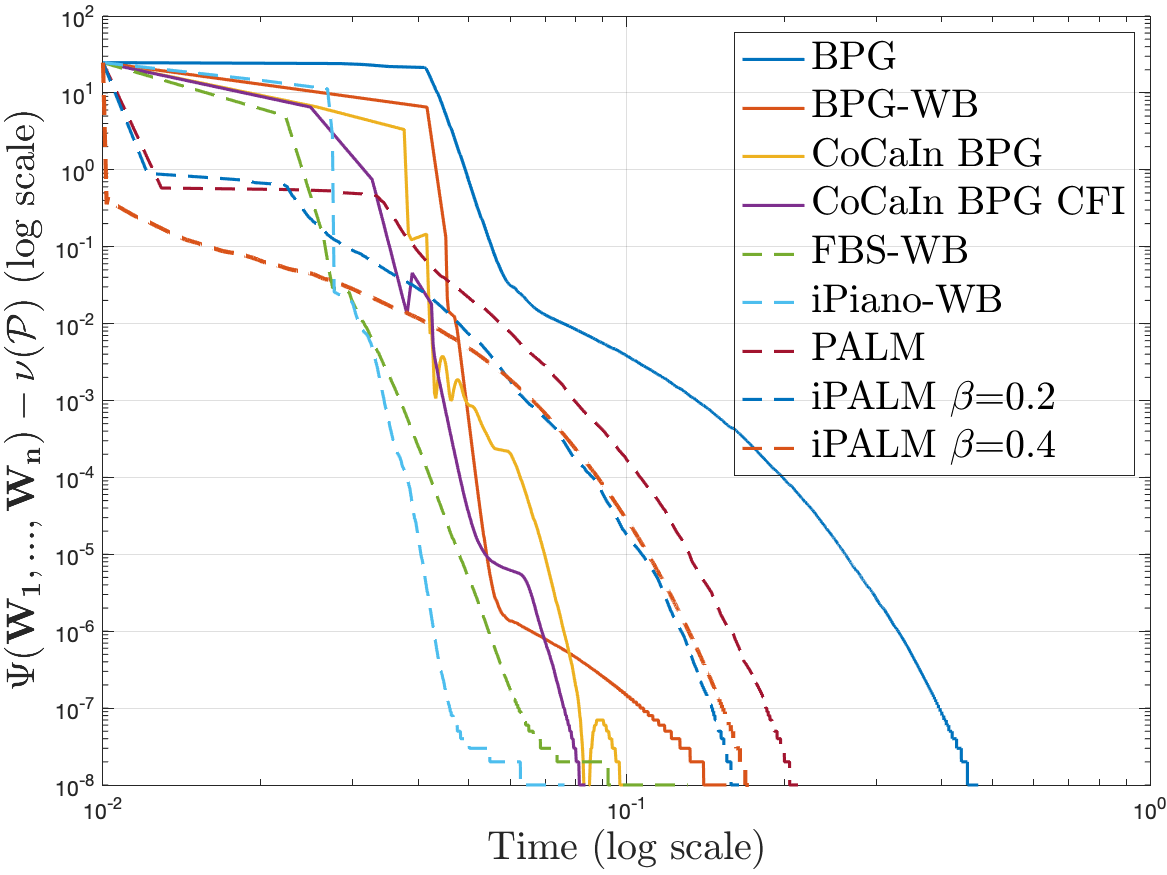} \\
    \small (a) L2-Regularization ($N=3$)
  \end{tabular}
  \begin{tabular}[b]{c}
    \includegraphics[width=.3\textwidth]{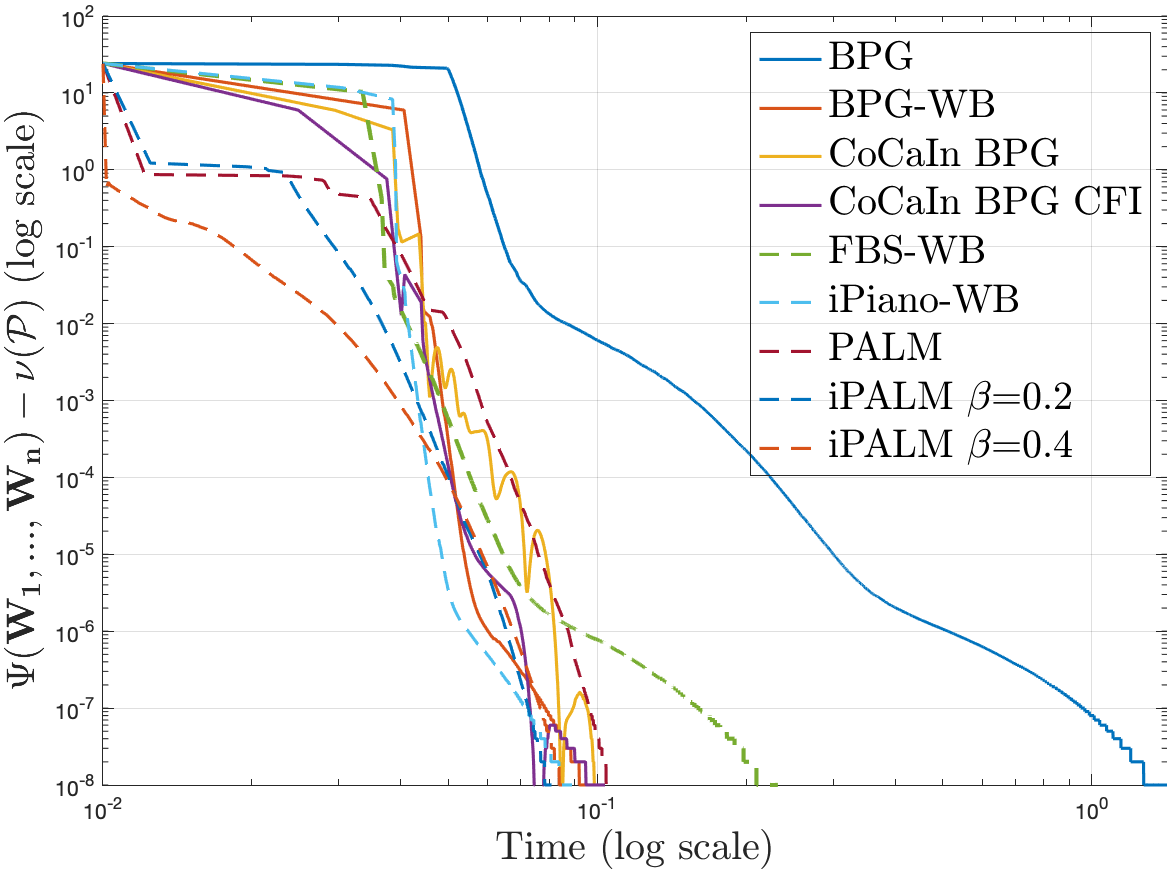} \\
    \small (b) L1-Regularization ($N=3$)
  \end{tabular} 
    \begin{tabular}[b]{c}
    \includegraphics[width=.3\textwidth]{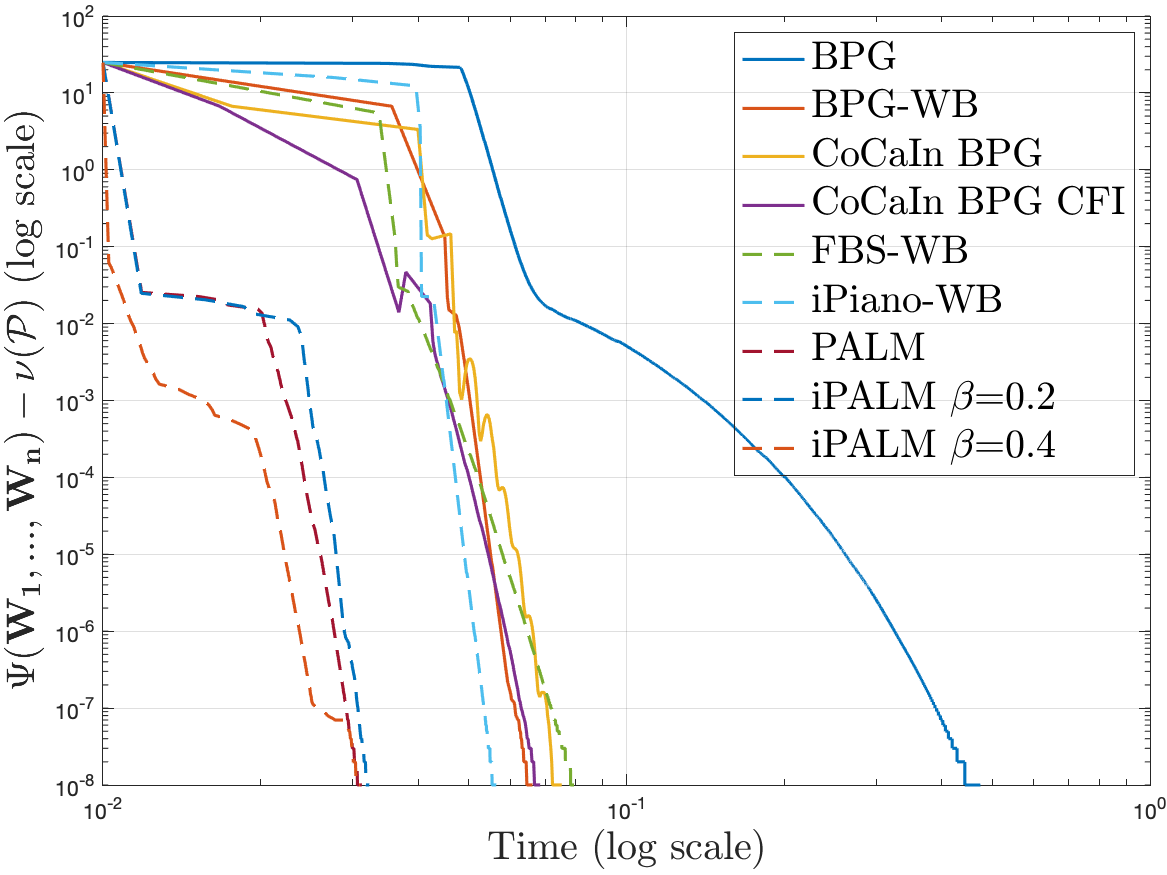} \\
    \small (c) No Regularization ($N=3$)
  \end{tabular} 
  \begin{tabular}[b]{c}
    \includegraphics[width=.3\textwidth]{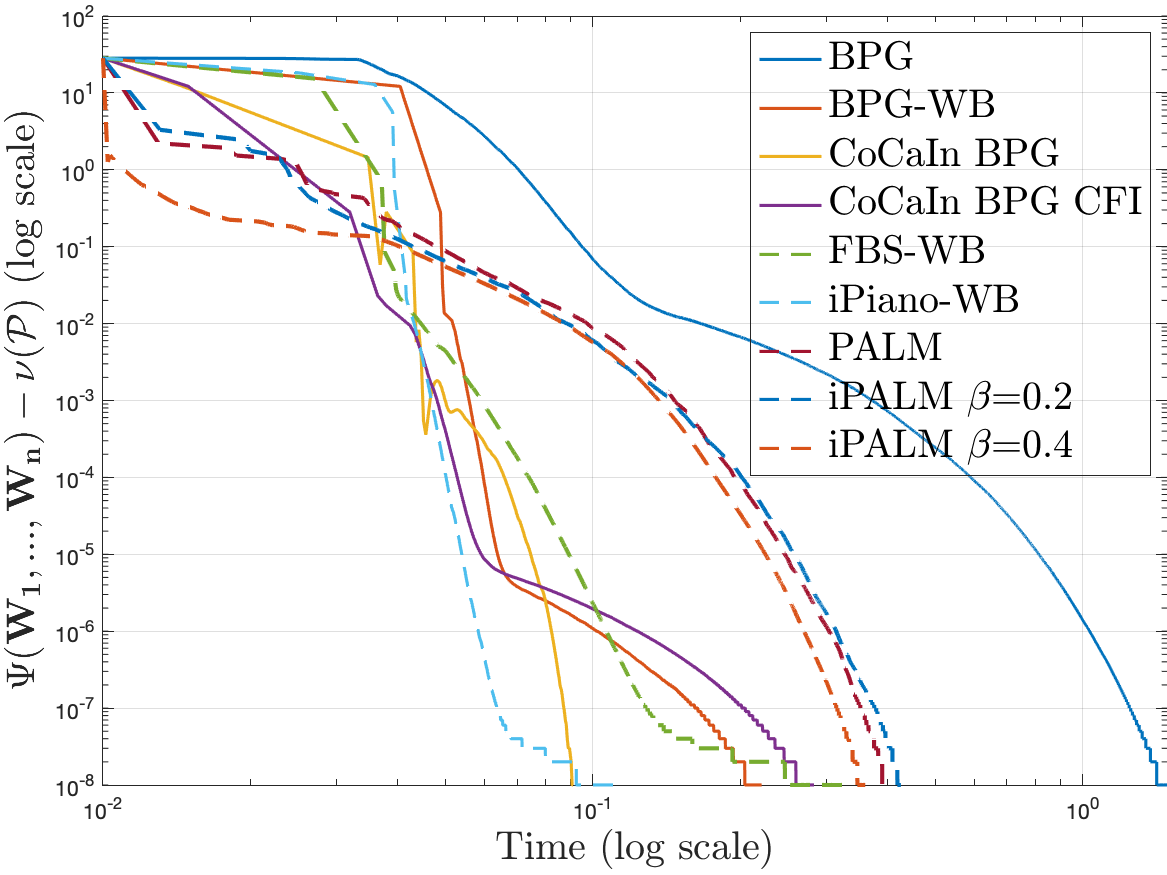} \\
    \small (d) L2-Regularization ($N=4$)
  \end{tabular}
  \begin{tabular}[b]{c}
    \includegraphics[width=.3\textwidth]{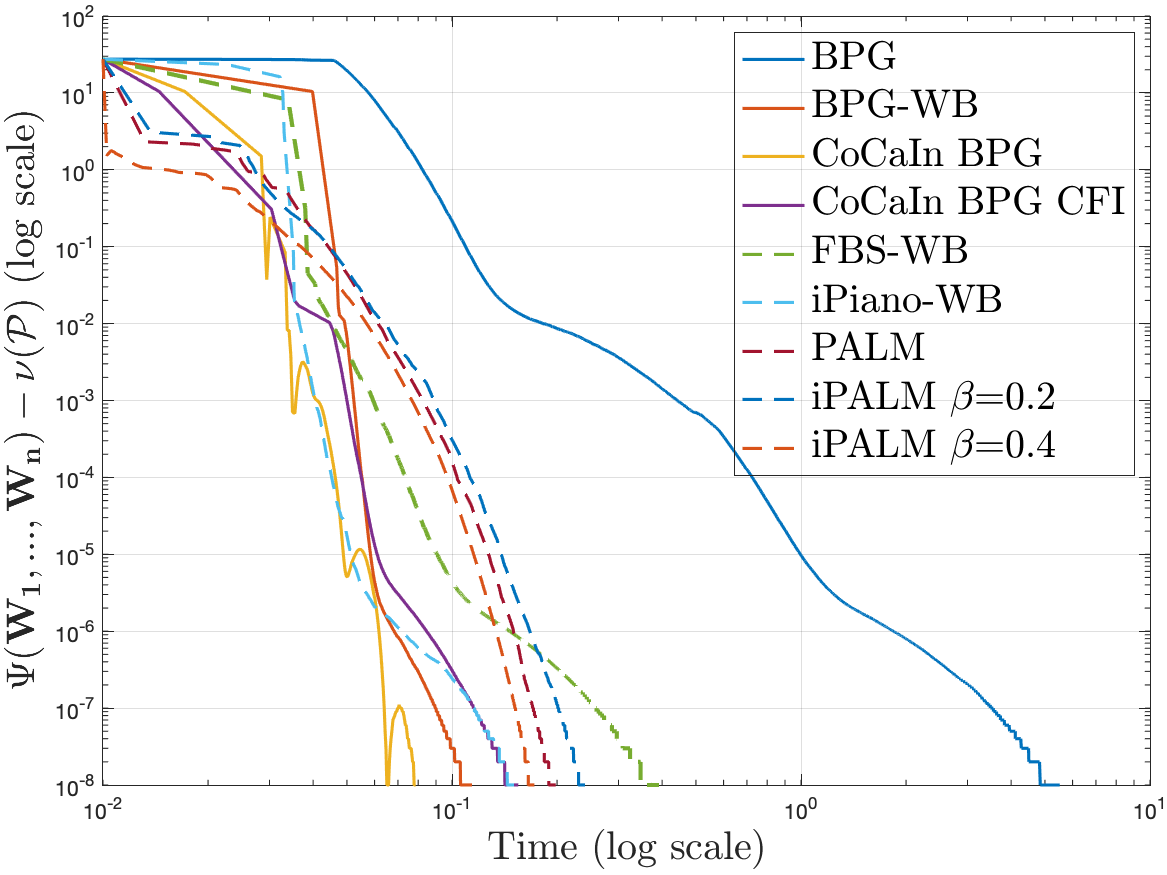} \\
    \small (e) L1-Regularization ($N=4$)
  \end{tabular}
      \begin{tabular}[b]{c}
    \includegraphics[width=.3\textwidth]{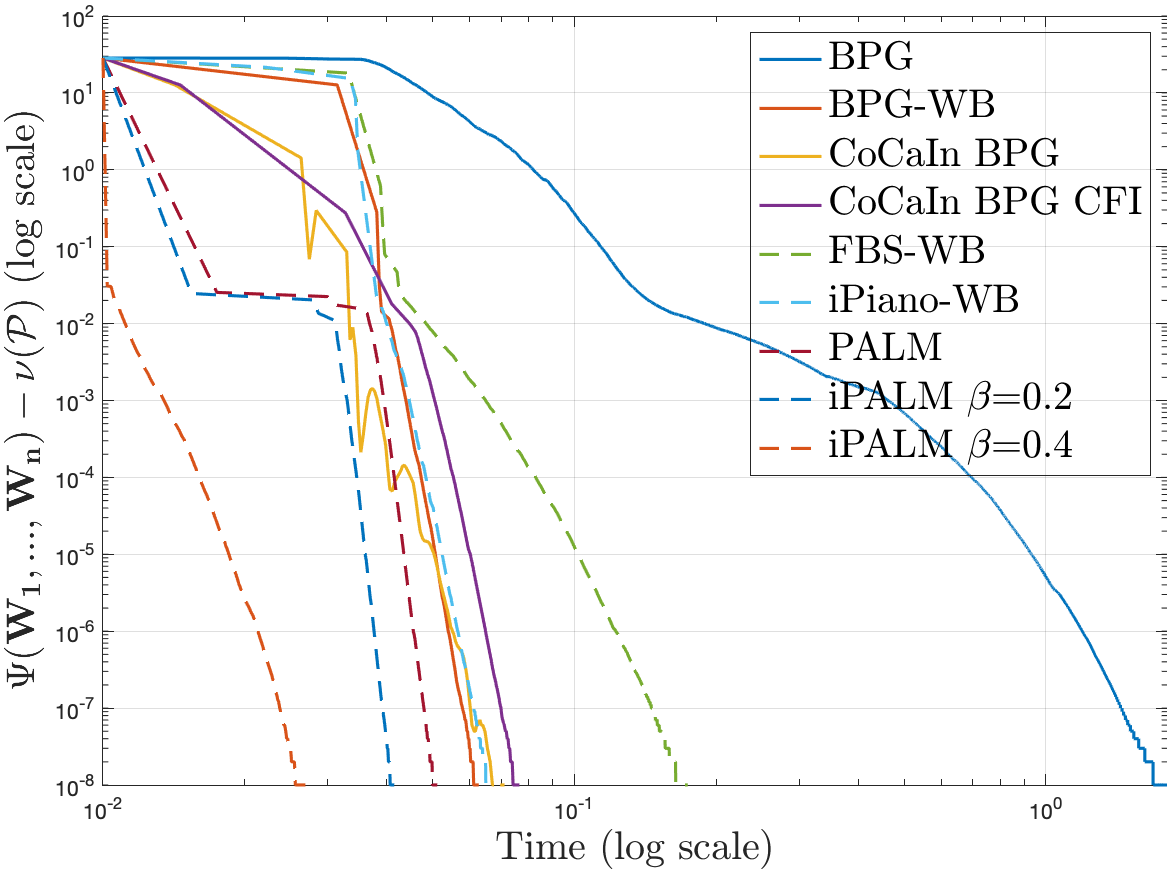} \\
    \small (f) No Regularization ($N=4$)
  \end{tabular} 
  \begin{tabular}[b]{c}
    \includegraphics[width=.3\textwidth]{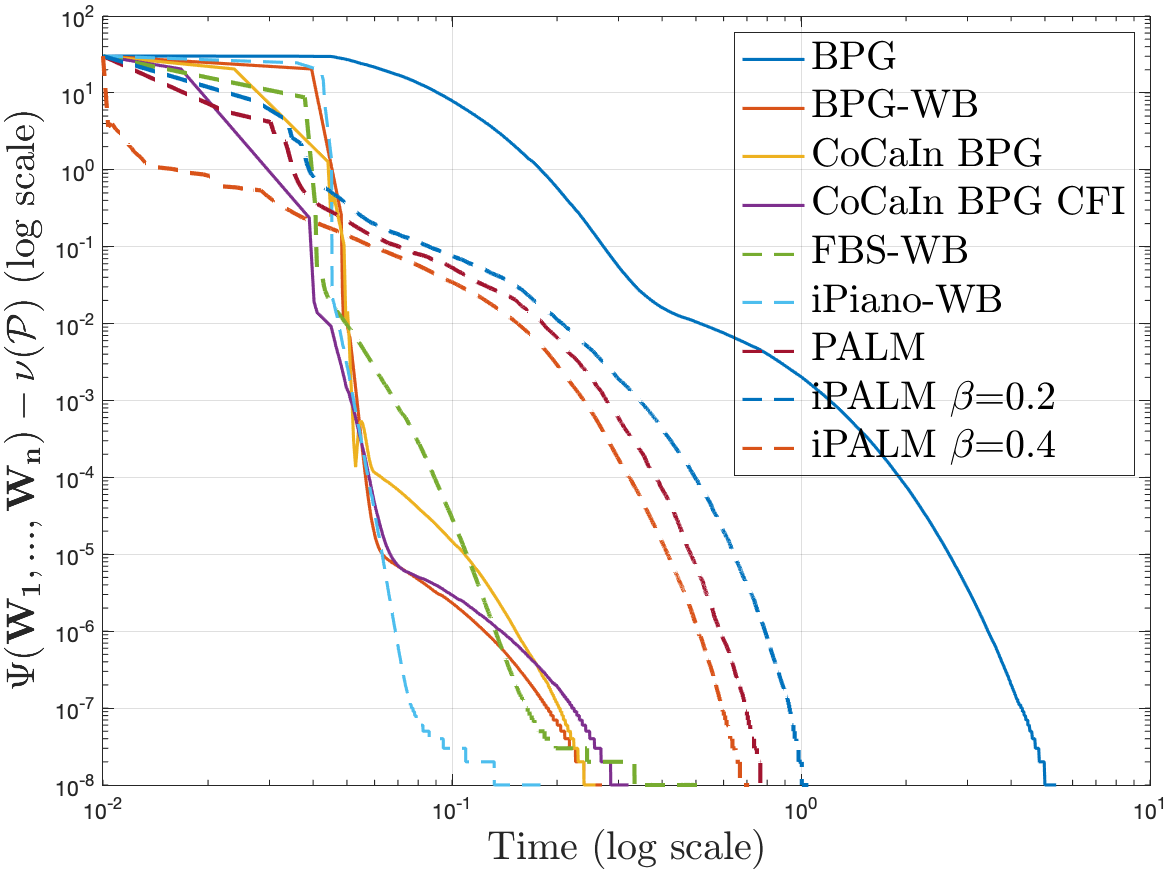} \\
    \small (g) L2-Regularization ($N=5$)
  \end{tabular}
  \begin{tabular}[b]{c}
    \includegraphics[width=.3\textwidth]{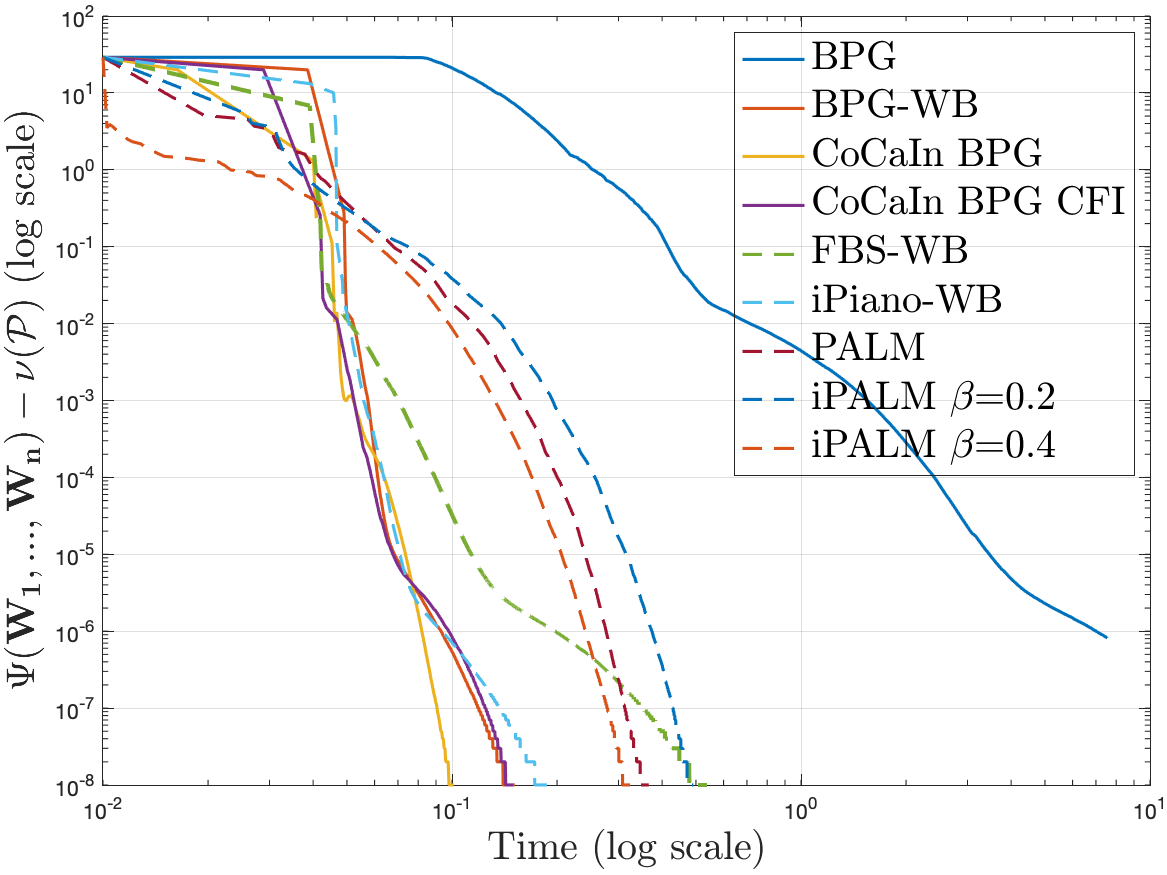} \\
    \small (h) L1-Regularization ($N=5$)
  \end{tabular} 
    \begin{tabular}[b]{c}
    \includegraphics[width=.3\textwidth]{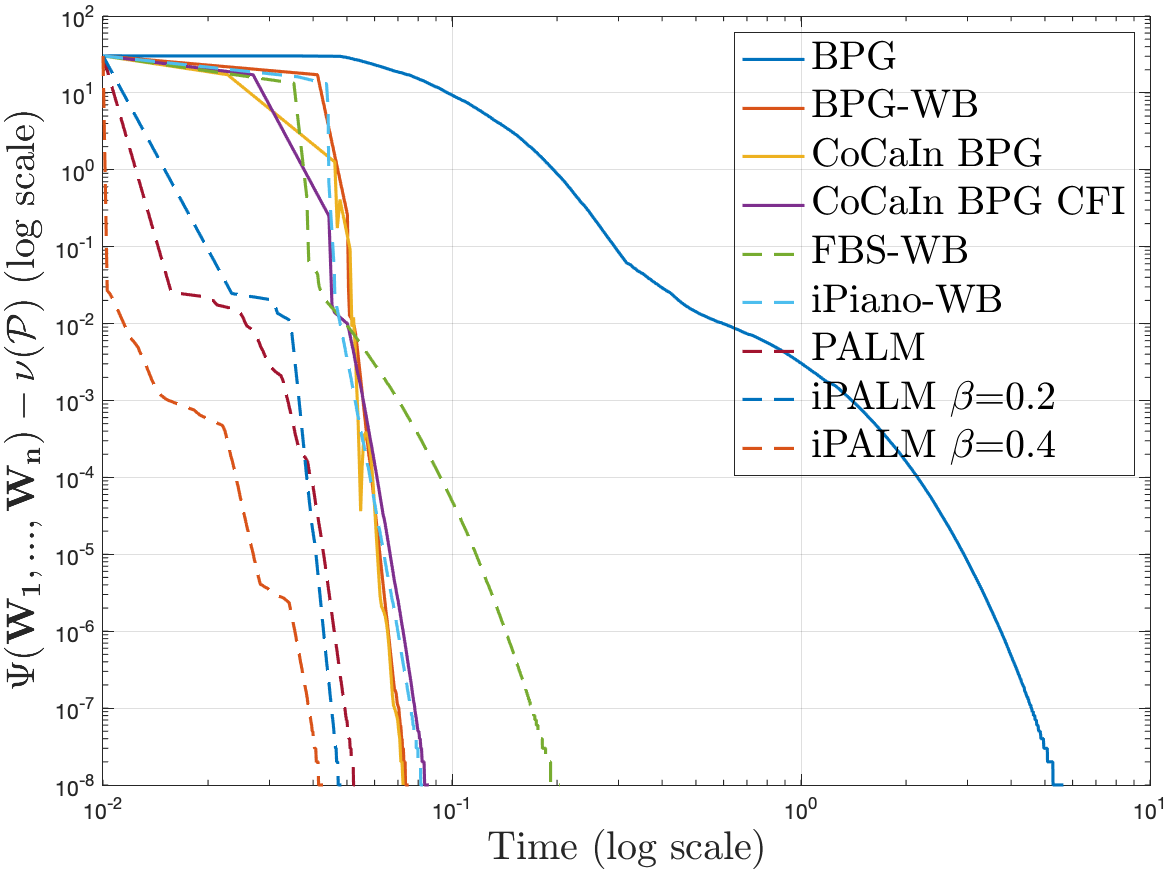} \\
    \small (g) No Regularization ($N=5$)
  \end{tabular} 
  \caption{Time plots illustrating the competitive performance of BPG methods.}
  \label{fig:exp1_time}
\end{figure*}

\subsection{Statistical evaluation}
For the statistical evaluation, we used the same experimental setting as before but varied the weight initialization: The weights are initialized randomly with values in $[0,0.1]$. We conduct 40 experiments with different seeds and plot the final result of the algorithms after 10,000 iterations. The results for three layers ($N=3$) are provided in Figures \ref{fig:exp1_stat_eval_l2reg}, \ref{fig:exp1_stat_eval_l1reg} and \ref{fig:exp1_stat_eval_noreg}. Note the different range of the $x$-axis for each algorithm.\ifpaper\else\medskip\fi

The performance of CoCaIn BPG is significantly better relative to the other algorithms in case of L2-Regularization. Here, BPG, BPG-WB and CoCaIn BPG CFI converge more often to a worse solution than FBS-based and the alternating algorithms. However, when L1-Regularization is used, both CoCaIn BPG and CoCaIn BPG CFI are superior, with CoCaIn BPG CFI being more stable than CoCaIn BPG. BPG does not fully converge within 10,000 steps.\ifpaper\else\medskip\fi

Without a regularizer, PALM and iPALM constantly generates the best results. CoCaIn BPG is competitive to FBS-WB but not to iPiano-WB.

\begin{figure*}[phbt!]
  \centering
  \scalebox{1}{\begin{tabular}[b]{c}
    \includegraphics[width=.3\textwidth]{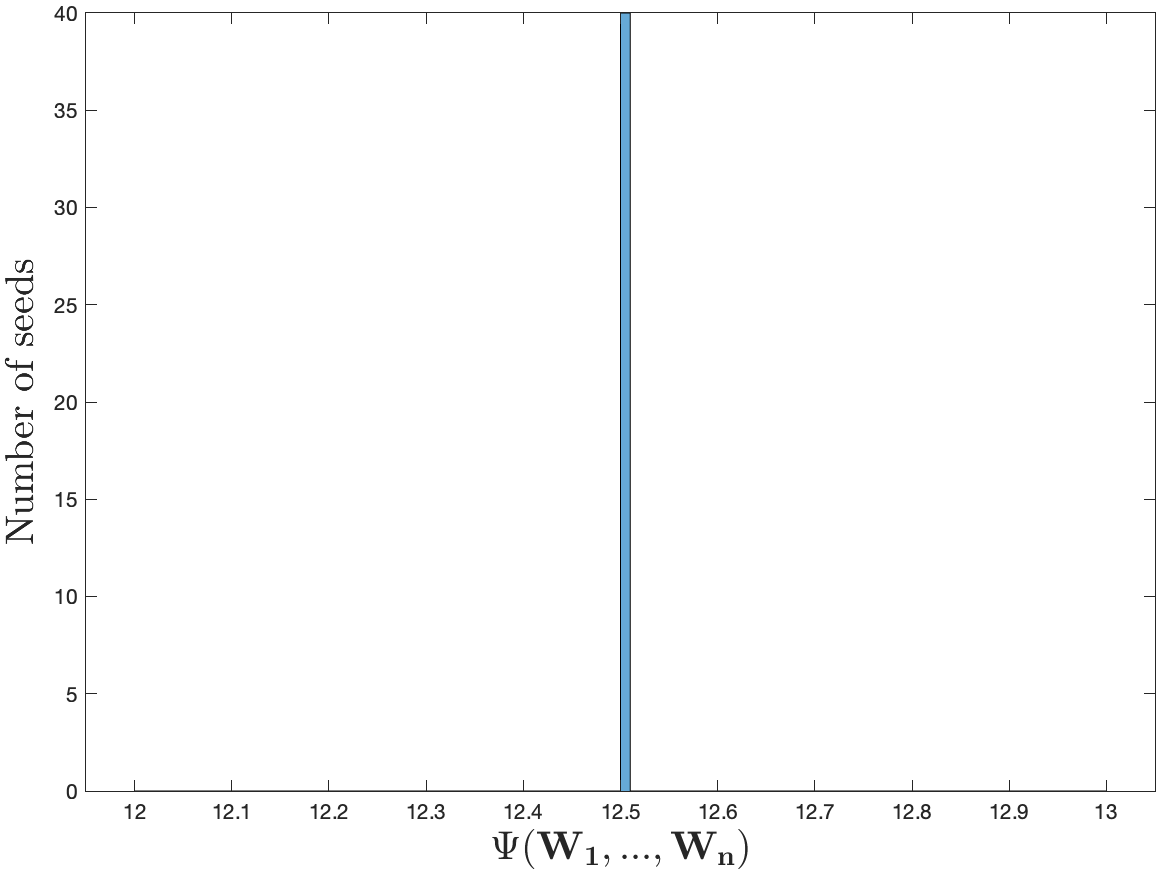} \\
    \small (a) BPG
  \end{tabular}} 
  \scalebox{1}{\begin{tabular}[b]{c}
    \includegraphics[width=.3\textwidth]{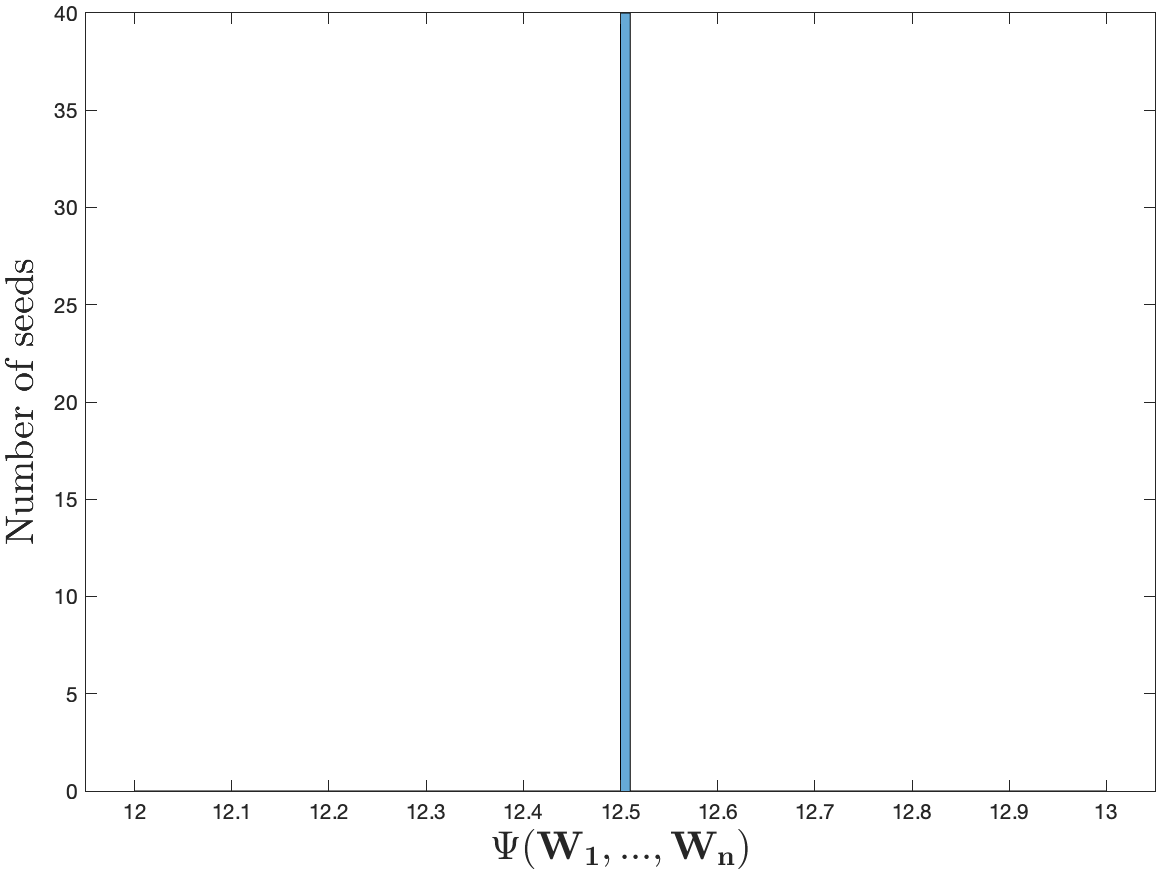} \\
    \small (b) BPG-WB
  \end{tabular}} 
  \scalebox{1}{\begin{tabular}[b]{c}
    \includegraphics[width=.3\textwidth]{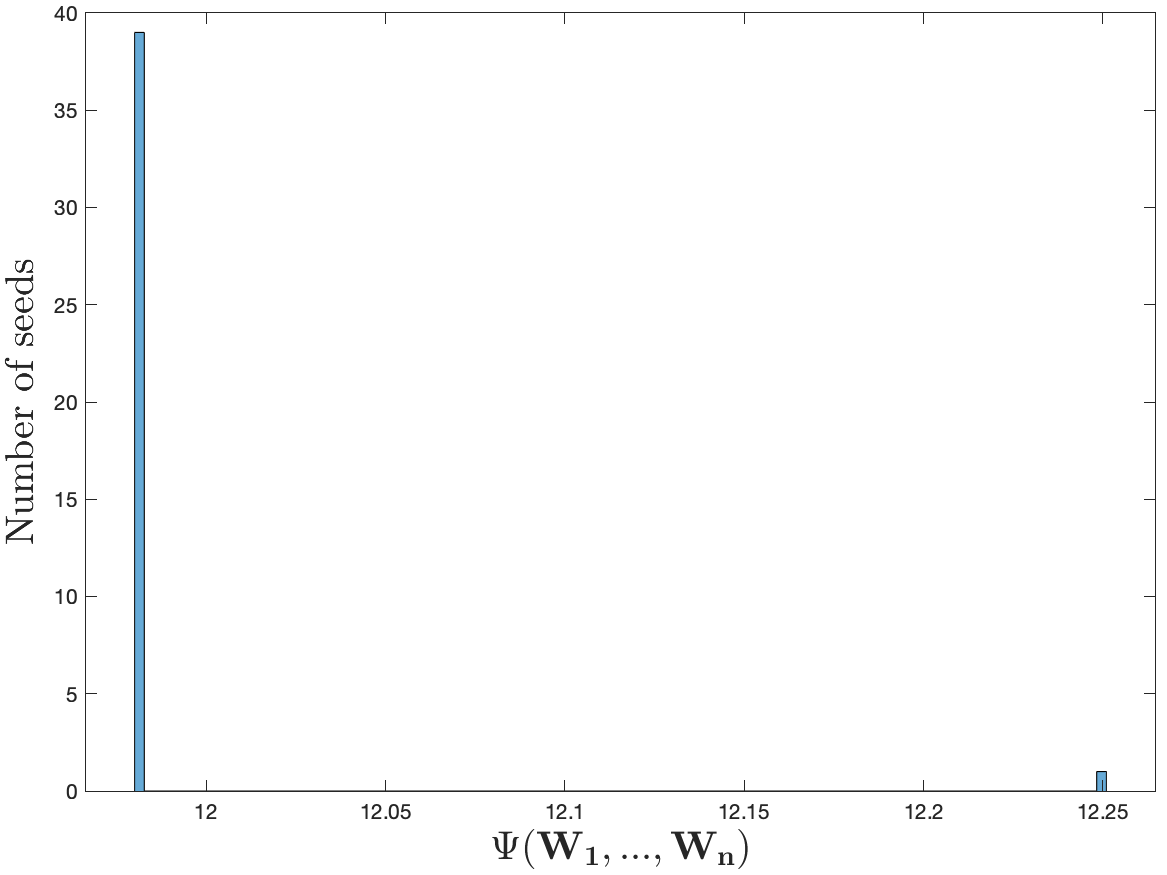} \\
    \small (c) CoCaIn BPG
  \end{tabular} } 
  \scalebox{1}{\begin{tabular}[b]{c}
    \includegraphics[width=.3\textwidth]{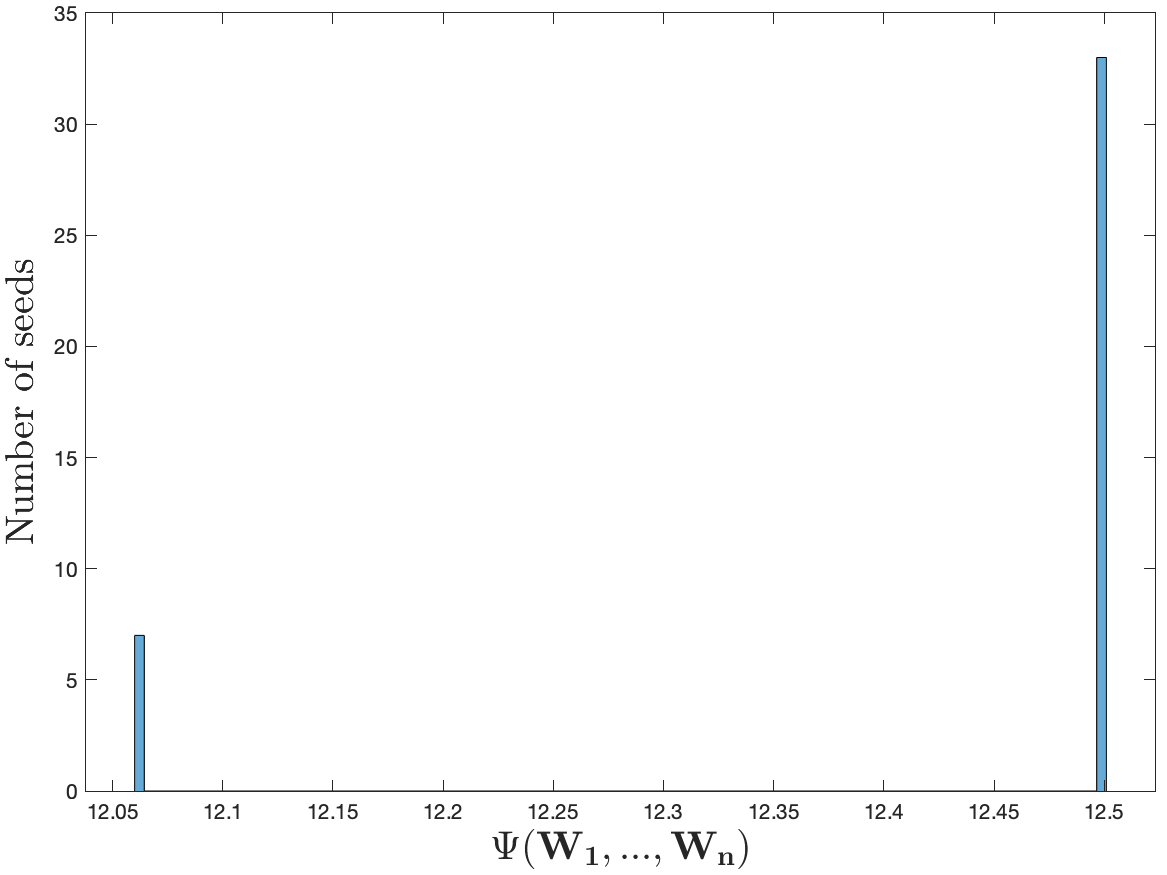} \\
    \small (d) CoCaIn BPG CFI
  \end{tabular} } 
  \scalebox{1}{\begin{tabular}[b]{c}
    \includegraphics[width=.3\textwidth]{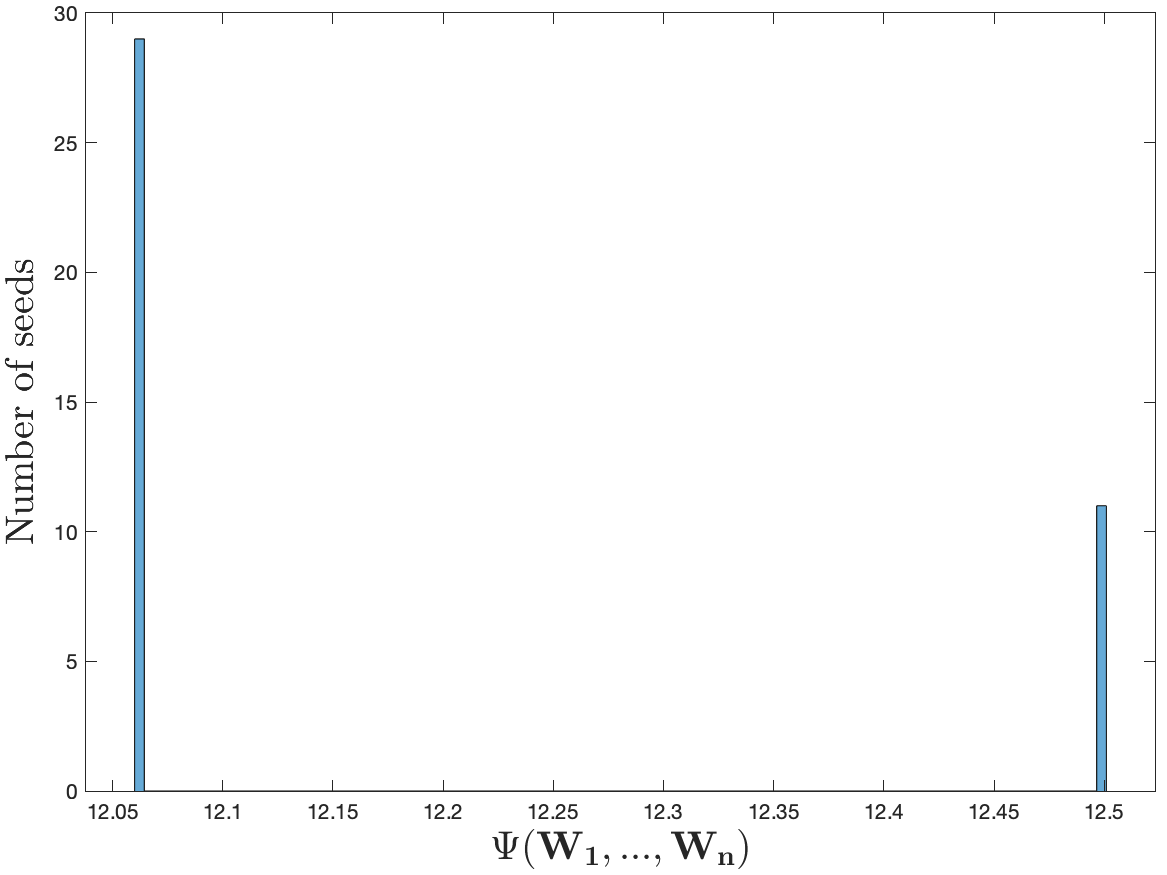} \\
    \small (e) PALM
  \end{tabular}} 
  \scalebox{1}{\begin{tabular}[b]{c}
    \includegraphics[width=.3\textwidth]{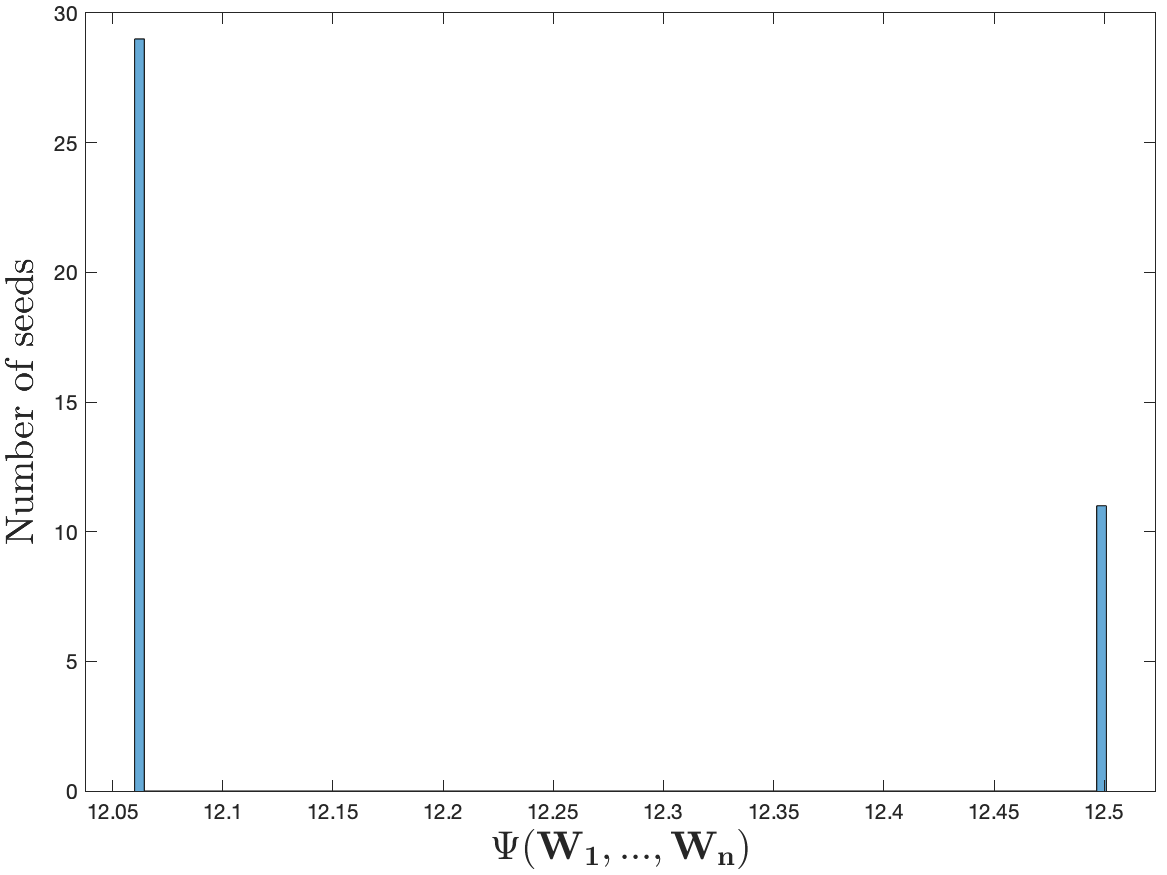} \\
    \small (f) iPALM ($\beta=0.2$)
  \end{tabular}} 
  \scalebox{1}{\begin{tabular}[b]{c}
    \includegraphics[width=.3\textwidth]{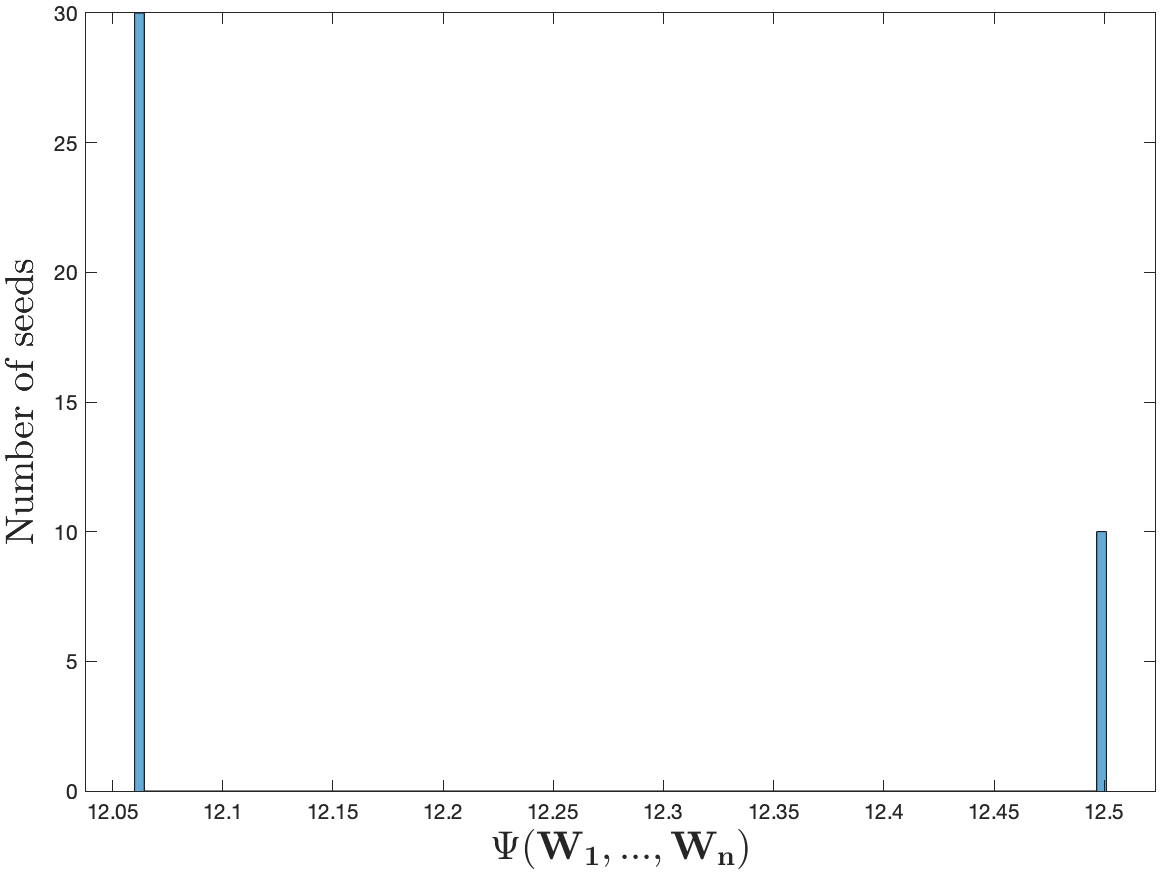} \\
    \small (g) iPALM ($\beta=0.4$)
  \end{tabular}} 
  \scalebox{1}{\begin{tabular}[b]{c}
    \includegraphics[width=.3\textwidth]{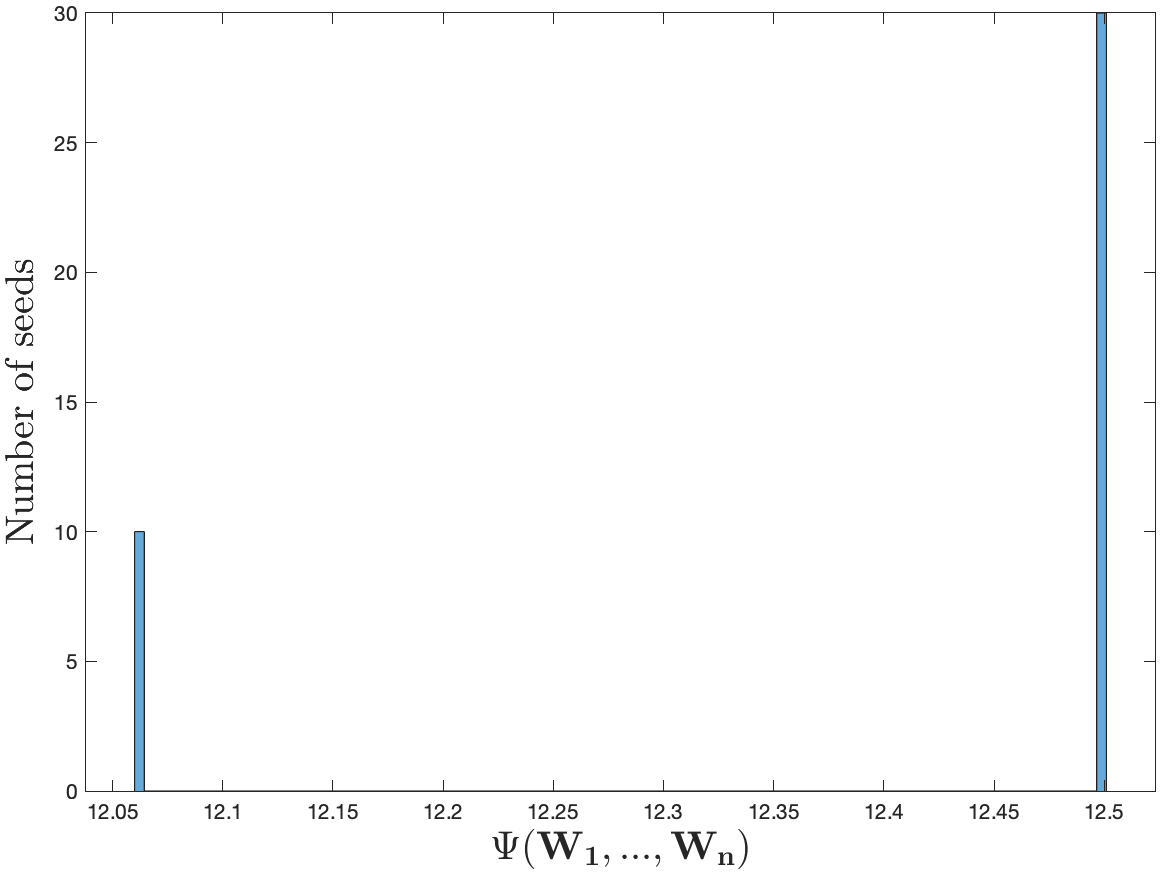} \\
    \small (h) FBS-WB
  \end{tabular}} 
  \scalebox{1}{\begin{tabular}[b]{c}
    \includegraphics[width=.3\textwidth]{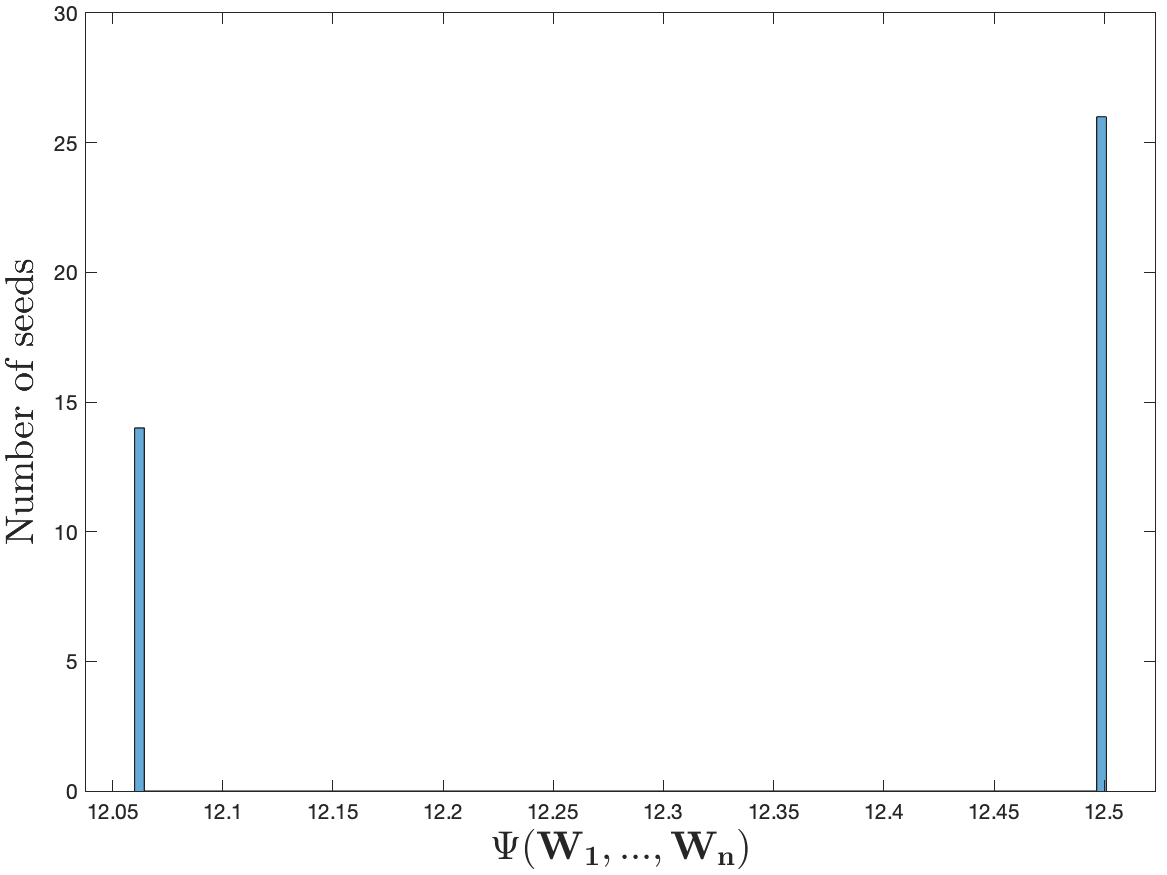} \\
    \small (i) iPiano-WB
  \end{tabular}} 
  \caption{Statistical evaluation - L2-regularization, $N=3$}
  \label{fig:exp1_stat_eval_l2reg}
\end{figure*}

\begin{figure*}[phbt!]
  \centering
  \scalebox{1}{\begin{tabular}[b]{c}
    \includegraphics[width=.3\textwidth]{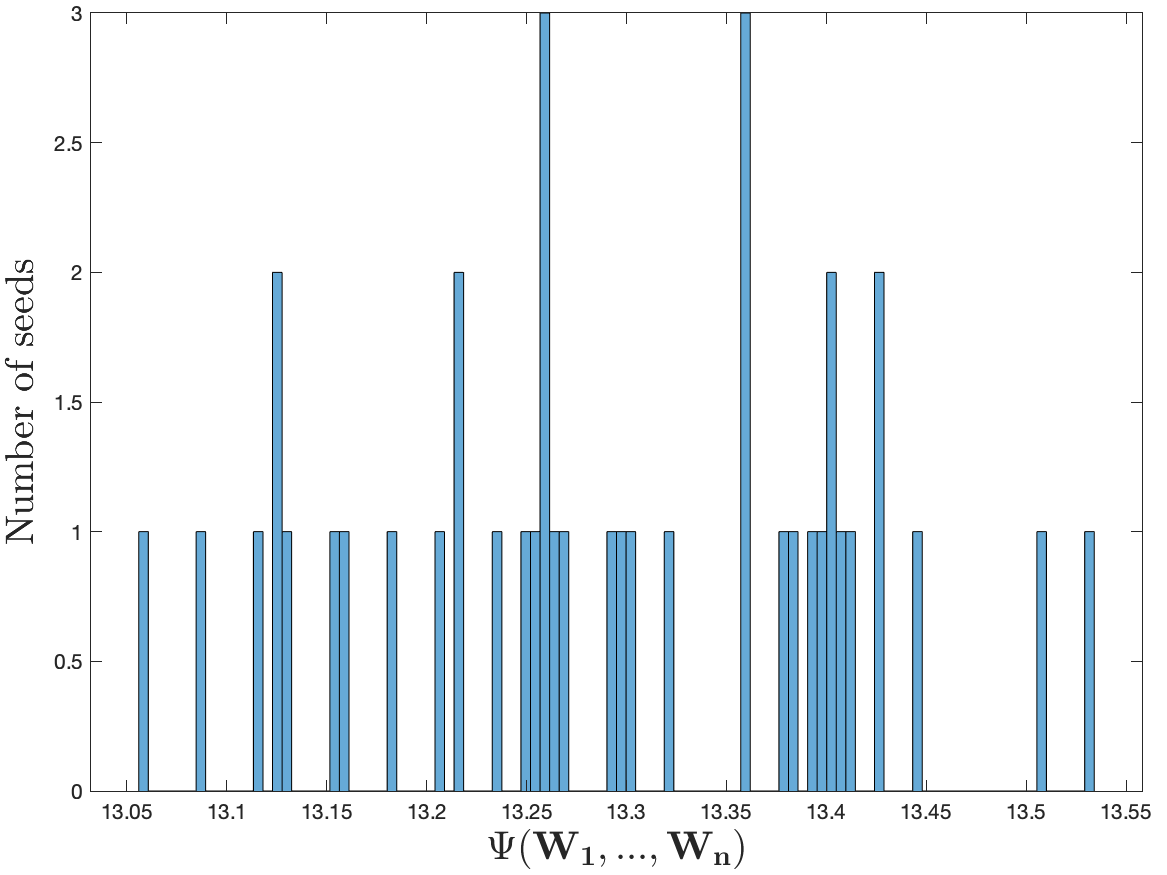} \\
    \small (a) BPG
  \end{tabular}} 
  \scalebox{1}{\begin{tabular}[b]{c}
    \includegraphics[width=.3\textwidth]{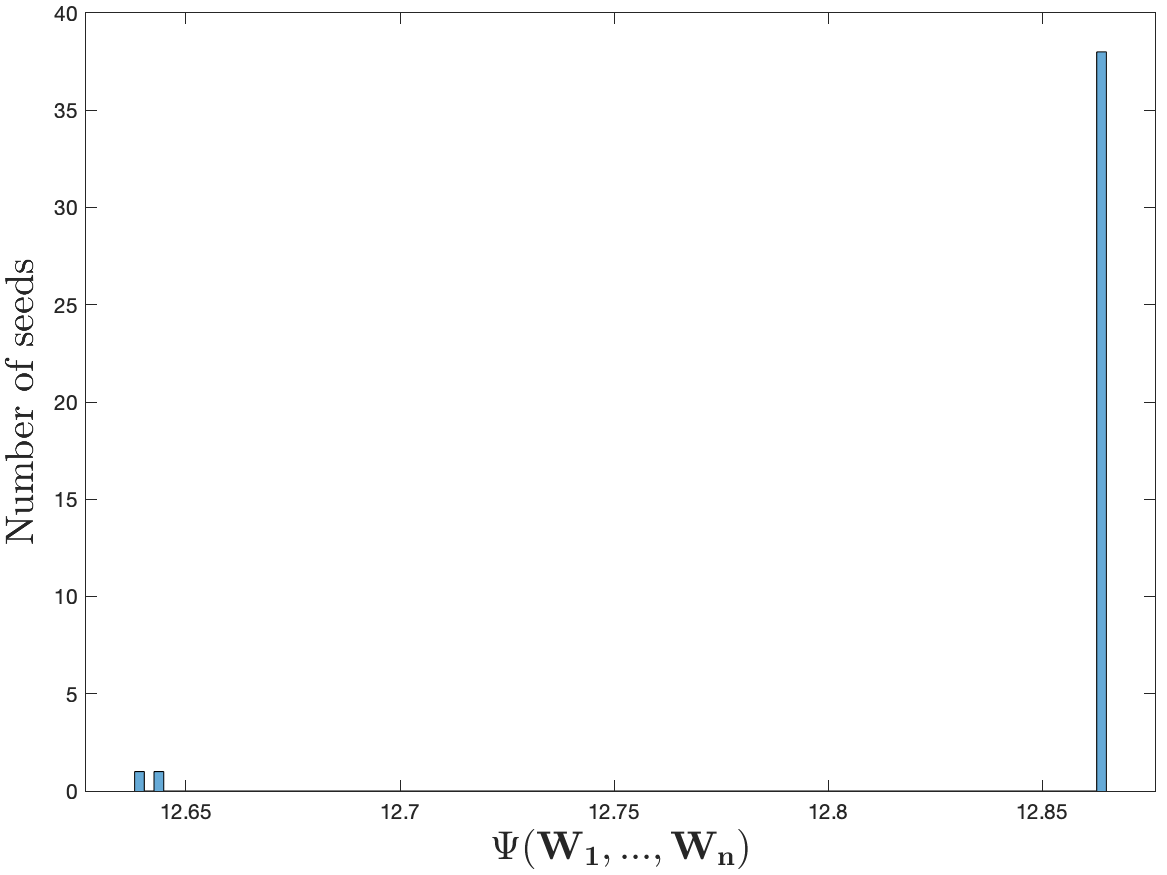} \\
    \small (b) BPG-WB
  \end{tabular}} 
  \scalebox{1}{\begin{tabular}[b]{c}
    \includegraphics[width=.3\textwidth]{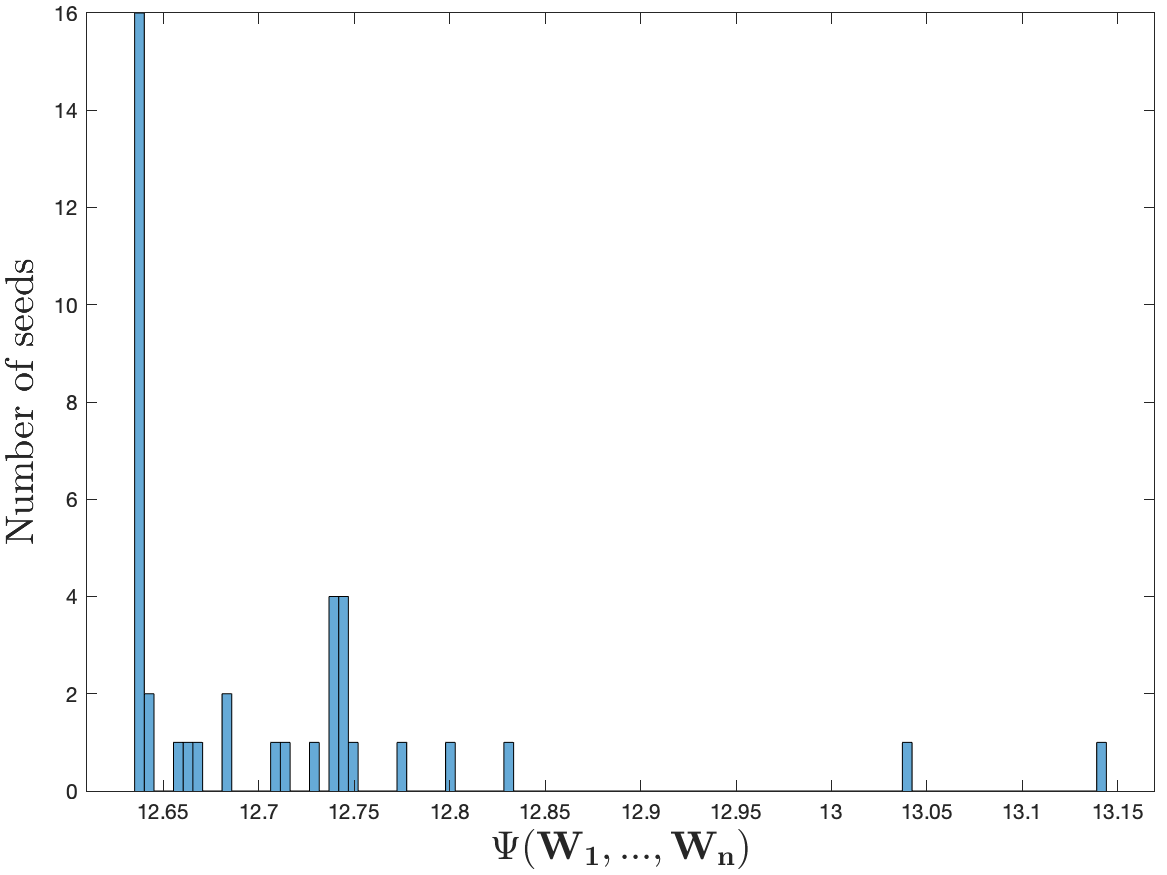} \\
    \small (c) CoCaIn BPG
  \end{tabular}} 
  \scalebox{1}{\begin{tabular}[b]{c}
    \includegraphics[width=.3\textwidth]{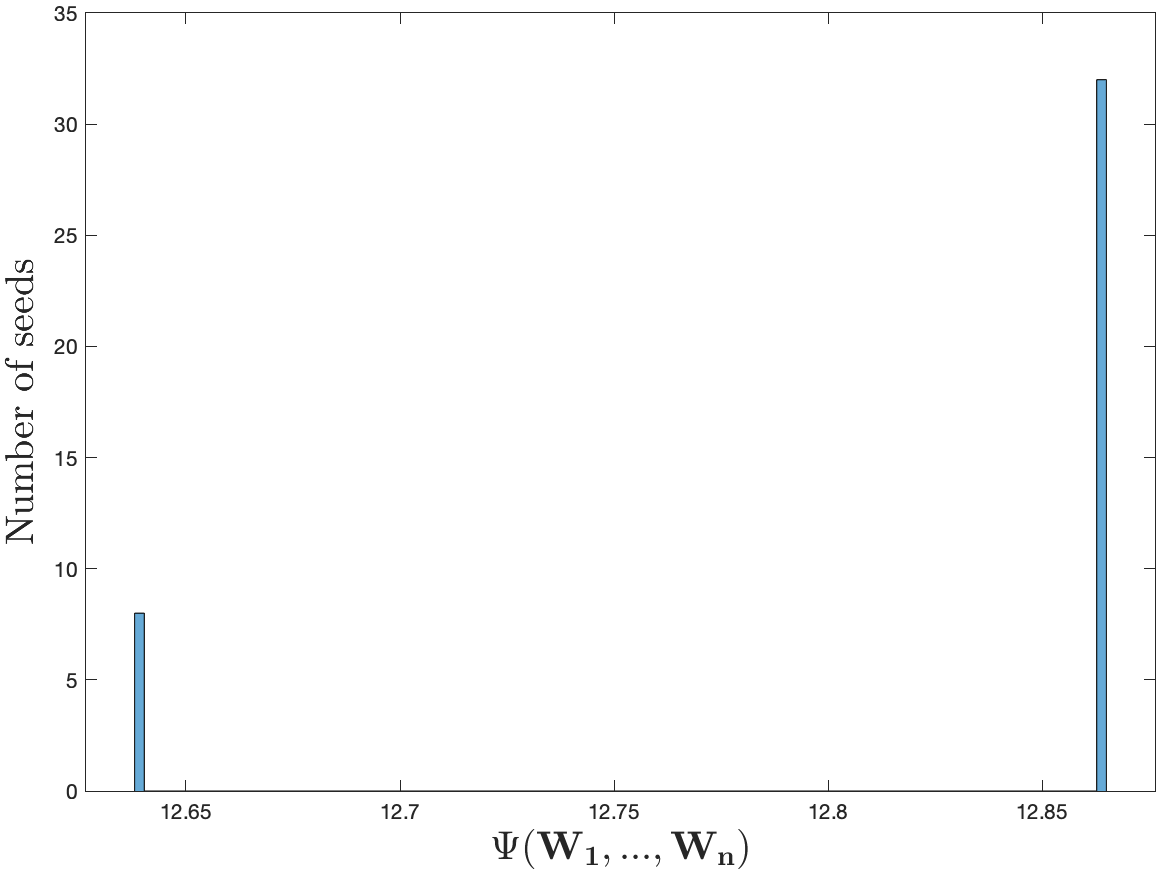} \\
    \small (d) CoCaIn BPG CFI
  \end{tabular}} 
  \scalebox{1}{\begin{tabular}[b]{c}
    \includegraphics[width=.3\textwidth]{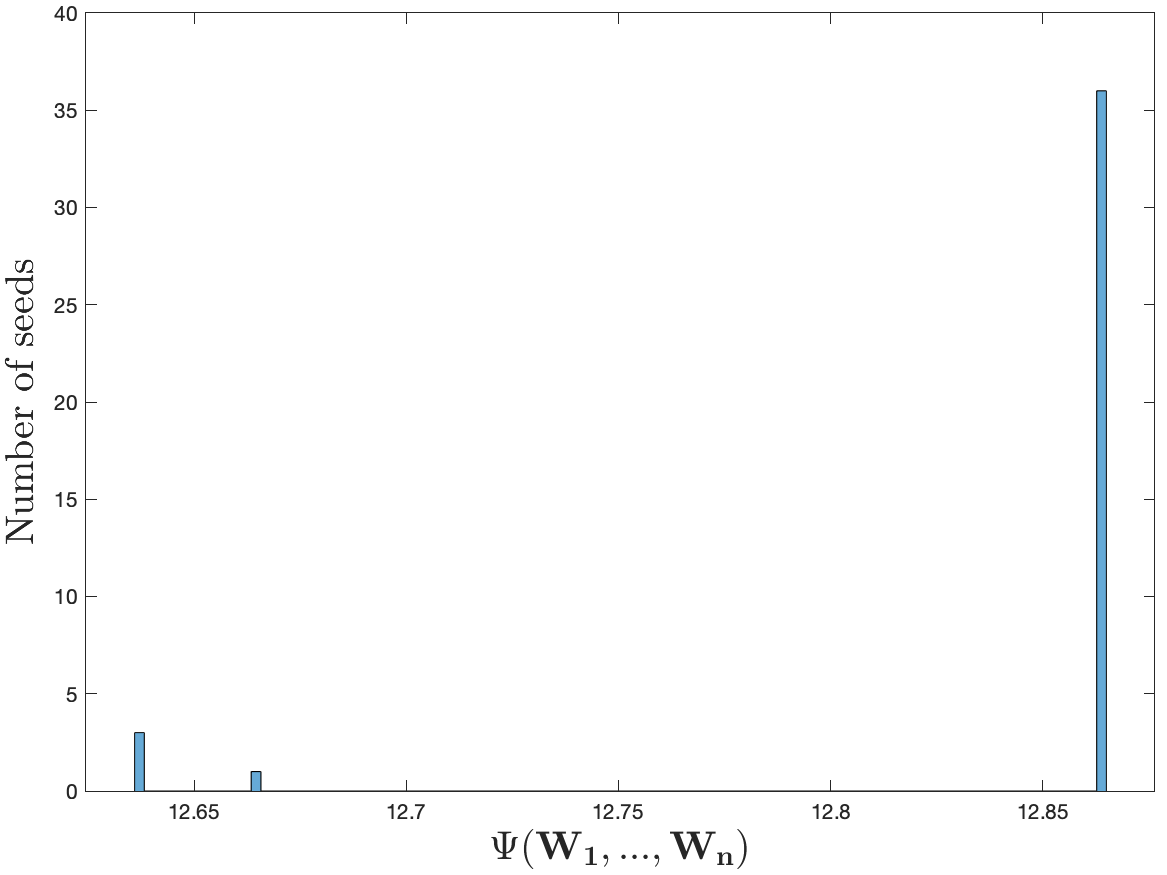} \\
    \small (e) PALM
  \end{tabular}} 
  \scalebox{1}{\begin{tabular}[b]{c}
    \includegraphics[width=.3\textwidth]{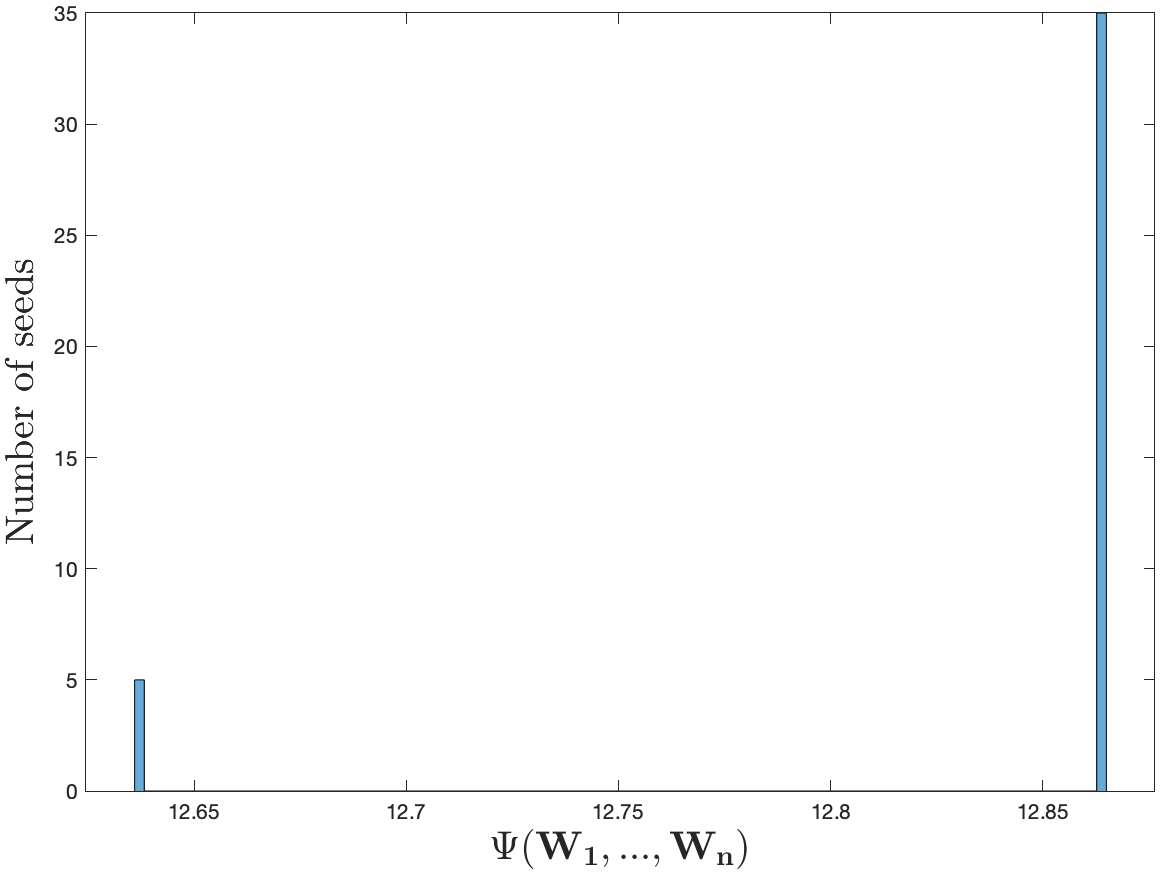} \\
    \small (f) iPALM ($\beta=0.2$)
  \end{tabular}} 
  \scalebox{1}{\begin{tabular}[b]{c}
    \includegraphics[width=.3\textwidth]{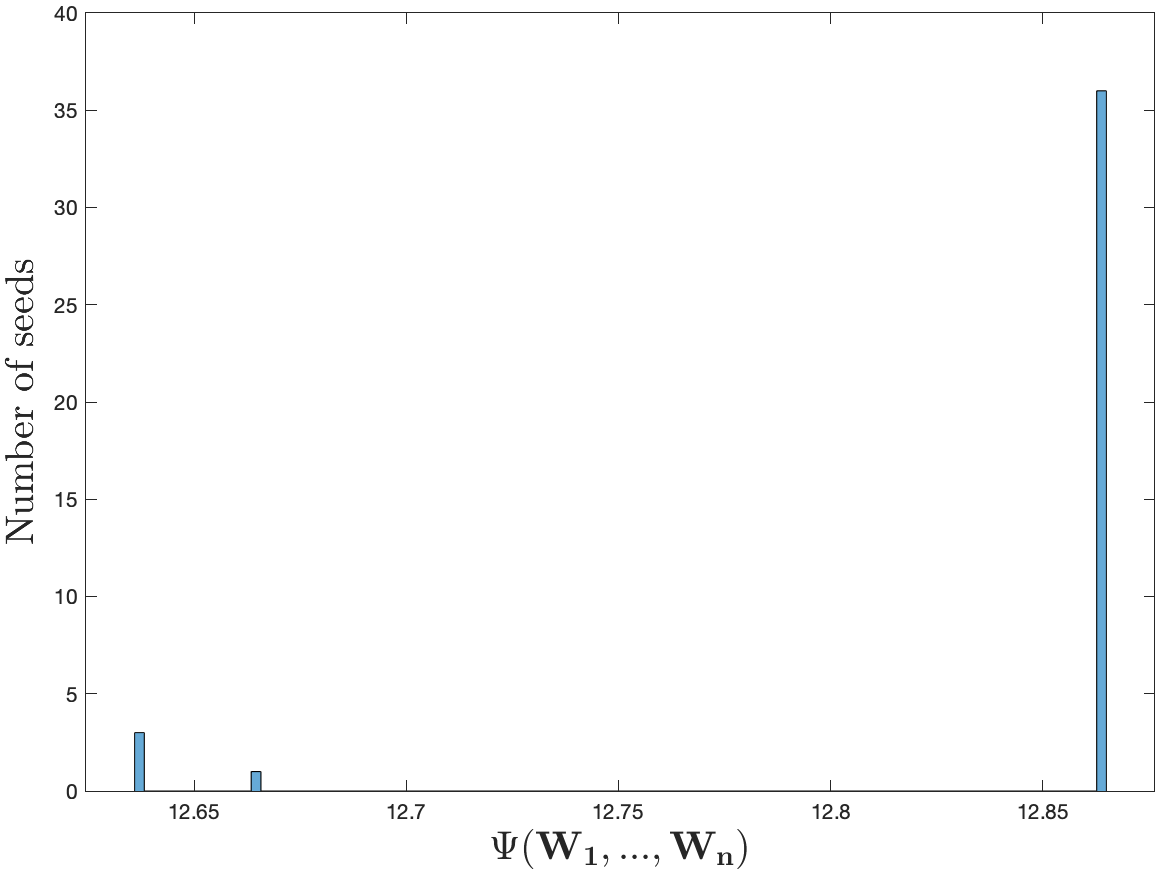} \\
    \small (g) iPALM ($\beta=0.4$)
  \end{tabular}} 
  \scalebox{1}{\begin{tabular}[b]{c}
    \includegraphics[width=.3\textwidth]{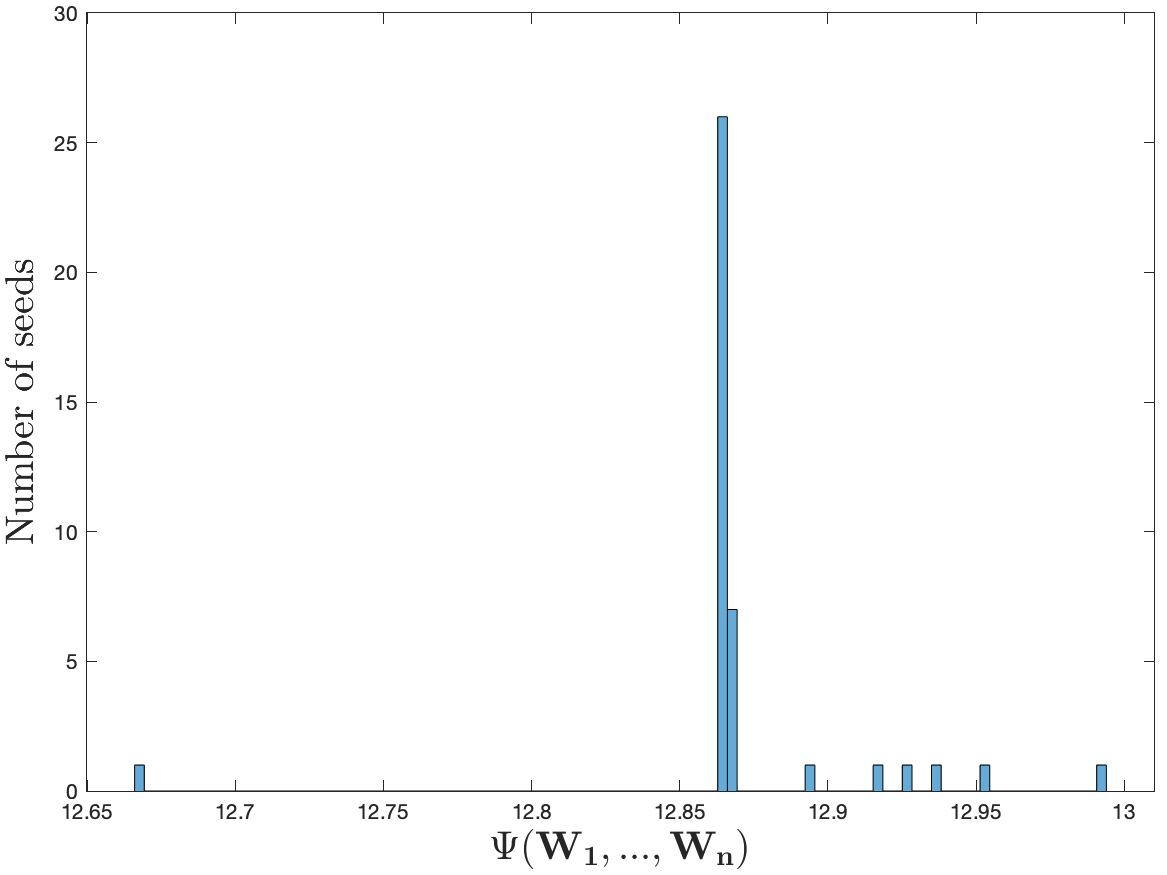} \\
    \small (h) FBS-WB
  \end{tabular}} 
  \scalebox{1}{\begin{tabular}[b]{c}
    \includegraphics[width=.3\textwidth]{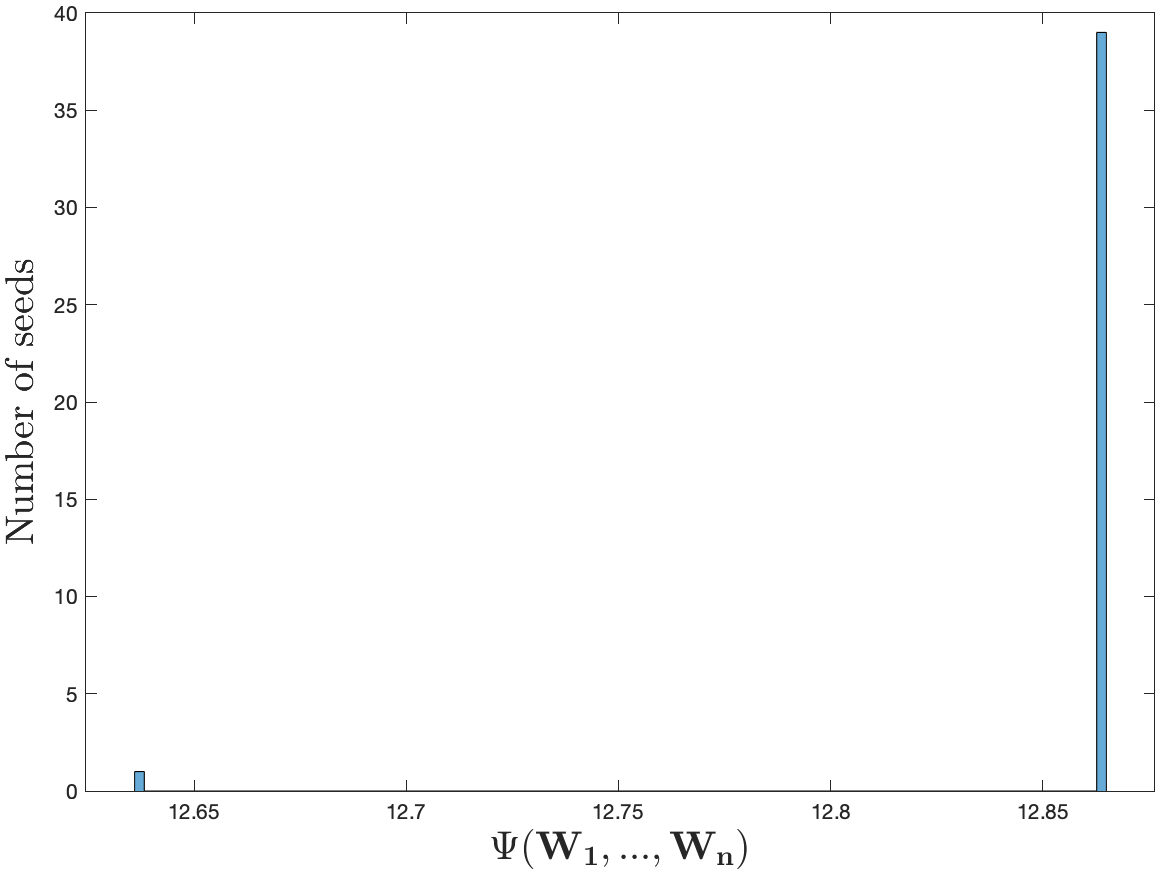} \\
    \small (i) iPiano-WB
  \end{tabular}} 
  \caption{Statistical evaluation - L1-regularization, $N=3$}
  \label{fig:exp1_stat_eval_l1reg}
\end{figure*}

\begin{figure*}[phbt!]
  \centering
  \scalebox{1}{\begin{tabular}[b]{c}
    \includegraphics[width=.3\textwidth]{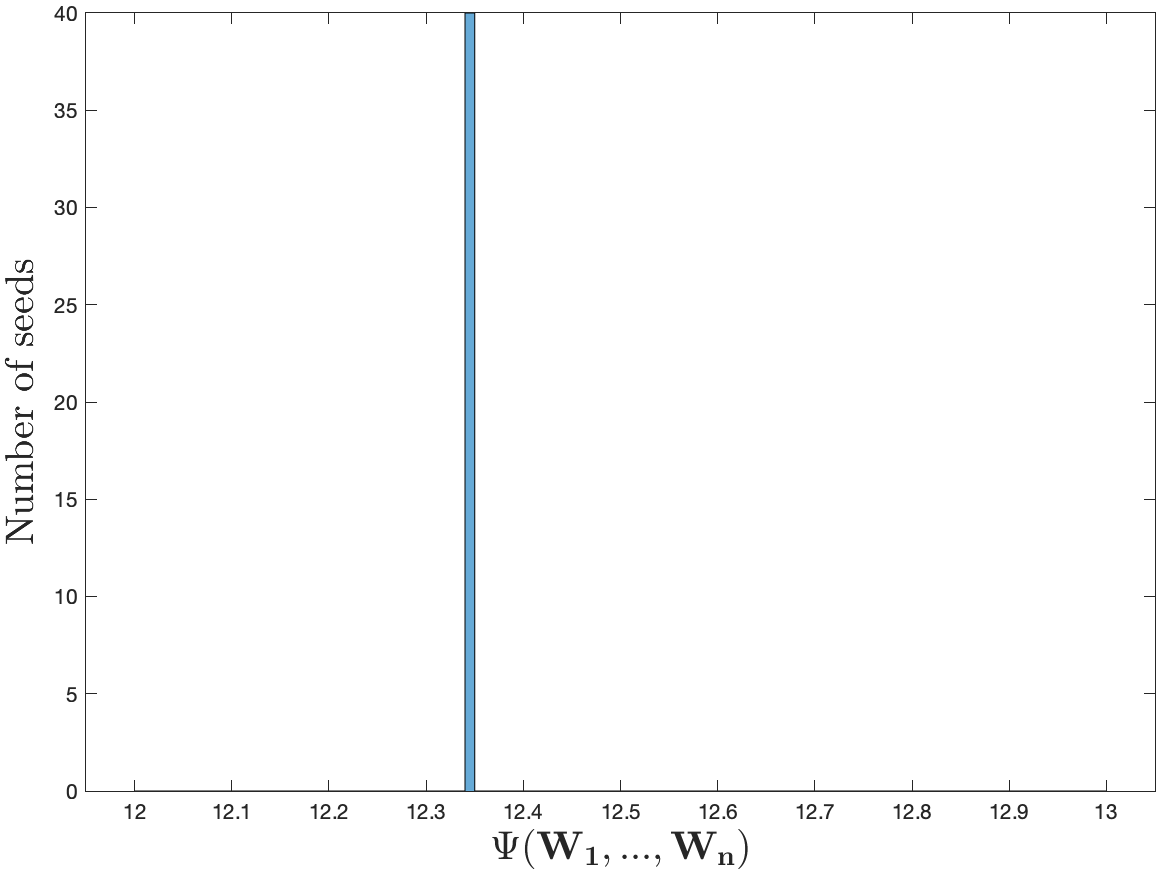} \\
    \small (a) BPG
  \end{tabular}} 
  \scalebox{1}{\begin{tabular}[b]{c}
    \includegraphics[width=.3\textwidth]{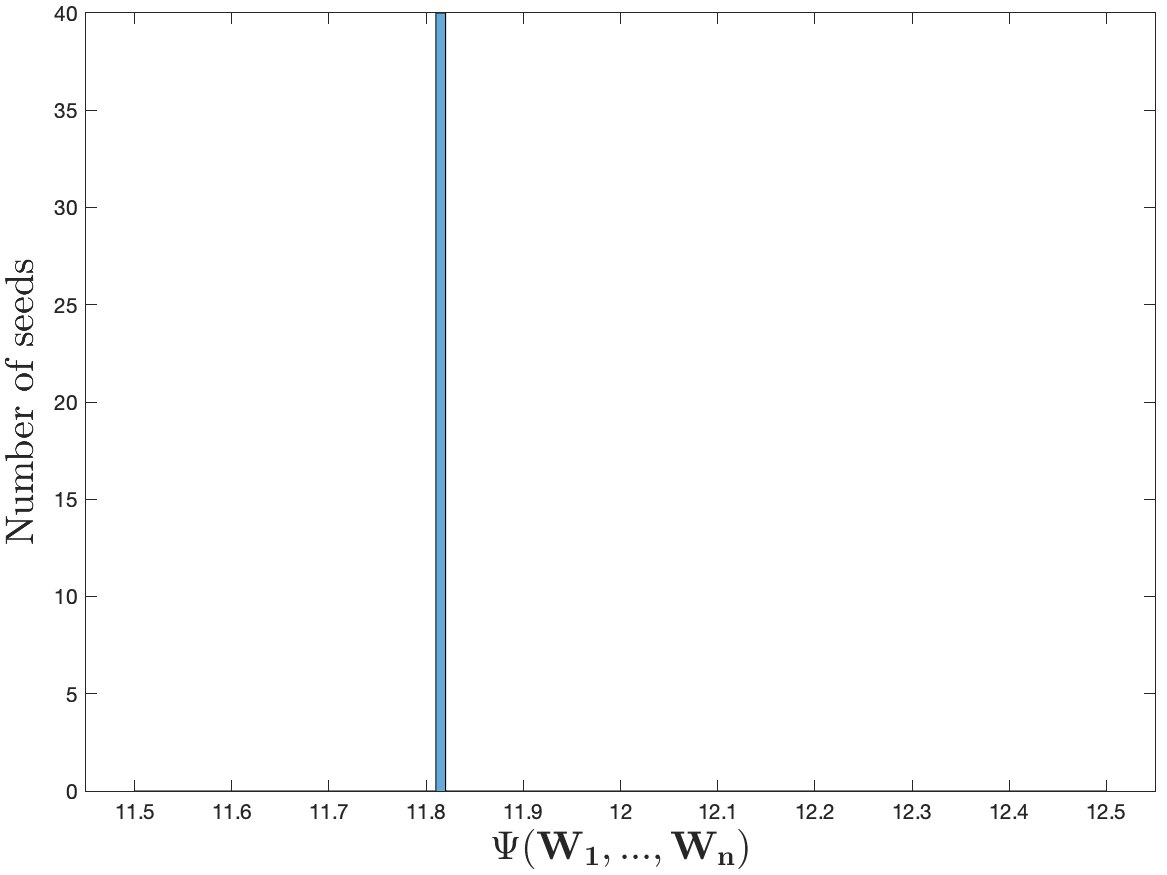} \\
    \small (b) BPG-WB
  \end{tabular}} 
  \scalebox{1}{\begin{tabular}[b]{c}
    \includegraphics[width=.3\textwidth]{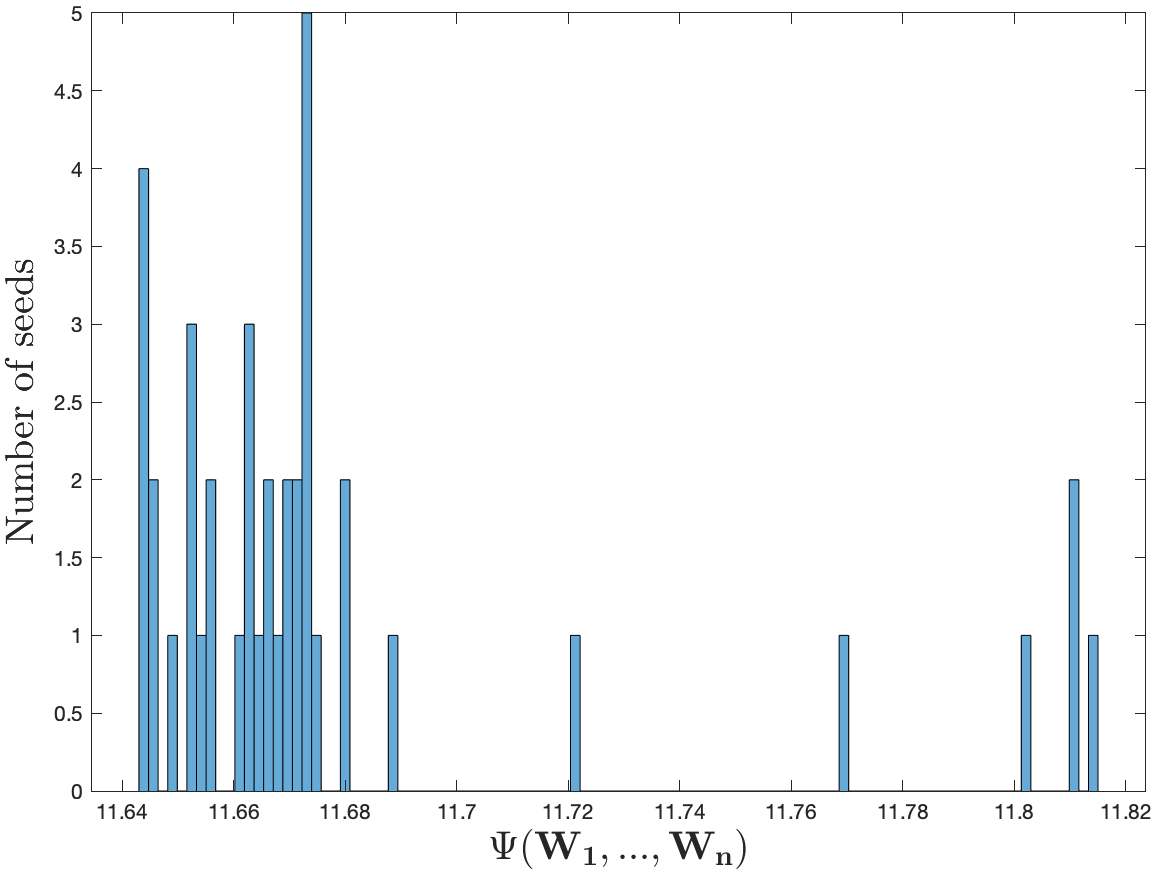} \\
    \small (c) CoCaIn BPG
  \end{tabular}} 
  \scalebox{1}{\begin{tabular}[b]{c}
    \includegraphics[width=.3\textwidth]{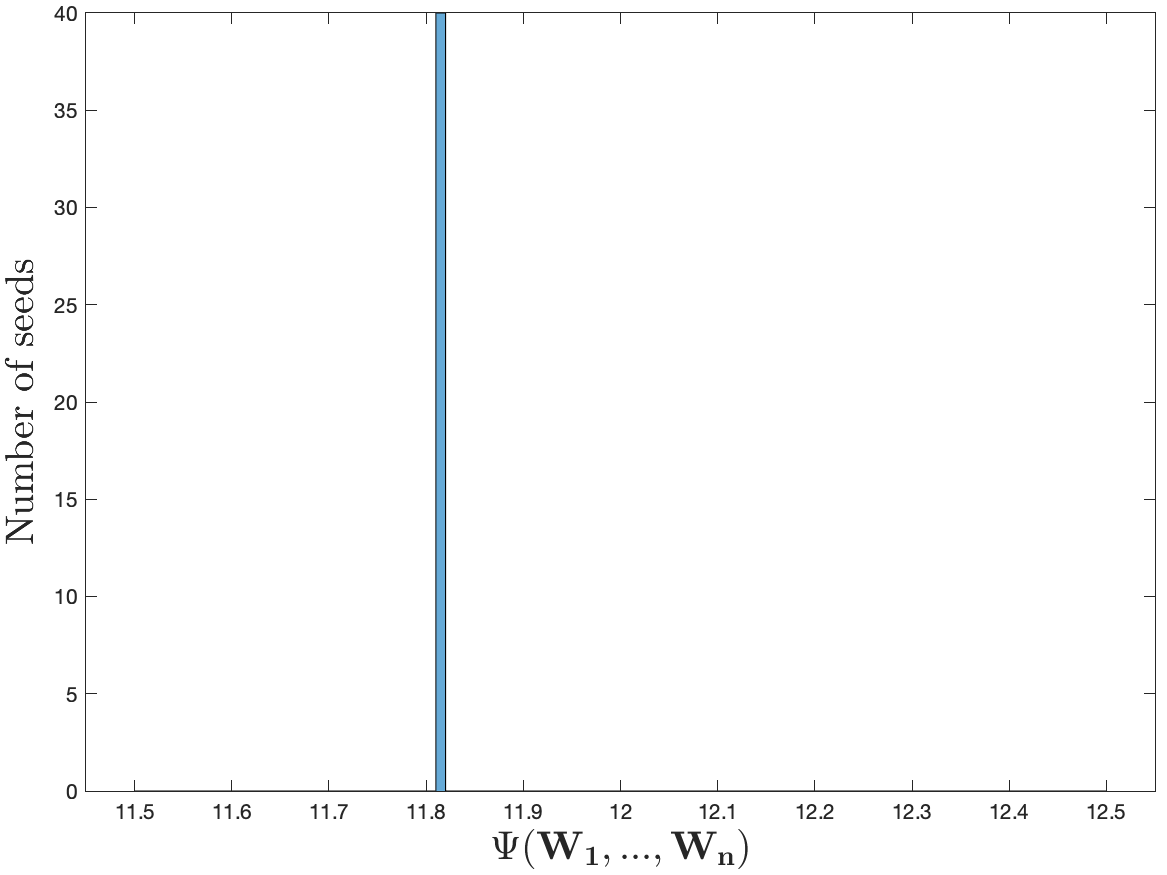} \\
    \small (d) CoCaIn BPG CFI
  \end{tabular} } 
  \scalebox{1}{\begin{tabular}[b]{c}
    \includegraphics[width=.3\textwidth]{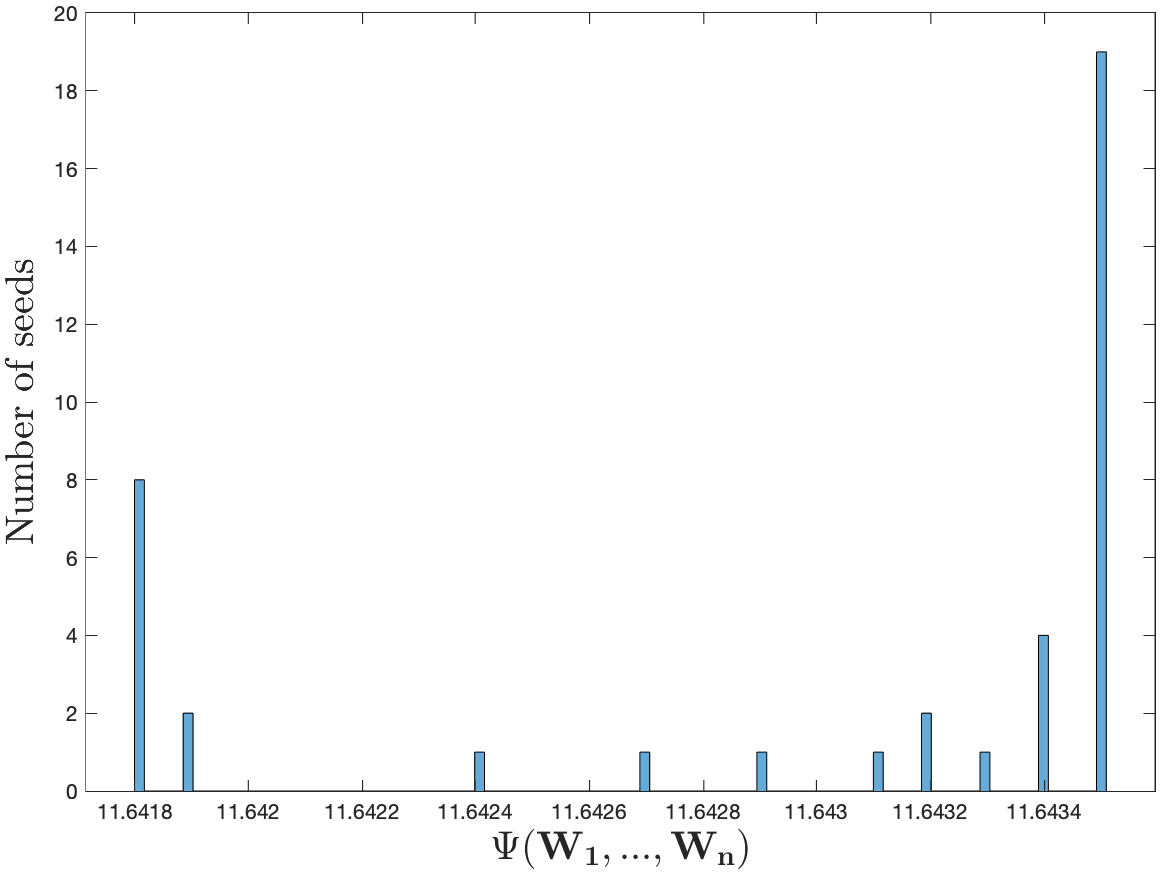} \\
    \small (e) PALM
  \end{tabular}} 
  \scalebox{1}{\begin{tabular}[b]{c}
    \includegraphics[width=.3\textwidth]{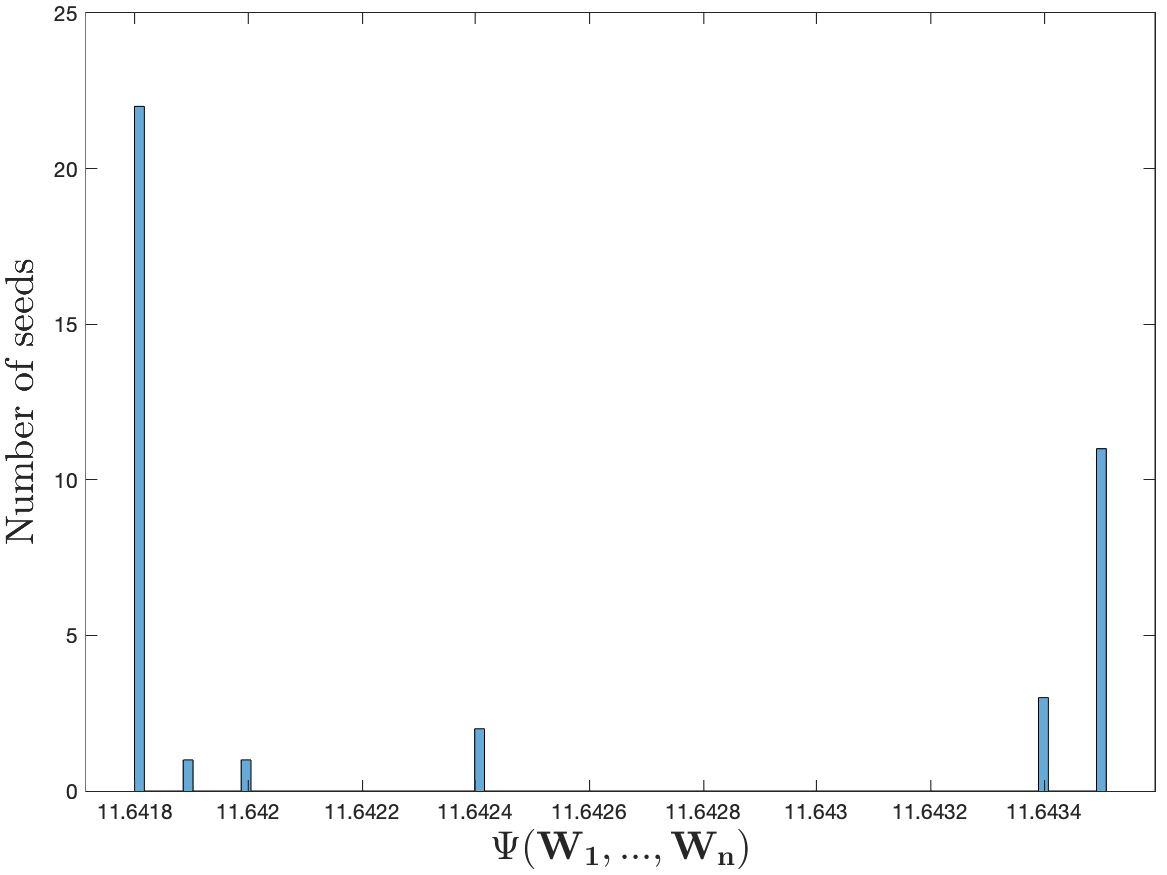} \\
    \small (f) iPALM ($\beta=0.2$)
  \end{tabular}} 
  \scalebox{1}{\begin{tabular}[b]{c}
    \includegraphics[width=.3\textwidth]{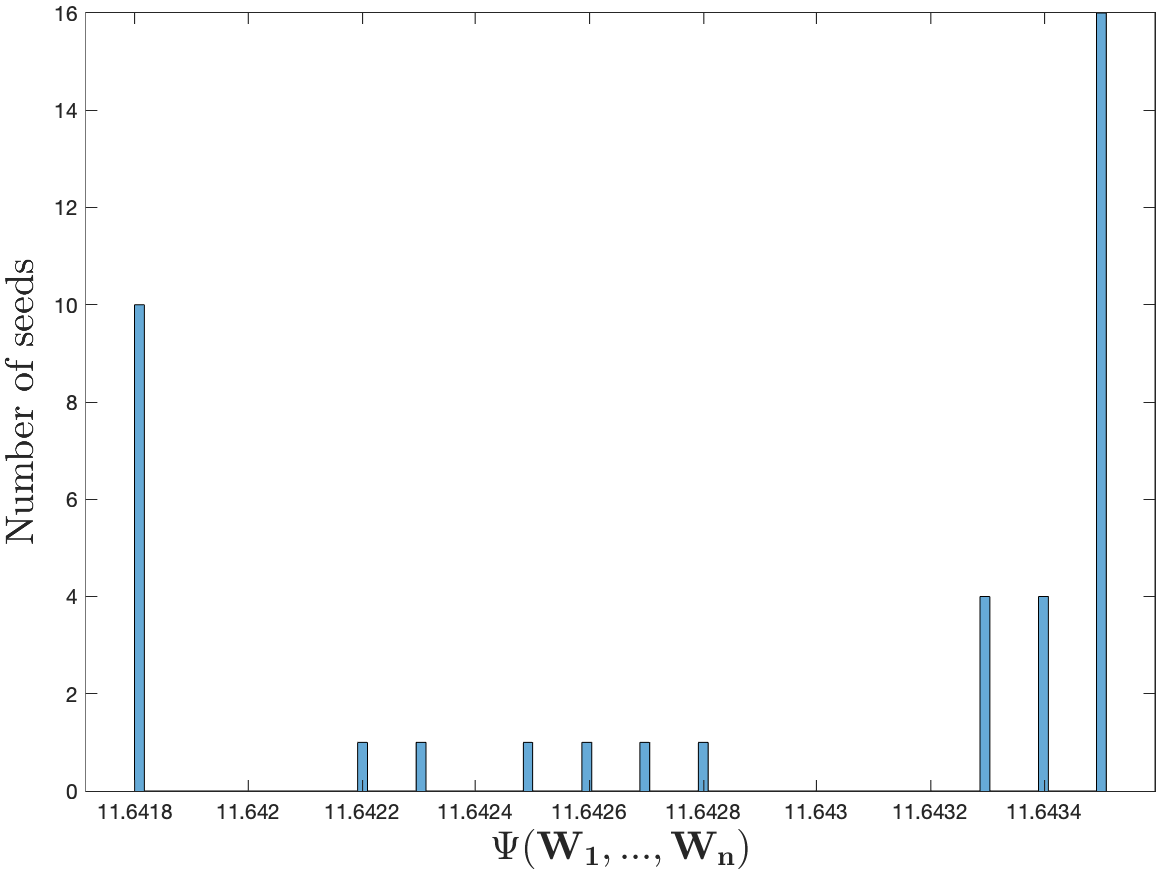} \\
    \small (g) iPALM ($\beta=0.4$)
  \end{tabular} } 
  \scalebox{1}{ \begin{tabular}[b]{c}
    \includegraphics[width=.3\textwidth]{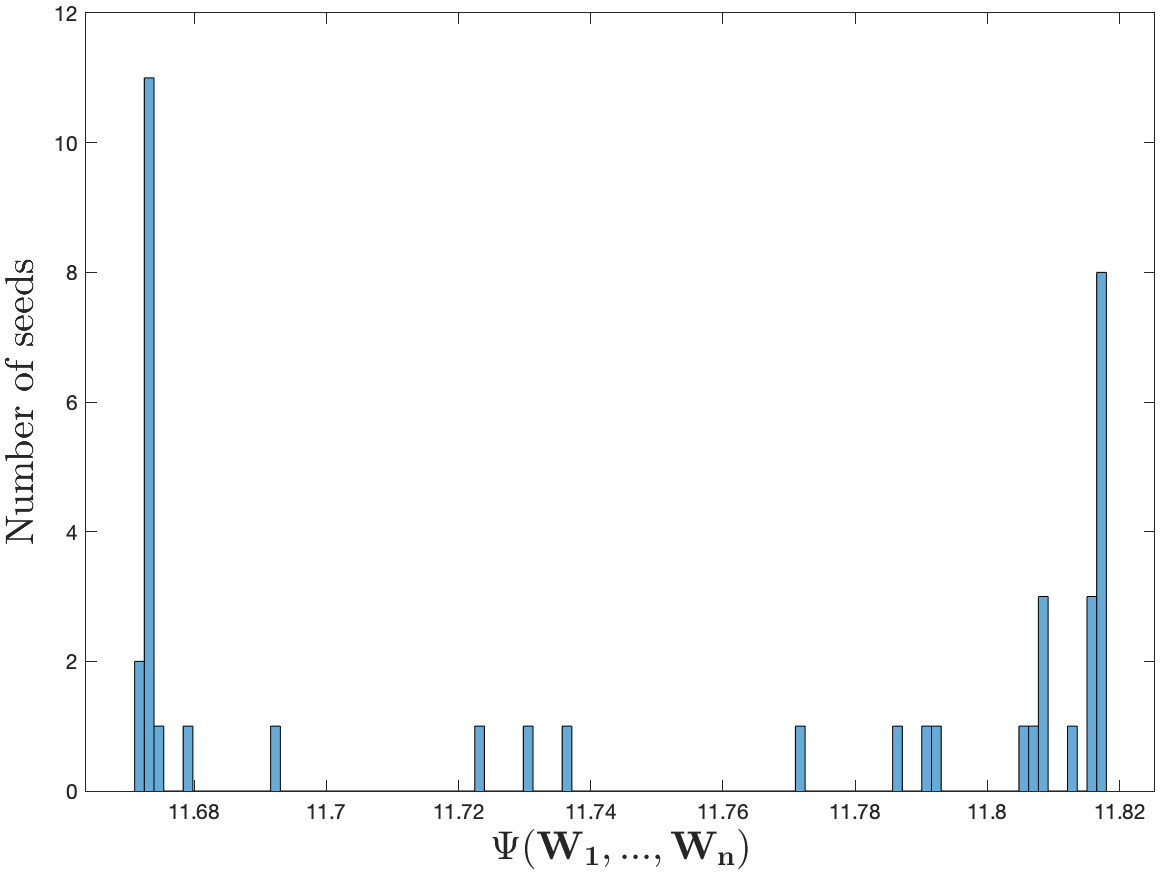} \\
    \small (h) FBS-WB
  \end{tabular}} 
  \scalebox{1}{\begin{tabular}[b]{c}
    \includegraphics[width=.3\textwidth]{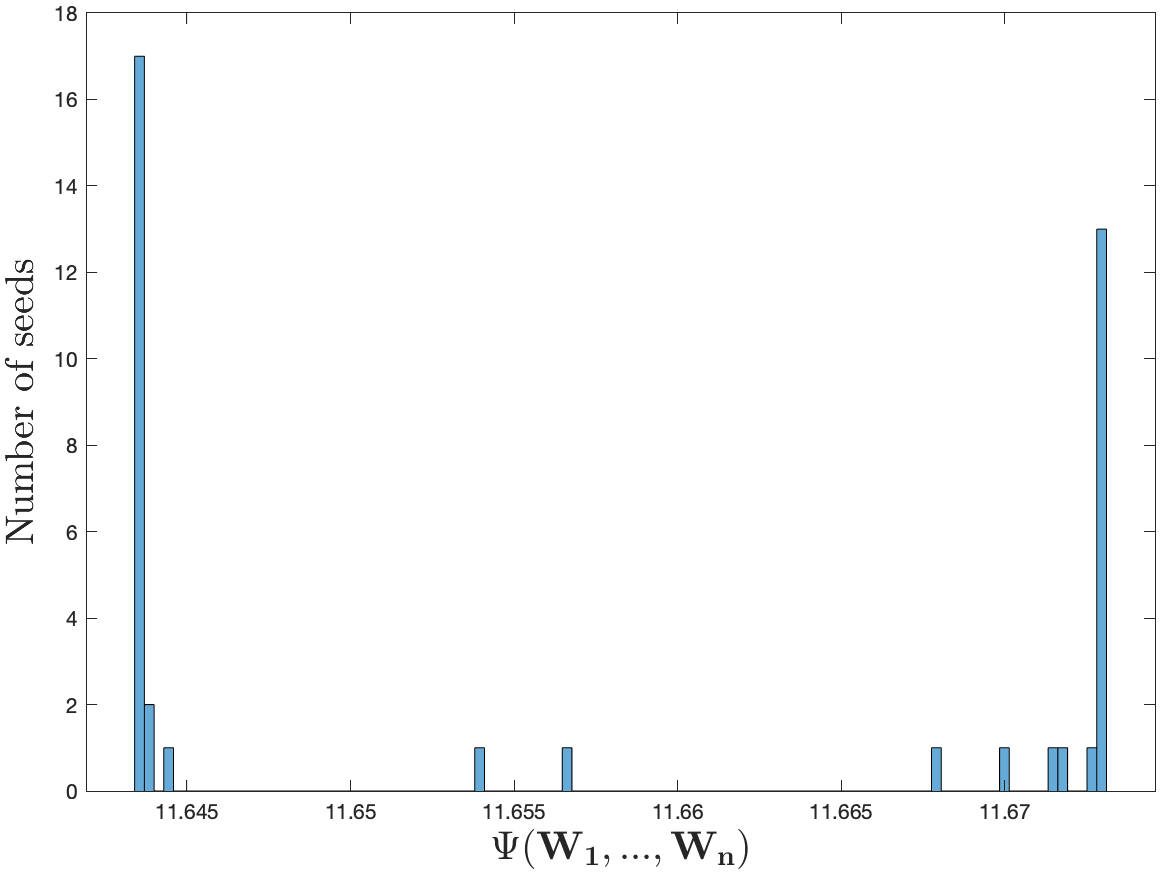} \\
    \small (i) iPiano-WB
  \end{tabular}} 
  \caption{Statistical evaluation - No regularization, $N=3$}
  \label{fig:exp1_stat_eval_noreg}
\end{figure*}

\subsection{Experiment 2}
In the second experiment we use the same hyperparameters, weight initialization and input ${\bf X} \in \mathbb{R}^{7 \times 50}$  as in Experiment 1. While we used independently generated input and output data in Experiment 1, the output data is now generated with ${\bf Y}={\bf AX}+0.0001{\bf N}$, where ${\bf A}$ is a randomly generated matrix in $[0,0.1]^{2 \times 7}$ and ${\bf N} \sim \mathcal{N}(0,1)$. Additionally, the weights are not squared matrices, i.e ${\bf W_1} \in \mathbb{R}^{2 \times 3}$. The results are provided in \ref{fig:exp2_rel_obj}.
While BPG-WB and CoCaIn BPG CFI achieve the best performance in a setting with L2-regualrizer or no regularizer, both algorithms can not compete with the alternating algorithms PALM and iPALM as well as iPiano-WB in case of L1-regularizers. Here, CoCaIn BPG is strong with a convergence better than iPiano-WB.  \ifpaper\else\medskip\fi

Finally, note that the proposed Bregman distances involve the norms of the weights, which can be very large for large $N$ and might result in numerically instability. An important open research problem, is to develop numerically stable Bregman distances.

\begin{figure*}[phbt!]
  \centering
  \begin{tabular}[b]{c}
    \includegraphics[width=.3\textwidth]{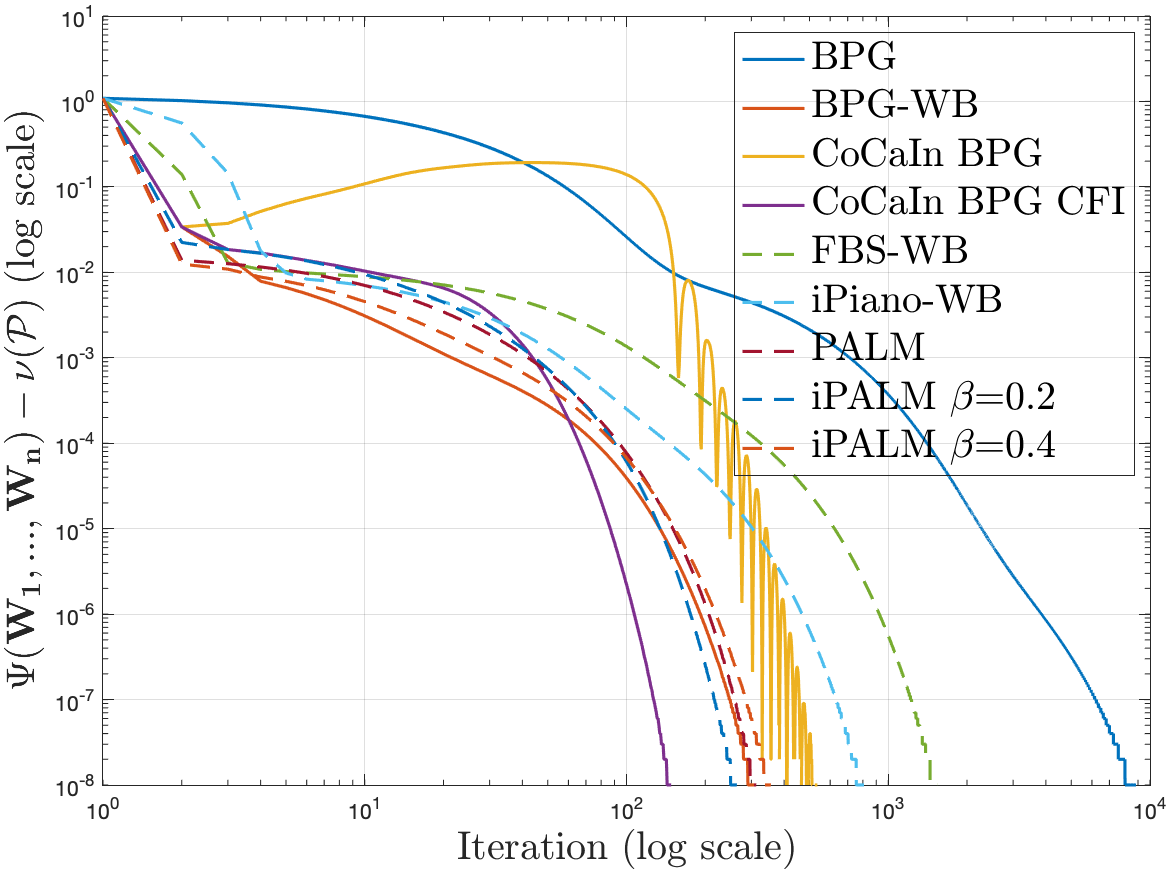} \\
    \small (a) L2-Regularization ($N=3$)
  \end{tabular}
  \begin{tabular}[b]{c}
    \includegraphics[width=.3\textwidth]{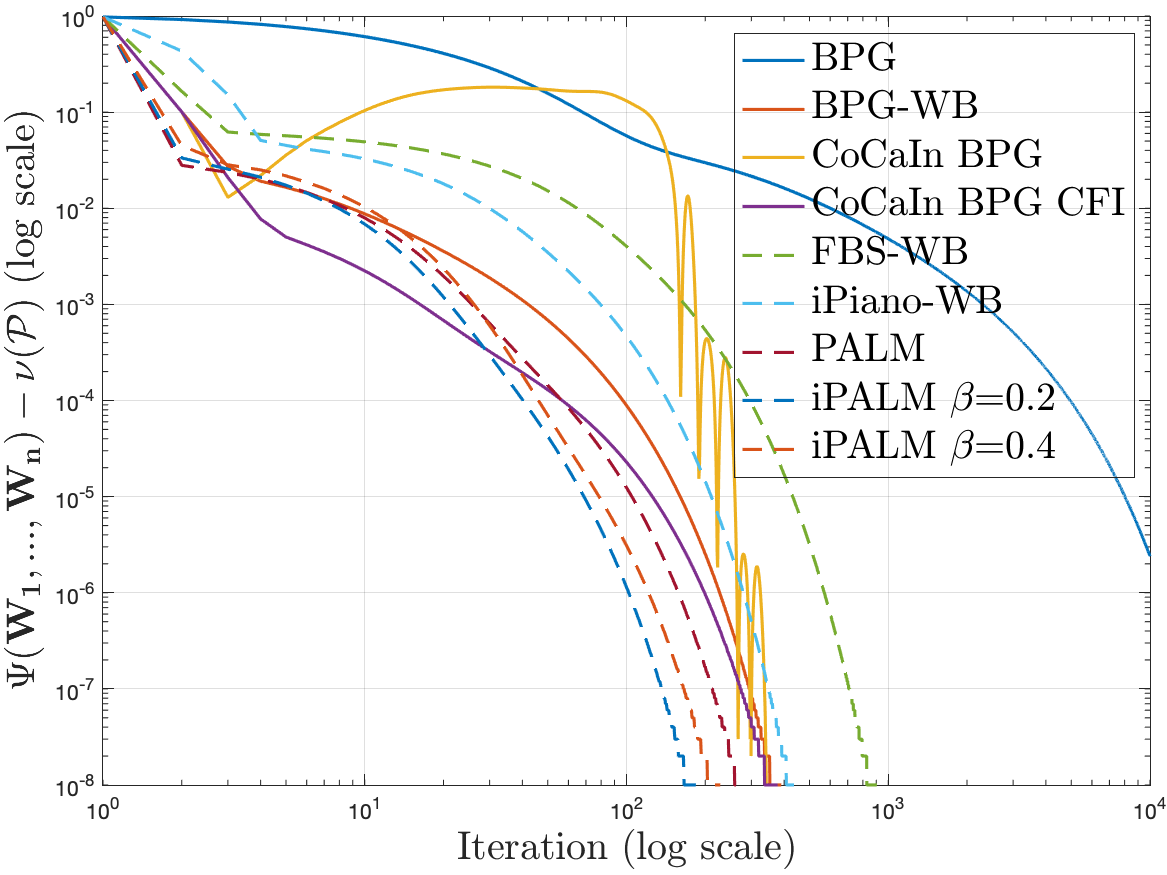} \\
    \small (b) L1-Regularization ($N=3$)
  \end{tabular}
  \begin{tabular}[b]{c}
    \includegraphics[width=.3\textwidth]{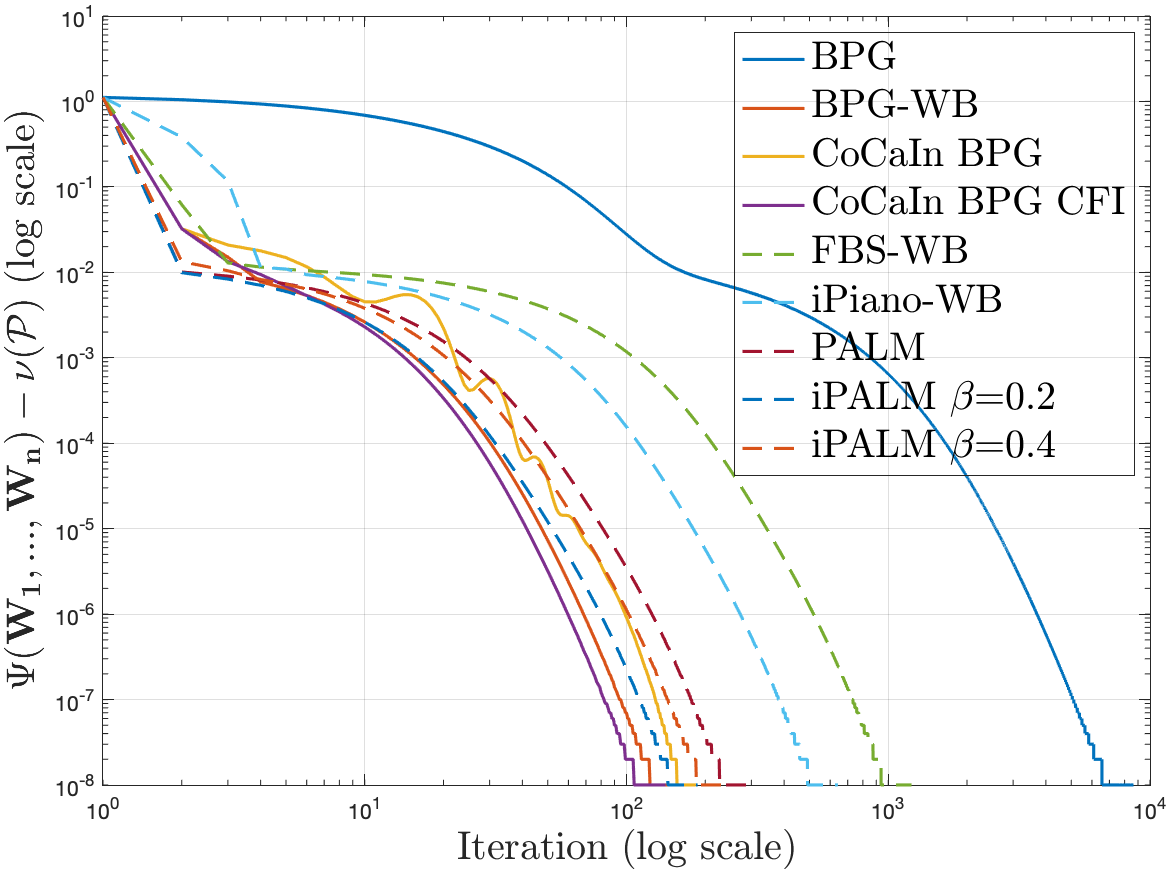} \\
    \small (c) No Regularization ($N=3$)
  \end{tabular} 
  \begin{tabular}[b]{c}
    \includegraphics[width=.3\textwidth]{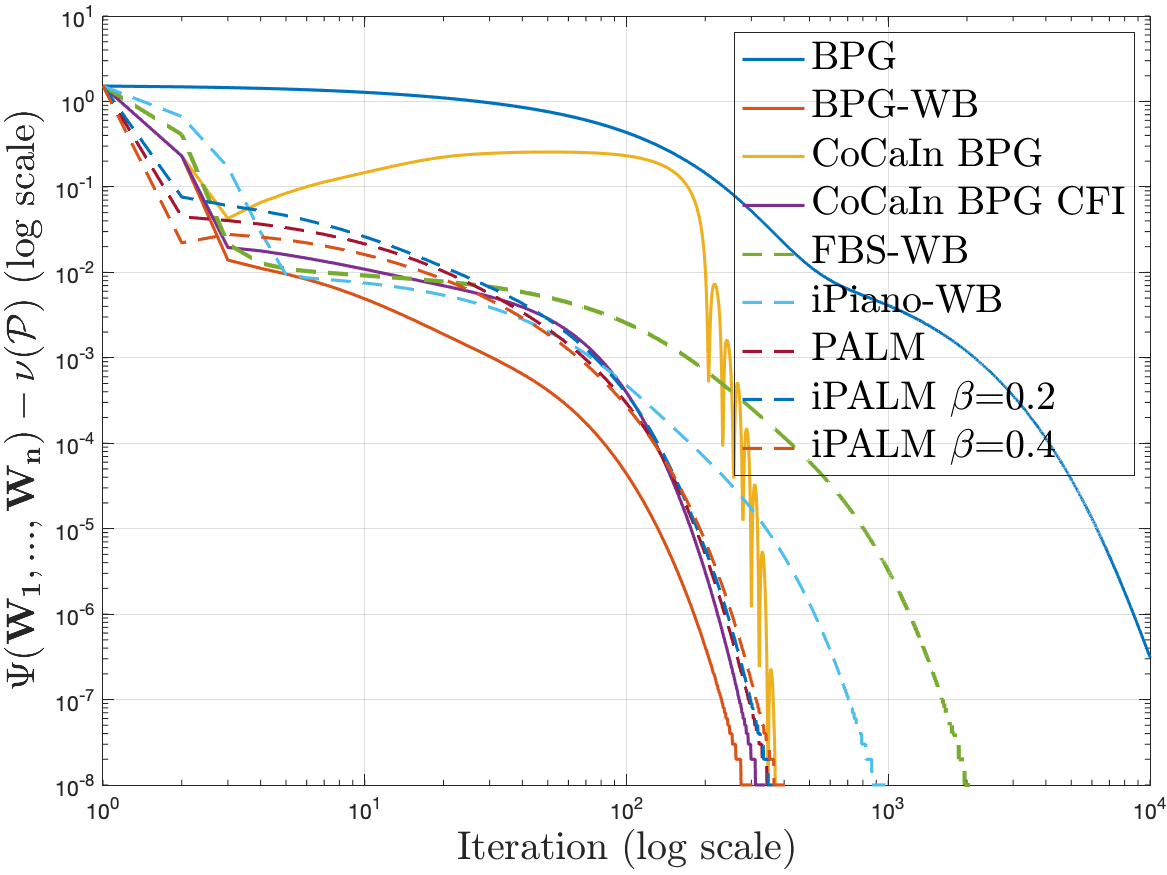} \\
    \small (d) L2-Regularization ($N=4$)
  \end{tabular}
  \begin{tabular}[b]{c}
    \includegraphics[width=.3\textwidth]{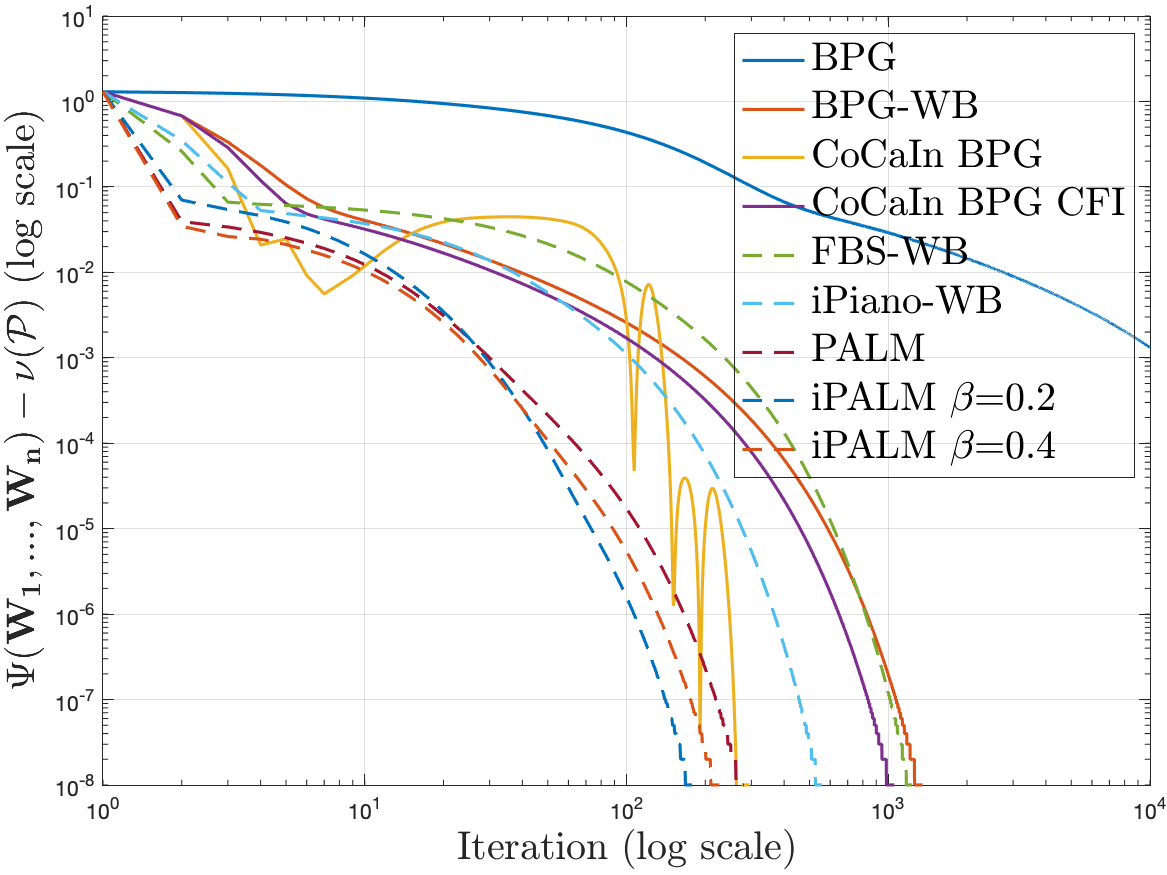} \\
    \small (e) L1-Regularization ($N=4$)
  \end{tabular}
  \begin{tabular}[b]{c}
  \includegraphics[width=.3\textwidth]{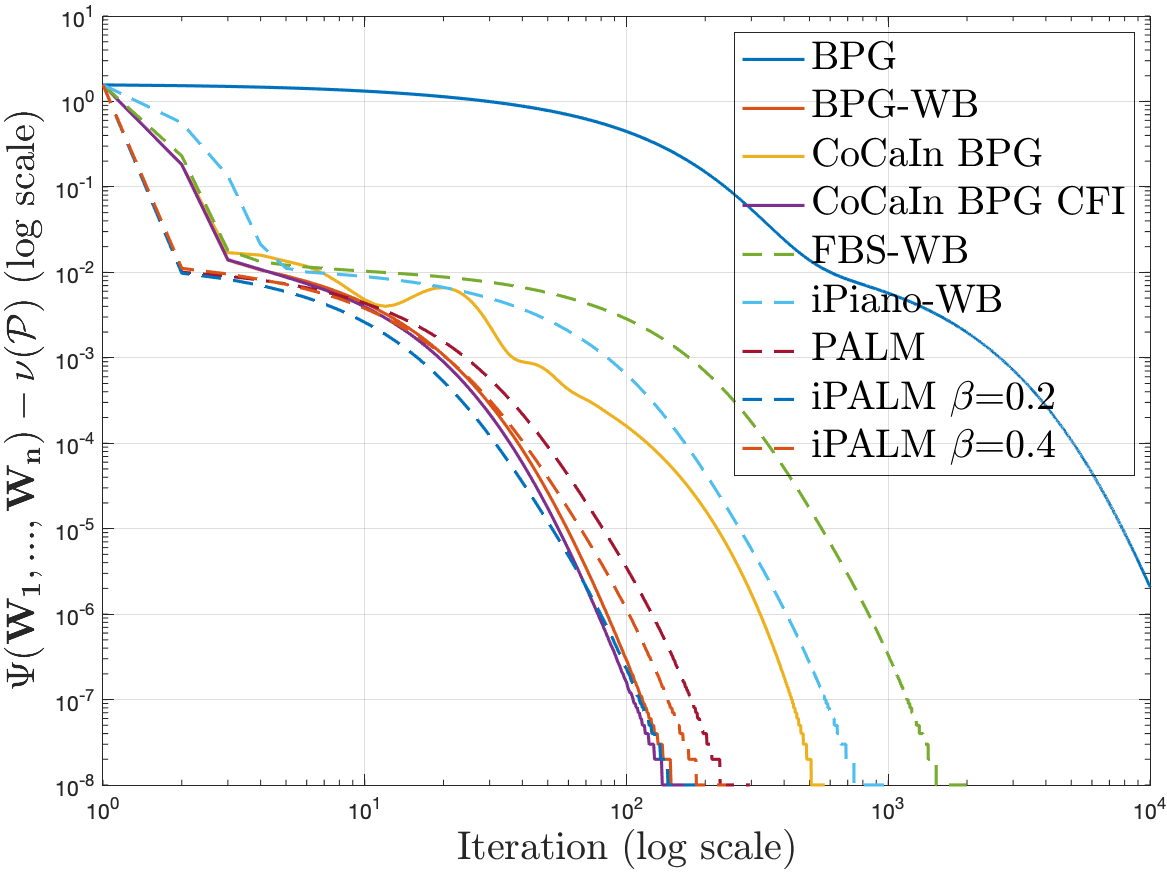} \\
    \small (f) No Regularization ($N=4$)
  \end{tabular} 
  \begin{tabular}[b]{c}
    \includegraphics[width=.3\textwidth]{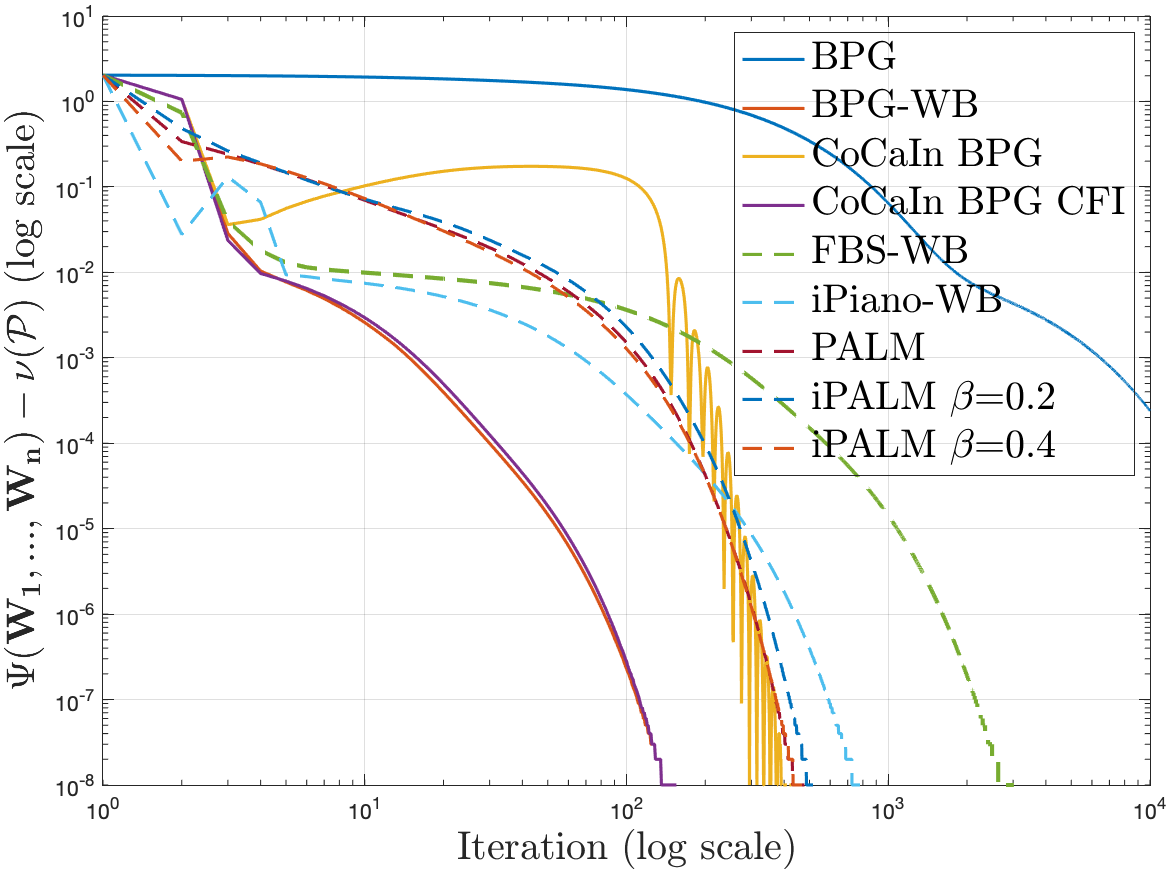} \\
    \small (g) L2-Regularization ($N=5$)
  \end{tabular}
  \begin{tabular}[b]{c}
    \includegraphics[width=.3\textwidth]{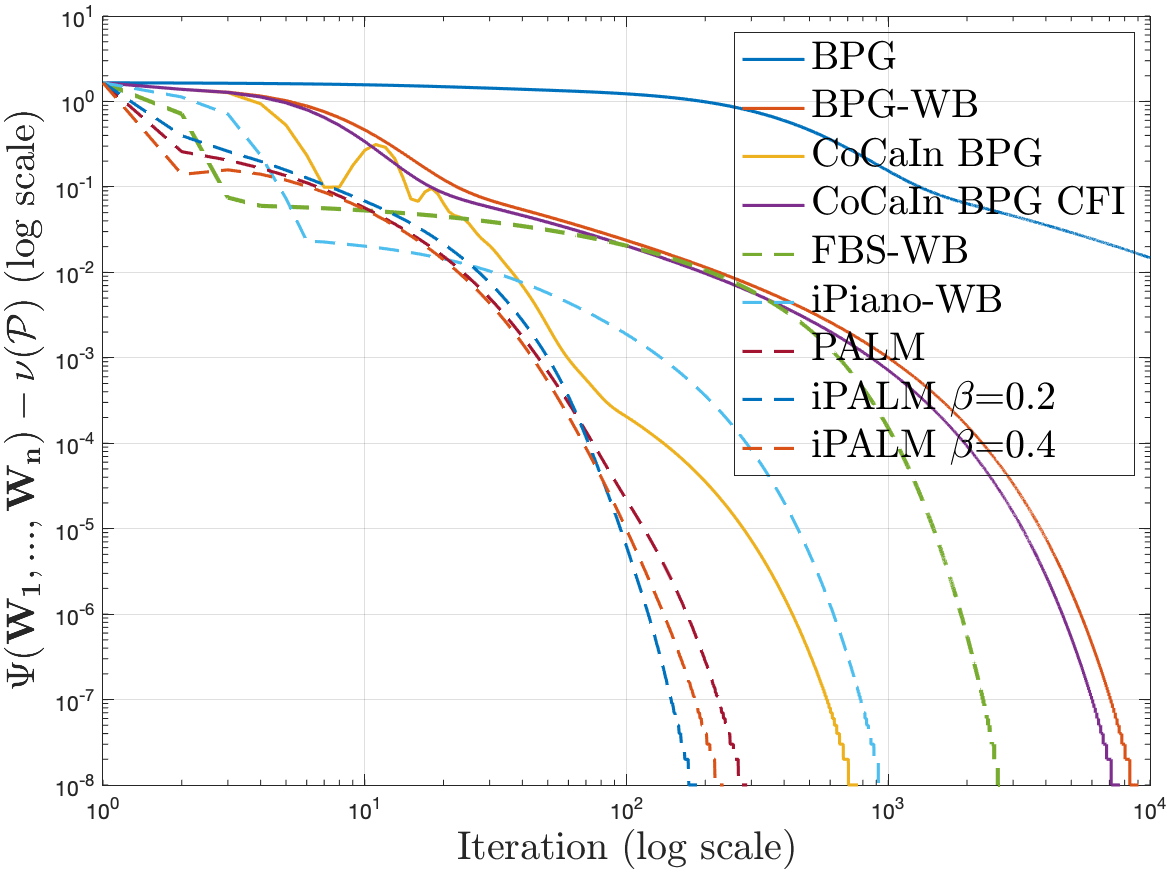} \\
    \small (h) L1-Regularization ($N=5$)
  \end{tabular}
  \begin{tabular}[b]{c}
  \includegraphics[width=.3\textwidth]{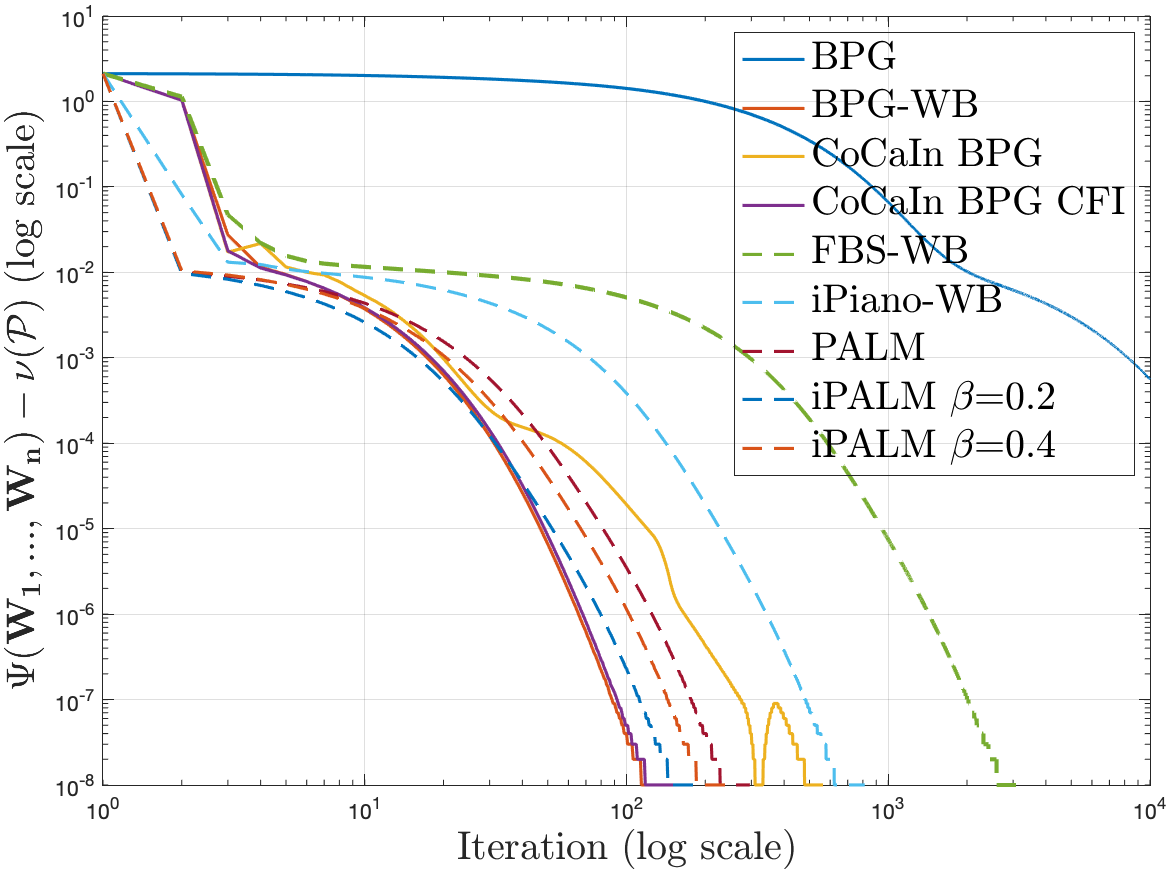} \\
    \small (i) No Regularization ($N=5$)
  \end{tabular} 
  \caption{Convergence plots for Experiment 2}
  \label{fig:exp2_rel_obj}
\end{figure*}

\ifpaper
\else
\newpage
\bibliographystyle{plain}
\bibliography{notes}

\def\cprime{$'$}
\begin{thebibliography}{10}

\bibitem{AHGP2019}
M.~Ahookhosh, L.~T.~K. Hien, N.~Gillis, and P.~Patrinos.
\newblock Multi-block {B}regman proximal alternating linearized minimization
  and its application to sparse orthogonal nonnegative matrix factorization.
\newblock {\em arXiv preprint arXiv:1908.01402}, 2019.

\bibitem{ACHL2019}
S.~Arora, N.~Cohen, W.~Hu, and Y.~Luo.
\newblock Implicit regularization in deep matrix factorization.
\newblock {\em ArXiv preprint arXiv:1905.13655}, 2019.

\bibitem{AB2009}
H.~Attouch and J.~Bolte.
\newblock On the convergence of the proximal algorithm for nonsmooth functions
  involving analytic features.
\newblock {\em Mathematical Programming}, 116(1-2):5--16, 2009.

\bibitem{ABRS2010}
H.~Attouch, J.~Bolte, P.~Redont, and A.~Soubeyran.
\newblock Proximal alternating minimization and projection methods for
  nonconvex problems: an approach based on the {K}urdyka-{{\L}}ojasiewicz
  inequality.
\newblock {\em Mathematics of Operations Research}, 35(2):438--457, 2010.

\bibitem{BBT2016}
H.~H. Bauschke, J.~Bolte, and M.~Teboulle.
\newblock A descent lemma beyond {L}ipschitz gradient continuity: first-order
  methods revisited and applications.
\newblock {\em Mathematics of Operations Research}, 42(2):330--348, 2017.

\bibitem{BT2003}
A.~Beck and M.~Teboulle.
\newblock Mirror descent and nonlinear projected subgradient methods for convex
  optimization.
\newblock {\em Operations Research Letters}, 31(3):167--175, 2003.

\bibitem{bellkligler2019blind}
Sefi Bell-Kligler, Assaf Shocher, and Michal Irani.
\newblock Blind super-resolution kernel estimation using an internal-gan, 2019.

\bibitem{BDLS2007}
J.~Bolte, A.~Daniilidis, A.S. Lewis, and M.~Shiota.
\newblock Clarke subgradients of stratifiable functions.
\newblock {\em SIAM Journal on Optimization}, 18(2):556--572, 2007.

\bibitem{BST2014}
J.~Bolte, S.~Sabach, and M.~Teboulle.
\newblock Proximal alternating linearized minimization for nonconvex and
  nonsmooth problems.
\newblock {\em Mathematical Programming}, 146(1-2):459--494, 2014.

\bibitem{BSTV2018}
J.~Bolte, S.~Sabach, M.~Teboulle, and Y.~Vaisbourd.
\newblock First order methods beyond convexity and {L}ipschitz gradient
  continuity with applications to quadratic inverse problems.
\newblock {\em SIAM Journal on Optimization}, 28(3):2131--2151, 2018.

\bibitem{choromanska2015loss}
Anna Choromanska, Mikael Henaff, Michael Mathieu, G{\'e}rard~Ben Arous, and
  Yann LeCun.
\newblock The loss surfaces of multilayer networks.
\newblock In {\em Artificial Intelligence and Statistics}, pages 192--204,
  2015.

\bibitem{DDM2018}
D.~Davis, D.~Drusvyatskiy, and K.~J. MacPhee.
\newblock Stochastic model-based minimization under high-order growth.
\newblock {\em ArXiv preprint arXiv:1807.00255}, 2018.

\bibitem{DBA2019}
R.~A. Dragomir, A.~d'Aspremont, and J.~Bolte.
\newblock Quartic first-order methods for low rank minimization.
\newblock {\em ArXiv preprint arXiv:1901.10791}, 2019.

\bibitem{DHS2011}
J.~Duchi, E.~Hazan, and Y.~Singer.
\newblock Adaptive subgradient methods for online learning and stochastic
  optimization.
\newblock {\em Journal of Machine Learning Research}, 12(Jul):2121--2159, 2011.

\bibitem{GBS2019}
G.~Gidel, F.~Bach, and S.~Lacoste-Julien.
\newblock Implicit regularization of discrete gradient dynamics in deep linear
  neural networks.
\newblock {\em arXiv preprint arXiv:1904.13262}, 2019.

\bibitem{goodfellow2016deep}
Ian Goodfellow, Yoshua Bengio, and Aaron Courville.
\newblock {\em Deep learning}.
\newblock MIT press, 2016.

\bibitem{GWBNS2017}
S.~Gunasekar, B.~E. Woodworth, S.~Bhojanapalli, B.~Neyshabur, and N.~Srebro.
\newblock Implicit regularization in matrix factorization.
\newblock In {\em Advances in Neural Information Processing Systems}, pages
  6151--6159, 2017.

\bibitem{hanzely2018fastest}
Filip Hanzely and Peter Richt{\'a}rik.
\newblock Fastest rates for stochastic mirror descent methods.
\newblock {\em ArXiv preprint arXiv:1803.07374}, 2018.

\bibitem{HGP2019}
L.~T.~K. Hien, N.~Gillis, and P.~Patrinos.
\newblock Inertial block mirror descent method for non-convex non-smooth
  optimization.
\newblock {\em ArXiv preprint arXiv:1903.01818}, 2019.

\bibitem{K2016}
K.~Kawaguchi.
\newblock Deep learning without poor local minima.
\newblock In {\em Advances in neural information processing systems}, pages
  586--594, 2016.

\bibitem{KB2014}
D.~P. Kingma and J.~Ba.
\newblock Adam: A method for stochastic optimization.
\newblock {\em ArXiv preprint arXiv:1412.6980}, 2014.

\bibitem{LOC2019}
E.~Laude, P.~Ochs, and D.~Cremers.
\newblock Bregman proximal mappings and {B}regman-{M}oreau envelopes under
  relative prox-regularity.
\newblock {\em ArXiv preprint arXiv:1907.04306}, 2019.

\bibitem{LZTW2019}
Q.~Li, Z.~Zhu, G.~Tang, and M.~B. Wakin.
\newblock Provable {B}regman-divergence based methods for nonconvex and
  non-lipschitz problems.
\newblock {\em arXiv preprint arXiv:1904.09712}, 2019.

\bibitem{LFN2018}
H.~Lu, R.~M. Freund, and Y.~Nesterov.
\newblock Relatively smooth convex optimization by first-order methods, and
  applications.
\newblock {\em SIAM Journal on Optimization}, 28(1):333--354, 2018.

\bibitem{MH2017}
M.~C. Mukkamala and M.~Hein.
\newblock Variants of {RMSP}rop and {A}dagrad with logarithmic regret bounds.
\newblock In {\em Proceedings of the 34th International Conference on Machine
  Learning}, pages 2545--2553, 2017.

\bibitem{MO2019a}
M.~C. Mukkamala and P.~Ochs.
\newblock Beyond alternating updates for matrix factorization with inertial
  {B}regman proximal gradient algorithms.
\newblock {\em ArXiv preprint arXiv:1905.09050}, 2019.

\bibitem{MOPS2019}
M.~C. Mukkamala, P.~Ochs, T.~Pock, and S.~Sabach.
\newblock Convex-concave backtracking for inertial {B}regman proximal gradient
  algorithms in non-convex optimization.
\newblock {\em ArXiv preprint arXiv:1904.03537}, 2019.

\bibitem{N1998}
Y.~Nesterov.
\newblock Introductory lectures on convex optimization: a basic course, 2004.

\bibitem{OCBP2014}
P.~Ochs, Y.~Chen, T.~Brox, and T.~Pock.
\newblock i{P}iano: inertial proximal algorithm for nonconvex optimization.
\newblock {\em SIAM Journal on Imaging Sciences}, 7(2):1388--1419, 2014.

\bibitem{PS2016}
T.~Pock and S.~Sabach.
\newblock Inertial proximal alternating linearized minimization (i{PALM}) for
  nonconvex and nonsmooth problems.
\newblock {\em SIAM Journal on Imaging Sciences}, 9(4):1756--1787, 2016.

\bibitem{RW1998-B}
R.~T. Rockafellar and R.~J.-B. Wets.
\newblock {\em Variational Analysis}, volume 317 of {\em Fundamental Principles
  of Mathematical Sciences}.
\newblock Springer-Verlag, Berlin, 1998.

\bibitem{pmlr-v89-wu19b}
Yifan Wu, Barnabas Poczos, and Aarti Singh.
\newblock Towards understanding the generalization bias of two layer
  convolutional linear classifiers with gradient descent.
\newblock In Kamalika Chaudhuri and Masashi Sugiyama, editors, {\em Proceedings
  of Machine Learning Research}, volume~89 of {\em Proceedings of Machine
  Learning Research}, pages 1070--1078. PMLR, 16--18 Apr 2019.

\bibitem{XY2013}
Y.~Xu and W.~Yin.
\newblock A block coordinate descent method for regularized multiconvex
  optimization with applications to nonnegative tensor factorization and
  completion.
\newblock {\em SIAM Journal on imaging sciences}, 6(3):1758--1789, 2013.

\bibitem{YSJ2018}
C.~Yun, S.~Sra, and A.~Jadbabaie.
\newblock Global optimality conditions for deep neural networks.
\newblock In {\em International Conference on Learning Representations}, 2018.

\bibitem{ZBMJC2019}
X.~Zhang, R.~Barrio, M.~Martinez, H.~Jiang, and L.~Cheng.
\newblock {B}regman proximal gradient algorithm with extrapolation for a class
  of nonconvex nonsmooth minimization problems.
\newblock {\em ArXiv preprint arXiv:1904.11295}, 2019.

\end{thebibliography}
\fi
\end{document}